\documentclass[msom,sglanonrev]{informs4}
\usepackage{float}
\RequirePackage{tgtermes}
\RequirePackage{newtxtext}
\RequirePackage{newtxmath}
\RequirePackage{bm}
\RequirePackage{endnotes}

\OneAndAHalfSpacedXII %

\usepackage[utf8]{inputenc}
\usepackage[english]{babel}
\usepackage{graphicx,amsfonts,amssymb}
\makeatletter
\@ifpackageloaded{natbib}{}{\usepackage[authoryear,sort&compress]{natbib}}
\makeatother
\usepackage[colorlinks]{hyperref}
\usepackage{hypernat}
\usepackage{xcolor}
\definecolor{NavyBlue}{RGB}{35,35,142}
\definecolor{RawSienna}{RGB}{199,97,20}
\hypersetup{
    colorlinks,%
    citecolor=NavyBlue,%
    filecolor=NavyBlue,%
    linkcolor=RawSienna,%
    urlcolor=NavyBlue
}

\usepackage{xparse}



\usepackage{color}
\usepackage[colorinlistoftodos,textsize=scriptsize]{todonotes}
\newcommand{\comAP}[1]{\todo[inline,color=blue!20]{A: #1}}

\theoremstyle{plain}

  {\renewcommand{\qedsymbol}{$\triangle$}%
  \pushQED{\qed}\begin{remX}}
  {\popQED\end{remX}}
\newtheorem{exX}{Example}
  {\renewcommand{\qedsymbol}{$\triangle$}%
  \pushQED{\qed}\begin{exX}}
  {\popQED\end{exX}}

\DeclareMathOperator*{\esssup}{ess\,sup} %
\DeclareMathOperator*{\cov}{cov} %
\DeclareMathOperator*{\vect}{vect} %
\DeclareMathOperator*{\val}{val} %
\DeclareMathOperator*{\Sol}{Sol} %

\newcommand{\lra}[2]{\overset{#2}{\underset{#1}{\longrightarrow}}}
\newcommand{\ra}{\rightarrow}
\newcommand*{\arrow}[1][]{\mathbin{\tikz \draw [densely dashed,#1] (0pt,0.3ex) -- (1.5em,0.3ex);}}%

\newcommand{\re}{q}
\newcommand{\rleq}{\preceq}
\newcommand{\rgeq}{\succeq}
\newcommand{\nrleq}{\npreceq}
\newcommand{\nrgeq}{\nsucceq}
\newcommand{\rplus}{\oplus}
\newcommand{\bigrplus}{\bigoplus}
\newcommand{\rset}{M}
\newcommand{\rcost}{c}
\newcommand{\rmeas}{\rho}
\newcommand{\meet}{\wedge} %
\newcommand{\bigmeet}{\bigwedge}
\newcommand{\join}{\vee} %
\newcommand{\bigjoin}{\bigvee}

\NewDocumentCommand\prt{O{}m}{\mathrm{prt}\IfNoValueTF{#1}{}{_{#1}}(#2)}
\NewDocumentCommand\cpt{O{}m}{\mathrm{cpt}\IfNoValueTF{#1}{}{_{#1}}(#2)}
\NewDocumentCommand\asc{O{}m}{\mathrm{asc}\IfNoValueTF{#1}{}{_{#1}}(#2)}
\NewDocumentCommand\casc{O{}m}{\overline{\mathrm{asc}}\IfNoValueTF{#1}{}{_{#1}}(#2)}
\NewDocumentCommand\cascg{mm}{\casc[#1]{#2}}
\NewDocumentCommand\dsc{O{}m}{\mathrm{dsc}\IfNoValueTF{#1}{}{_{#1}}(#2)}
\NewDocumentCommand\dscg{mm}{\dsc{#1}{#2}}
\NewDocumentCommand\cdsc{O{}m}{\overline{\mathrm{dsc}}\IfNoValueTF{#1}{}{_{#1}}(#2)}
\NewDocumentCommand\cld{O{}m}{\mathrm{cld}\IfNoValueTF{#1}{}{_{#1}}(#2)}
\NewDocumentCommand\fa{O{}m}{\mathrm{fa}\IfNoValueTF{#1}{}{_{#1}}(#2)}
\NewDocumentCommand\mb{ooo}{\mathrm{mb}\IfNoValueTF{#1}{}{\IfNoValueTF{#2}{(#1)}{\IfNoValueTF{#3}{_{#1}(#2)}{_{#1|#2}(#3)}}}}

\usepackage{bbm}
\newcommand{\ind}{\mathbbm{1}}

\newcommand{\bfa}{\boldsymbol{a}}
\newcommand{\bfz}{\boldsymbol{z}}
\newcommand{\bfe}{\boldsymbol{e}}
\newcommand{\bfr}{\boldsymbol{r}}
\newcommand{\bft}{\boldsymbol{t}}
\newcommand{\bfy}{\boldsymbol{y}}
\newcommand{\bfu}{\boldsymbol{u}}
\newcommand{\bfi}{\boldsymbol{i}}
\newcommand{\bfo}{\boldsymbol{o}}
\newcommand{\bfp}{\boldsymbol{p}}
\newcommand{\bfq}{\boldsymbol{q}}
\newcommand{\bfs}{\boldsymbol{s}}
\newcommand{\bfd}{\boldsymbol{d}}
\newcommand{\bff}{\boldsymbol{f}}
\newcommand{\bfg}{\boldsymbol{g}}
\newcommand{\bfh}{\boldsymbol{h}}
\newcommand{\bfj}{\boldsymbol{j}}
\newcommand{\bfk}{\boldsymbol{k}}
\newcommand{\bfl}{\boldsymbol{l}}
\newcommand{\bfm}{\boldsymbol{m}}
\newcommand{\bfw}{\boldsymbol{w}}
\newcommand{\bfx}{\boldsymbol{x}}
\newcommand{\bfc}{\boldsymbol{c}}
\newcommand{\bfv}{\boldsymbol{v}}
\newcommand{\bfb}{\boldsymbol{b}}
\newcommand{\bfn}{\boldsymbol{n}}
\newcommand{\bfA}{\boldsymbol{A}}
\newcommand{\bfZ}{\boldsymbol{Z}}
\newcommand{\bfE}{\boldsymbol{E}}
\newcommand{\bfR}{\boldsymbol{R}}
\newcommand{\bfT}{\boldsymbol{T}}
\newcommand{\bfY}{\boldsymbol{Y}}
\newcommand{\bfU}{\boldsymbol{U}}
\newcommand{\bfI}{\boldsymbol{I}}
\newcommand{\bfO}{\boldsymbol{O}}
\newcommand{\bfP}{\boldsymbol{P}}
\newcommand{\bfQ}{\boldsymbol{Q}}
\newcommand{\bfS}{\boldsymbol{S}}
\newcommand{\bfD}{\boldsymbol{D}}
\newcommand{\bfF}{\boldsymbol{F}}
\newcommand{\bfG}{\boldsymbol{G}}
\newcommand{\bfH}{\boldsymbol{H}}
\newcommand{\bfJ}{\boldsymbol{J}}
\newcommand{\bfK}{\boldsymbol{K}}
\newcommand{\bfL}{\boldsymbol{L}}
\newcommand{\bfM}{\boldsymbol{M}}
\newcommand{\bfW}{\boldsymbol{W}}
\newcommand{\bfX}{\boldsymbol{X}}
\newcommand{\bfC}{\boldsymbol{C}}
\newcommand{\bfV}{\boldsymbol{V}}
\newcommand{\bfB}{\boldsymbol{B}}
\newcommand{\bfN}{\boldsymbol{N}}

\newcommand{\bfalpha}{\boldsymbol{\alpha}}
\newcommand{\bfnu}{\boldsymbol{\nu}}
\newcommand{\bfbeta}{\boldsymbol{\beta}}
\newcommand{\bfxi}{\boldsymbol{\xi}}
\newcommand{\bfXi}{\boldsymbol{\Xi}}
\newcommand{\bfgamma}{\boldsymbol{\gamma}}
\newcommand{\bfGamma}{\boldsymbol{\Gamma}}
\newcommand{\bfdelta}{\boldsymbol{\delta}}
\newcommand{\bfDelta}{\boldsymbol{\Delta}}
\newcommand{\bfpi}{\boldsymbol{\pi}}
\newcommand{\bfPi}{\boldsymbol{\Pi}}
\newcommand{\bfepsilon}{\boldsymbol{\epsilon}}
\newcommand{\bfvarepsilon}{\boldsymbol{\varepsilon}}
\newcommand{\bfrho}{\boldsymbol{\rho}}
\newcommand{\bfvarrho}{\boldsymbol{\varrho}}
\newcommand{\bfzeta}{\boldsymbol{\zeta}}
\newcommand{\bfsigma}{\boldsymbol{\sigma}}
\newcommand{\bfSigma}{\boldsymbol{\Sigma}}
\newcommand{\bfeta}{\boldsymbol{\eta}}
\newcommand{\bftau}{\boldsymbol{\tau}}
\newcommand{\bftheta}{\boldsymbol{\theta}}
\newcommand{\bfvartheta}{\boldsymbol{\vartheta}}
\newcommand{\bfTheta}{\boldsymbol{\Theta}}
\newcommand{\bfupsilon}{\boldsymbol{\upsilon}}
\newcommand{\bfUpsilon}{\boldsymbol{\Upsilon}}
\newcommand{\bfiota}{\boldsymbol{\iota}}
\newcommand{\bfphi}{\boldsymbol{\phi}}
\newcommand{\bfvarphi}{\boldsymbol{\varphi}}
\newcommand{\bfPhi}{\boldsymbol{\Phi}}
\newcommand{\bfkappa}{\boldsymbol{\kappa}}
\newcommand{\bfchi}{\boldsymbol{\chi}}
\newcommand{\bflambda}{\boldsymbol{\lambda}}
\newcommand{\bfLambda}{\boldsymbol{\Lambda}}
\newcommand{\bfpsi}{\boldsymbol{\psi}}
\newcommand{\bfPsi}{\boldsymbol{\Psi}}
\newcommand{\bfmu}{\boldsymbol{\mu}}
\newcommand{\bfomega}{\boldsymbol{\omega}}

\newcommand{\rma}{\mathrm{a}}
\newcommand{\rmz}{\mathrm{z}}
\newcommand{\rme}{\mathrm{e}}
\newcommand{\rmr}{\mathrm{r}}
\newcommand{\rmt}{\mathrm{t}}
\newcommand{\rmy}{\mathrm{y}}
\newcommand{\rmu}{\mathrm{u}}
\newcommand{\rmi}{\mathrm{i}}
\newcommand{\rmo}{\mathrm{o}}
\newcommand{\rmp}{\mathrm{p}}
\newcommand{\rmq}{\mathrm{q}}
\newcommand{\rms}{\mathrm{s}}
\newcommand{\rmd}{\mathrm{d}}
\newcommand{\rmf}{\mathrm{f}}
\newcommand{\rmg}{\mathrm{g}}
\newcommand{\rmh}{\mathrm{h}}
\newcommand{\rmj}{\mathrm{j}}
\newcommand{\rmk}{\mathrm{k}}
\newcommand{\rml}{\mathrm{l}}
\newcommand{\rmm}{\mathrm{m}}
\newcommand{\rmw}{\mathrm{w}}
\newcommand{\rmx}{\mathrm{x}}
\newcommand{\rmc}{\mathrm{c}}
\newcommand{\rmv}{\mathrm{v}}
\newcommand{\rmb}{\mathrm{b}}
\newcommand{\rmn}{\mathrm{n}}
\newcommand{\rmA}{\mathrm{A}}
\newcommand{\rmZ}{\mathrm{Z}}
\newcommand{\rmE}{\mathrm{E}}
\newcommand{\rmR}{\mathrm{R}}
\newcommand{\rmT}{\mathrm{T}}
\newcommand{\rmY}{\mathrm{Y}}
\newcommand{\rmU}{\mathrm{U}}
\newcommand{\rmI}{\mathrm{I}}
\newcommand{\rmO}{\mathrm{O}}
\newcommand{\rmP}{\mathrm{P}}
\newcommand{\rmQ}{\mathrm{Q}}
\newcommand{\rmS}{\mathrm{S}}
\newcommand{\rmD}{\mathrm{D}}
\newcommand{\rmF}{\mathrm{F}}
\newcommand{\rmG}{\mathrm{G}}
\newcommand{\rmH}{\mathrm{H}}
\newcommand{\rmJ}{\mathrm{J}}
\newcommand{\rmK}{\mathrm{K}}
\newcommand{\rmL}{\mathrm{L}}
\newcommand{\rmM}{\mathrm{M}}
\newcommand{\rmW}{\mathrm{W}}
\newcommand{\rmX}{\mathrm{X}}
\newcommand{\rmC}{\mathrm{C}}
\newcommand{\rmV}{\mathrm{V}}
\newcommand{\rmB}{\mathrm{B}}
\newcommand{\rmN}{\mathrm{N}}

\newcommand{\calA}{\mathcal{A}}
\newcommand{\calZ}{\mathcal{Z}}
\newcommand{\calE}{\mathcal{E}}
\newcommand{\calR}{\mathcal{R}}
\newcommand{\calT}{\mathcal{T}}
\newcommand{\calY}{\mathcal{Y}}
\newcommand{\calU}{\mathcal{U}}
\newcommand{\calI}{\mathcal{I}}
\newcommand{\calO}{\mathcal{O}}
\newcommand{\calP}{\mathcal{P}}
\newcommand{\calQ}{\mathcal{Q}}
\newcommand{\calS}{\mathcal{S}}
\newcommand{\calD}{\mathcal{D}}
\newcommand{\calF}{\mathcal{F}}
\newcommand{\calG}{\mathcal{G}}
\newcommand{\calH}{\mathcal{H}}
\newcommand{\calJ}{\mathcal{J}}
\newcommand{\calK}{\mathcal{K}}
\newcommand{\calL}{\mathcal{L}}
\newcommand{\calM}{\mathcal{M}}
\newcommand{\calW}{\mathcal{W}}
\newcommand{\calX}{\mathcal{X}}
\newcommand{\calC}{\mathcal{C}}
\newcommand{\calV}{\mathcal{V}}
\newcommand{\calB}{\mathcal{B}}
\newcommand{\calN}{\mathcal{N}}

\newcommand{\sfA}{\mathsf{A}}
\newcommand{\sfZ}{\mathsf{Z}}
\newcommand{\sfE}{\mathsf{E}}
\newcommand{\sfR}{\mathsf{R}}
\newcommand{\sfT}{\mathsf{T}}
\newcommand{\sfY}{\mathsf{Y}}
\newcommand{\sfU}{\mathsf{U}}
\newcommand{\sfI}{\mathsf{I}}
\newcommand{\sfO}{\mathsf{O}}
\newcommand{\sfP}{\mathsf{P}}
\newcommand{\sfQ}{\mathsf{Q}}
\newcommand{\sfS}{\mathsf{S}}
\newcommand{\sfD}{\mathsf{D}}
\newcommand{\sfF}{\mathsf{F}}
\newcommand{\sfG}{\mathsf{G}}
\newcommand{\sfH}{\mathsf{H}}
\newcommand{\sfJ}{\mathsf{J}}
\newcommand{\sfK}{\mathsf{K}}
\newcommand{\sfL}{\mathsf{L}}
\newcommand{\sfM}{\mathsf{M}}
\newcommand{\sfW}{\mathsf{W}}
\newcommand{\sfX}{\mathsf{X}}
\newcommand{\sfC}{\mathsf{C}}
\newcommand{\sfV}{\mathsf{V}}
\newcommand{\sfB}{\mathsf{B}}
\newcommand{\sfN}{\mathsf{N}}

\newcommand{\bbA}{\mathbb{A}}
\newcommand{\bbZ}{\mathbb{Z}}
\newcommand{\bbE}{\mathbb{E}}
\newcommand{\bbR}{\mathbb{R}}
\newcommand{\bbT}{\mathbb{T}}
\newcommand{\bbY}{\mathbb{Y}}
\newcommand{\bbU}{\mathbb{U}}
\newcommand{\bbI}{\mathbb{I}}
\newcommand{\bbO}{\mathbb{O}}
\newcommand{\bbP}{\mathbb{P}}
\newcommand{\bbQ}{\mathbb{Q}}
\newcommand{\bbS}{\mathbb{S}}
\newcommand{\bbD}{\mathbb{D}}
\newcommand{\bbF}{\mathbb{F}}
\newcommand{\bbG}{\mathbb{G}}
\newcommand{\bbH}{\mathbb{H}}
\newcommand{\bbJ}{\mathbb{J}}
\newcommand{\bbK}{\mathbb{K}}
\newcommand{\bbL}{\mathbb{L}}
\newcommand{\bbM}{\mathbb{M}}
\newcommand{\bbW}{\mathbb{W}}
\newcommand{\bbX}{\mathbb{X}}
\newcommand{\bbC}{\mathbb{C}}
\newcommand{\bbV}{\mathbb{V}}
\newcommand{\bbB}{\mathbb{B}}
\newcommand{\bbN}{\mathbb{N}}

\usepackage{graphicx}
\makeatletter
\newcommand*\bigcdot{\mathpalette\bigcdot@{.65}}
\newcommand*\bigcdot@[2]{\mathbin{\vcenter{\hbox{\scalebox{#2}{$\m@th#1\bullet$}}}}}
\makeatother
\newcommand{\dotp}{\bigcdot}

\newcommand{\KL}{{D_\mathrm{KL}}}

\usepackage{longtable}
\usepackage{tabularx}
\usepackage{threeparttable}
\usepackage{adjustbox}
\usepackage{caption}
\usepackage{booktabs}
\usepackage{algorithm}
\usepackage{multirow}
\usepackage{algpseudocode}
\usepackage{enumitem}
\usepackage{tikz}
\usepackage{pgfplots}
\usetikzlibrary{plotmarks}
\pgfplotsset{compat=1.18}

\usepackage[normalem]{ulem}
\usetikzlibrary{shapes.geometric, shapes.misc, backgrounds, fit}
\usetikzlibrary{arrows.meta,positioning,calc}
\usepackage{newtxtext,newtxmath}
\DeclareMathOperator*{\minimize}{minimize}
\DeclareMathOperator*{\maximize}{maximize}
\usepackage[figuresright]{rotating}
\usepackage{subfigure}
\usepackage{longtable}
\usepackage{siunitx}
\usepackage{colortbl}
\usepackage{multirow}
\definecolor{myblue}{RGB}{24,95,165}
\definecolor{myred}{RGB}{163,45,45}
\definecolor{rowgray}{gray}{0.94}
\usepackage{colortbl}
\usepackage{microtype}
\definecolor{purplered}{RGB}{180,50,140}
\definecolor{headgray}{HTML}{E7EBF0}
\definecolor{subgray}{HTML}{F4F6F8}
\definecolor{posgreen}{HTML}{1B5E20}
\definecolor{sectionblue}{HTML}{16324F}
\definecolor{captiongray}{HTML}{333A42}
\hypersetup{colorlinks=true,linkcolor=sectionblue,urlcolor=sectionblue,citecolor=sectionblue}
\renewcommand{\arraystretch}{1.12}
\newcommand{\Yue}[1]{{\color{purplered}#1}}
\newcommand{\Antoine}[1]{\textcolor{orange}{Antoine: #1}}
\newcommand{\inlAP}[1]{\begin{color}{blue}\texttt{A:#1}\end{color}}
\newcommand\PNN{PDFL\xspace}
\newcommand\OCNN{ODFL\xspace}

\newcommand\MMNL{{MixMNL}\xspace}

\newcommand\nb{n \xspace}
\newcommand\Products{\mathcal{N}\xspace}

\usepackage{xspace}
\newcommand\state{\boldsymbol{x}_t\xspace}
\newcommand\feature{\phi\xspace}
\newcommand\prob{prob\xspace}
\newcommand{\namedeqtag}[1]{\stepcounter{equation}\tag{\textsf{#1}}}

\newcommand\TimeIndItinerary{ItPri\xspace}
\newcommand\TimeDepItinerary{ItPri$^{(t)}$\xspace}

\newcommand\TimeIndApp{JoPri\xspace}
\newcommand\TimeDepApp{JoPri$^{(t)}$\xspace}

\newcommand\TimeIndUnified{JoComPri\xspace}
\newcommand\TimeDepUnified{JoComPri$^{(t)}$\xspace}

\newtheorem{result}{Finding}

\makeatletter
\@ifundefined{TITLE}{\newcommand{\TITLE}[1]{\title{#1}}}{}
\@ifundefined{RUNAUTHOR}{\newcommand{\RUNAUTHOR}[1]{}}{}
\@ifundefined{RUNTITLE}{\newcommand{\RUNTITLE}[1]{}}{}
\@ifundefined{ARTICLEAUTHORS}{\newcommand{\ARTICLEAUTHORS}[1]{\author{#1}}}{}
\@ifundefined{AUTHOR}{\newcommand{\AUTHOR}[1]{#1}}{}
\@ifundefined{AFF}{\newcommand{\AFF}[1]{\\{\small #1}}}{}
\@ifundefined{ABSTRACT}{\newcommand{\ABSTRACT}[1]{\begin{abstract}#1\end{abstract}}}{}
\@ifundefined{KEYWORDS}{\newcommand{\KEYWORDS}[1]{\noindent\textbf{Keywords: }#1\par\medskip}}{}
\@ifundefined{MANUSCRIPTNO}{\newcommand{\MANUSCRIPTNO}[1]{}}{}
\@ifundefined{HISTORY}{\newcommand{\HISTORY}[1]{}}{}
\@ifundefined{TheoremsNumberedThrough}{\newcommand{\TheoremsNumberedThrough}{}}{}
\@ifundefined{EquationsNumberedThrough}{\newcommand{\EquationsNumberedThrough}{}}{}
\@ifundefined{ECRepeatTheorems}{\newcommand{\ECRepeatTheorems}{}}{}
\@ifundefined{OneAndAHalfSpacedXI}{\newcommand{\OneAndAHalfSpacedXI}{}}{}
\@ifundefined{newblock}{\newcommand{\newblock}{}}{}
\@ifundefined{BIBand}{\newcommand{\BIBand}{and}}{}
\makeatother

\makeatletter
\renewenvironment{proof}[1][Proof]{%
  \par\noindent{\it #1.}\enspace\ignorespaces
}{%
  \hfill\qed\par
}
\makeatother

\TheoremsNumberedThrough
\EquationsNumberedThrough
\ECRepeatTheorems
\OneAndAHalfSpacedXI

\MANUSCRIPTNO{}

\makeatletter
\renewcommand{\theARTICLETOP}{}
\renewcommand{\theARTICLETOPLEFT}{}
\renewcommand{\theARTICLETOPRIGHT}{}
\makeatother
\RRHSecondLine{}
\LRHSecondLine{}
\hypersetup{hypertexnames=false,
  pdftitle={Scalable Dynamic Pricing of Substitutable Products through Structure-Guided Policy Learning},
  pdfauthor={Yue Su, Antoine Desir, Axel Parmentier}}

\begin{document}
\RUNAUTHOR{Su, D\'esir, and Parmentier}
\ARTICLEAUTHORS{
\AUTHOR{Yue Su}\AFF{Centre INRIA de l’Université de Lille, 59000 Lille, France
\\yue.su@inria.fr}
\AUTHOR{Antoine D\'esir}\AFF{INSEAD Business School, 77300 Fontainebleau, France.\\antoine.desir@insead.edu}
\AUTHOR{Axel Parmentier}\AFF{CERMICS, Ecole des Ponts ParisTech, 77420 Champs-sur-Marne, France\\axel.parmentier@enpc.fr}
}

\RUNTITLE{Structure-Guided Learning for Dynamic Pricing}
\TITLE{Scalable Dynamic Pricing of Substitutable Products through Structure-Guided Policy Learning}
\ABSTRACT{%
\textbf{\textit{Problem definition:}} We study dynamic pricing of substitutable products with finite, product-specific inventories. Customer substitution couples pricing decisions across products, while the inventory state makes exact dynamic programming intractable at realistic scale. \textbf{\textit{Methodology / results:}} We develop two MNL-guided policy-learning approaches that replace the dynamic program with a statistical mapping from inventory states to pricing decisions. The first learns prices directly, while the second learns inventory opportunity costs and converts them into prices using the optimal MNL pricing rule. Both policies are trained using decision-focused learning. We also propose an efficient method to generate anticipative customer-choice targets to train our policies.  \textbf{\textit{Managerial implications:}} We conduct an extensive numerical evaluation of our approaches. On small instances for which the optimal dynamic program can be computed, the learned policies achieve average optimality gaps below 0.4\%. On larger airline-motivated instances, they consistently improve on the tested revenue-management benchmarks. The comparison between the two architectures also highlights the role of model structure: using the MNL pricing characterization is particularly effective when demand is well described by MNL, while directly learning prices provides greater flexibility under heterogeneous mixed-MNL demand.
}
\maketitle

\section{Introduction}

Firms with finite capacity increasingly use dynamic prices to manage demand over short selling horizons. In many applications, however, demand streams are not independent: customers choose among substitutable products after comparing their attributes, availability, and prices. A prominent example is an airline selling several departure options in the same origin-destination market. Each itinerary has limited remaining capacity, and changing the price of one itinerary shifts demand toward or away from every other option. A useful pricing policy must therefore value scarce inventory over time while accounting for substitution across products.

This is a challenging problem, even when the underlying choice model is multinomial-logit (MNL). Indeed, under this model, \citet{dong2009dynamic} show that the optimal  price  has a very elegant characterization: a product-specific inventory opportunity cost plus a common markup. However, this does not remove the computational bottleneck: the opportunity costs are marginal values, which need to be computed through dynamic programming (DP). The number of states grows multiplicatively with product capacities and with the number of products, so this traditional approach does not scale due to the curse of dimensionality. Fluid and deterministic approximations, decomposition, and approximate value functions alleviate this burden by simplifying the Bellman problem or its dynamic coupling \citep{gallego1997multiproduct,liu2008choice,zhang2009approximate,meissner2012network}. More recently, successes in machine learning have prompted applications of reinforcement learning to dynamic pricing and airline revenue management \citep{rana2014real,bondoux2020reinforcement,shihab2019autonomous}. These model-free approaches replace an explicit dynamic program with a parameterized policy that scales to large state spaces. However, they forego any structure of the problem.

In this paper, we propose a hybrid approach that leverages the expressiveness and scalability of machine learning models while incorporating some economic structure from the revenue management literature. Specifically, our core idea is to learn the pricing policy defined by the high-dimensional DP by replacing it with a simpler MNL-guided surrogate mapping. We use decision-focused learning (DFL) to fit this mapping. This hybrid design is in line with the recent call for combining operations research structure and interpretability with machine learning flexibility and data-driven capabilities, instead of treating them as substitutes \citep{wiberg2026synergizing}.

Our use of DFL differs from its canonical role in predict-then-optimize systems, where it typically trains a prediction through a downstream optimizer. Instead, here, we use DFL only to remove the need for the costly value-function evaluation.
This use creates a distinctive supervision problem. Ideally, one would imitate optimal prices, but producing those labels requires solving the same high-dimensional DP we seek to avoid. A common alternative in optimization-augmented policy learning is to imitate high-quality anticipative decisions, i.e., decisions computed with perfect knowledge of future realizations \citep{baty2024combinatorial,hoppe2026structured}.
Anticipative \emph{prices}, however, are poor labels in our setting: multiple price vectors can induce the same customer decision, and the prices assigned to products not selected in a perfect-information scenario are essentially arbitrary. To overcome this, we propose to imitate the anticipative \emph{customer decision} rather than the anticipative pricing decision.

\vspace{0.5cm}

\noindent \textbf{Main contributions.} Our main contributions are as follows.
\begin{itemize}[leftmargin=*]
    \item We propose and compare two MNL-guided policy designs. Our price-based DFL architecture (hereafter \PNN) learns prices through an MNL behavioral choice layer, whereas our opportunity-cost-based DFL architecture (hereafter \OCNN) adds an opportunity-cost parameterization, inspired by the optimal price structure under the MNL model.

    \item We show that these architectures can be trained through the DFL methodology by casting them as regularized linear optimization under an appropriate convex regularizer.

    \item A key design in our framework is to imitate anticipative choices instead of anticipative prices. When the underlying choice model is MNL, we show that this can be done effectively by solving a capacitated maximum-cost flow problem.
    Moreover, when the underlying choice model is a discrete mixture of MNL (hereafter \MMNL), generating anticipative choices can be reduced to the case of MNL. Our approach is therefore without loss of generality since \MMNL models can approximate any random utility model \citep{mcfadden2000mixed}.

    \item In our numerical study, we showcase the performance of our methods. On small instances where we can compute the optimal DP solution, our policies
    are within 0.4\% of the optimal revenue on average.
    In larger instances, we significantly improve on state-of-the-art baselines, achieving an average revenue improvement of 12.0\% across four large-scale experimental settings.
    We also identify when each architecture is most useful: \OCNN dominates when the underlying choice model is MNL. However, when the underlying model is a \MMNL, \PNN is preferable. Finally, we show how to embed our solution in a fully data-driven setting.
\end{itemize}

The remainder of the paper is organized as follows. Section~\ref{sec:literature_review} positions our work with respect to the literature. Section~\ref{sec:problem_description} formulates the parallel-inventory pricing problem. Section~\ref{sec:architecture} presents the \PNN and \OCNN policy mappings. Section~\ref{sec:learning_algorithm} casts them as DFL pipelines with structured losses and then develops the anticipative choice-label generation problem. Sections~\ref{sec:experimental_design} and~\ref{sec:numerical_results} describe the experimental design and numerical results.

\section{Literature Review} \label{sec:literature_review}

Our paper relates to four streams of work: choice-based dynamic pricing, scalable approximation in revenue management, decision-focused and optimization-augmented learning, and hybrid machine learning–operations research approaches.

\subsection{Choice-Based Dynamic Pricing with Finite Product Inventories}

Revenue management (RM) models can first be distinguished by their resource structure (a single resource, multiple product-specific resources, or a network of shared resources) and by whether demands are independent or coupled through customer choice \citep{gallego2019revenue}. A second distinction concerns the control: much of choice-based RM controls availability or assortments, whereas our decision is a continuous price vector and every in-stock product remains available. Our focal problem lies in the choice-based, dynamic-pricing part of this taxonomy. Each product has a dedicated inventory, and selling product $i$ consumes one unit of that inventory only. In network RM notation the consumption matrix is the identity. We accordingly use the term \emph{choice-based dynamic pricing with parallel product-specific inventories}. A true airline-network extension would allow an itinerary to consume capacity on several shared flight legs.

For independent demand, \citet{gallego1994optimal} establish foundational structural and asymptotic results for finite-inventory dynamic pricing with a single product. \citet{gallego1997multiproduct} study multiple products with general demand interactions and develop a deterministic approximation that provides a scalable upper bound. The closest structural ancestor is \citet{dong2009dynamic}, who study finite-horizon dynamic pricing of multiple substitutable, nonreplenishable products under MNL choice. Their key result decomposes each statewise optimal price into a product-specific marginal inventory value and a common markup. This leads to a DP algorithm that still suffers from the curse of dimensionality, but for which the inner optimization problem in the Bellman equation is tractable. This tractability result forms the foundation of our \OCNN architecture.
\citet{akccay2010joint} show that vertical differentiation can yield aggregation, monotonicity, and a polynomial algorithm, whereas horizontal differentiation retains high-dimensional coupling. \citet{suh2011dynamic} further show that marginal inventory values have cleaner monotonicity than the optimal prices themselves. Together, these results motivate opportunity costs as a more natural learning target than the full value function or a presumed monotone price path.

Static multiproduct pricing provides complementary structure. \citet{hanson1996optimizing} and \citet{li2011pricing} exploit transformations from prices to market shares under MNL and nested-logit demand, and \citet{gallego2014multiproduct} derive markup structure for more general nested-logit specifications. Richer heterogeneity makes even the static problem harder: \citet{vandegeer2022price} develop an approximation scheme for finite-mixture logit whose complexity is polynomial in products but exponential in the number of customer segments. This contrast motivates using a tractable MNL representation while explicitly testing policies in \MMNL environments.

The closest recent dynamic-pricing paper is \citet{goyal2026showall}. They study sequential customers, finite inventories of substitutable products, and a show-all constraint requiring every available product to carry a finite feasible price. They note that tractability had remained unresolved even under MNL, and develop asymptotic and approximation guarantees, including a large-inventory MNL policy. Their theory and our approach are complementary: they address the computational barrier through approximation and asymptotic analysis, whereas we amortize the dynamic optimization by learning a state-dependent policy from decision-oriented offline supervision. \citet{song2026dynamic} study a related dynamic pricing and inventory problem with periodic replenishment and backlogging. Replenishment changes the intertemporal scarcity mechanism, so that setting is less direct than the fixed-inventory show-all problem but demonstrates broader recent interest in structured multiproduct pricing.

\subsection{Scalable Approximations}
\label{subsec:scalable_rm}

Choice-based network RM provides an adjacent literature on scalable optimization. \citet{talluri2004revenue} formulate RM under general discrete choice, primarily with offer-set control. Choice-based linear programs and column generation are developed by \citet{liu2008choice} and \citet{bront2009column}. In approximate DP, \citet{zhang2009approximate} use an affine value-function approximation; \citet{meissner2012network} introduce a separable nonlinear approximation with inventory-sensitive bid prices; and \citet{meissner2013enhanced} strengthen choice-network relaxations. For dynamic pricing with shared network resources, \citet{ke2019approximate} formulate approximate linear programs. These methods approximate or bound the Bellman value function. Our setting has parallel rather than shared resources, and \OCNN instead uses proxies for the marginal continuation values required by the pricing decision; \PNN learns the price policy more directly.

Deterministic and fluid reoptimization address scalability differently: instead of approximating the Bellman value function, they repeatedly or adaptively resolve a tractable deterministic problem. In choice-based network RM with exogenous prices, \citet{jasin2012resolving} periodically resolve a certainty-equivalent deterministic linear program and implement probabilistic controls. For endogenous pricing, \citet{jasin2014reoptimization} develops a self-adjusting price control based on one initial fluid optimization, optionally supplemented by a limited number of reoptimizations.

\subsection{Decision-Focused and Optimization-Augmented Learning}

In its canonical predict-then-optimize use, DFL trains a predictive model according to the quality of the decisions induced by a downstream optimizer, rather than predictive accuracy alone \citep{donti2017task,agrawal2019differentiable,elmachtoub2022smart}. The optimizer usually remains the map from predictions to decisions; this is distinct from fully end-to-end policy learning, which removes the explicit intermediate model. Our targets use a different interface. We keep the known or estimated choice model fixed and use decision-oriented targets and differentiable structured layers to learn the state-to-decision mapping that otherwise requires dynamic optimization. Fenchel--Young losses provide a general way to train such structured prediction maps generated by regularized optimization \citep{blondel2020learning}.

Closest methodologically are optimization-augmented learning approaches that imitate high-quality anticipative solutions: \citet{baty2024combinatorial} apply this idea to dynamic vehicle routing, and \citet{hoppe2026structured} use structured losses and optimization layers for policy learning. Reviews by \citet{mandi2024decision} and \citet{schiffer2026combinatorial} place these methods within the broader predict-and-optimize and combinatorial-optimization-augmented learning literature.

\subsection{Machine Learning and Operations Research Hybrid Approaches}

Our work also contributes to the stream of papers trying to propose hybrid machine learning and operations research approaches \citep{wiberg2026synergizing}. Two representative studies are \citet{aouad2025representing}, which introduce RUMnets, neural choice models that can approximate any random-utility model while remaining consistent with random-utility maximization, and  \citet{xie2026deepstock}, which regularizes deep-RL inventory policies with classical base-stock structure. These papers illustrate how domain structure can guide either what a learning model represents or how a policy is learned. Our work applies a similar principle to dynamic pricing.

\section{Problem Formulation} \label{sec:problem_description}

We consider a finite-horizon dynamic pricing problem with substitutable products and nonreplenishable inventories.
Let $\nb$ denote the number of products, $\Products =\{1,\ldots,\nb\}$ the product set, and $0$ the no-purchase option. We consider a discrete-time problem with $T$ decision periods. For each product $i$, we denote its inventory at time $t$ by $I_{i,t}$. At the beginning of period $t$, we denote the system state as $\boldsymbol x_t=(t,\boldsymbol I_t)$, where $\boldsymbol I_t=(I_{1,t},\ldots,I_{\nb,t})$, and we define the available set $\mathcal{A}(\boldsymbol I_t)$ as
\[
\mathcal{A}(\boldsymbol I_t)=\{i\in \Products:I_{i,t}>0\}.
\]
We also denote the starting inventory of product $i$ by $c_i$, and assume that $c_i$ is a nonnegative integer. We thus have $\boldsymbol I_1=\boldsymbol c=(c_1,\ldots,c_{\nb})$.  For simplicity, we index time by customer-arrival epochs and assume that exactly one customer arrives in each period and chooses at most one product, corresponding to the normalized case \(\lambda=1\) of the arrival model in \cite{dong2009dynamic}.

Each product has its own dedicated inventory. A sale of product $i$ consumes one unit of inventory $i$ and no other resource. Thus the model can be represented as a special multi-resource RM problem with an identity consumption matrix, but there is no physical resource sharing across products. The difficulty comes from demand-side substitution: the price and availability of one product affect the purchase probabilities of all the others. In the airline setting, this is known as the parallel flight problem \citep{zhang2005revenue}, in which products are parallel itineraries or departure options in the same origin-destination market.

At state $\boldsymbol x_t$, the firm posts a nonnegative price $r_{i,t}$ for every available product. We write
\[
\mathcal R(\boldsymbol I_t)=\mathbb R_+^{|\mathcal{A}(\boldsymbol I_t)|}
\]
for the feasible prices of available products and use the convention $r_{i,t}=+\infty$ for an out-of-stock product. We assume \textit{show-all constraints}, under which all available products are always displayed at finite, feasible prices, as in \cite{goyal2026showall}.
Given $\boldsymbol r_t\in\mathcal R(\boldsymbol I_t)$, let
$ \mathbb{P}_t(i\;|\;\boldsymbol r_t, \boldsymbol I_t)$ denote the probability that an arriving customer chooses product $i$ upon seeing the available product set $\mathcal{A}(\boldsymbol I_t)$.
If the customer buys product $i$ at time $t$, the inventory evolves as $\boldsymbol I_{t+1}=\boldsymbol I_t-\boldsymbol e_i$ where $\boldsymbol e_i\in\{0,1\}^{\nb}$ denotes the $i$-th unit vector; if there is no purchase, then $\boldsymbol I_{t+1}  = \boldsymbol I_{t} $. By convention, we set $\boldsymbol{e}_0 = \boldsymbol{0}$.

Let $Y_t\in\{0,1,\ldots,\nb\}$ denote the realized choice and $V_t(\boldsymbol I_t)$ be the maximum expected revenue from period $t$ through $T$, with terminal value $V_{T+1}(\boldsymbol I_{T+1})=0$. The Bellman equation is
\begin{equation}
\begin{aligned}
V_t(\boldsymbol I_t)
=\max_{\substack{
\boldsymbol r_t\in\mathcal R(\boldsymbol I_t)}
}
\mathbb{E}_{Y_t\sim\mathbb P_t(\,\cdot\mid\boldsymbol r_t,\boldsymbol I_t)} \!\Big[ r_{Y_t,t}
+V_{t+1}\!\left(\boldsymbol I_t-\boldsymbol e_{Y_t}\right) %
\Big].
\end{aligned}
\namedeqtag{Bellman}
\label{eq:bellman_original}
\end{equation}
Following the standard discrete-difference notation, we further define the marginal continuation value by
\begin{equation*}
\Delta_t^i(\boldsymbol I_t)
:=V_{t+1}(\boldsymbol I_t)-V_{t+1}(\boldsymbol I_t-\boldsymbol e_i),
\qquad \forall i\in \mathcal{A}(\boldsymbol I_t),
\end{equation*}
and collect these inventory opportunity costs in
$\boldsymbol\Delta_t(\boldsymbol I_t)
=(\Delta_t^i(\boldsymbol I_t))_{i\in \mathcal{A}(\boldsymbol I_t)}$. Subtracting and adding $V_{t+1}(\boldsymbol I_t)$ on the right-hand side of Equation~\eqref{eq:bellman_original} gives
\[
V_t(\boldsymbol I_t)
=V_{t+1}(\boldsymbol I_t)
+\max_{\boldsymbol r_t\in\mathcal R(\boldsymbol I_t)}
\xi_t(\boldsymbol r_t, \boldsymbol I_t),
\]
where
\begin{equation*}
\xi_t(\boldsymbol r_t, \boldsymbol I_t)
=\sum_{i\in \mathcal{A}(\boldsymbol I_t)}
\mathbb{P}_t(i\;|\boldsymbol r_t, \boldsymbol I_t)
\left[r_{i,t}- \Delta_t^i(\boldsymbol I_t) \right].
\end{equation*}
Thus the optimal action at a state solves the one-period pricing problem
\begin{equation}
\boldsymbol r_t^*(\boldsymbol I_t)
\in\argmax_{\boldsymbol r_t\in\mathcal R(\boldsymbol I_t)}
\xi_t(\boldsymbol r_t, \boldsymbol I_t).
\namedeqtag{statewise pricing}
\label{eq:focal_optimization_prob}
\end{equation}

This reformulation isolates the source of the computational challenge. Exact DP stores or evaluates the high-dimensional function $V_t$ over $\prod_i(c_i+1)$ inventory vectors. Yet once the opportunity-cost vector $\boldsymbol\Delta_t(\boldsymbol I_t)$ is known, the remaining statewise pricing problem is low dimensional and, under MNL choice, has a tractable markup characterization. The Bellman function affects the current price only through these marginal values. This observation motivates the central design choice in Section~\ref{sec:architecture}: learn a decision-relevant policy representation, either prices directly or an opportunity-cost parameterization sufficient to recover prices, rather than reconstruct the full Bellman function.

\section{Structure-Guided Policy Architectures} \label{sec:architecture}

Our objective is to learn a deployable pricing policy
\[
\pi_w:\boldsymbol x_t=(t,\boldsymbol I_t)\longmapsto\boldsymbol r_t
\]
that maps the observable state to a feasible price vector via imitation learning.
To this end, we need labels.
We cannot use optimal decisions as labels as this would require solving the original DP.
Instead, we use the common approach in optimization-augmented policy learning that consists of producing anticipative labels computed with full knowledge of a sampled future \citep{baty2024combinatorial,jungelLearningBasedOnlineOptimization2025}. Applying this idea directly here means producing anticipative price labels. However, those prices are degenerate: many price vectors induce the same realized purchase, while the prices of unselected products can be set arbitrarily. Consequently, we propose instead to use the induced customer choice as the target. Note that this is related to inverse optimization \citep{ahuja2001inverse,chan2021inverse}: we learn prices that induce desired allocation decisions. The two architectures we propose therefore start with  a statistical predictor $\varphi_w$ parametrized by $w$. Since we train using a first-order method and automatic differentiation, any neural network can be used for $\varphi_w$. Then, in both architectures, the final layer is an MNL choice map $\bbQ$ that turns pricing decisions ${\bfr}$ into customer-choice probabilities $\bfy$. Note that we assume in the next sections that $\bbQ$ is known and given. We explain in Section~\ref{sec:experimental_design} how we compute it in practice either from $\mathbb{P}$, if known, or directly from data.
These choice probabilities can be evaluated against the anticipative customer choice targets. The difference between the two architectures lies in what \(\varphi_w\) predicts. In the price-based architecture (\PNN), we directly predict the pricing decisions, while in the opportunity-cost-based architecture (\OCNN) we predict opportunity costs that are then turned into pricing decisions through an MNL-guided optimization layer.  In this section, we introduce our two architectures. Section~\ref{sec:learning_algorithm} then shows how DFL can be used to learn these from data.

\subsection{Price-Based Architecture}
\label{sec:MNLlayer}
We introduce the MNL choice layer, which is a deterministic layer defined by some parameters $\boldsymbol a$ and $\beta$, as
\begin{equation}
\begin{aligned}
\bbQ(i\;| \boldsymbol r, \boldsymbol I)
&=\frac{\exp(a_i-\beta r_{i})}
{1+\sum_{j\in \mathcal{A}(\boldsymbol I)}\exp(a_j-\beta r_{j})},
&& \forall i\in \mathcal{A}(\boldsymbol I),\\
\bbQ(0\;| \boldsymbol r, \boldsymbol I)
&=\frac{1}
{1+\sum_{j\in \mathcal{A}(\boldsymbol I)}\exp(a_j-\beta r_{j})},
\end{aligned}
\label{eq:MNL-choice}
\end{equation}
with zero coordinates for unavailable products. We explain in the numerical section how we select $\boldsymbol a$ and $\beta$ to create this proxy MNL environment. However, recall that we are trying here to learn a pricing policy and as a result, the goal is only to learn $\varphi_w$.  In \PNN, we therefore predict choice $\hat{\boldsymbol y}_t$ from state $\boldsymbol x_t$ as
\begin{align*}
\hat{\boldsymbol y}_t &=  \bbQ(\hat{\boldsymbol r}_t, \boldsymbol I_t), \\
\hat{\boldsymbol r}_t &=\varphi_w (\boldsymbol x_t).
\end{align*}

The MNL choice map is smooth and imposes interpretable substitution structure. However, we emphasize that it is only a proxy or surrogate map that we use as a structural training device: the true choice probabilities $\mathbb P$ in the dynamic program need not be MNL, and the anticipative targets may be generated from a richer scenario model such as \MMNL. Moreover, note that we are ultimately interested in the pricing decision. Therefore, when using the policy, we only need the pricing decisions ${\bfr}$ and do not need to apply this final layer.
Figure~\ref{fig:two_proposed_architectures} illustrates the architecture.

\begin{figure}[t]
\centering
\resizebox{0.95\textwidth}{!}{%
\begin{tikzpicture}[
font=\sffamily\small,
data/.style={ellipse,minimum width=1.3cm,minimum height=0.9cm,align=center,draw=gray!60,fill=gray!10,thick},
block/.style={rounded corners,minimum width=2.0cm,minimum height=1.3cm,align=center,draw=blue!60!black,fill=blue!5,thick},
coblock/.style={chamfered rectangle,chamfered rectangle xsep=4pt,minimum width=2.0cm,minimum height=1.3cm,align=center,draw=orange!80!black,fill=orange!10,thick},
fwd/.style={->,>={Latex[width=2.2mm,length=2.2mm]},draw=blue!60!black,line width=1.3pt}
]

\node[font=\bfseries] at (-2.0,0) {PDFL};
\node[data] (s1) at (0,0) {State\\$\boldsymbol x_t$};
\node[block,right=0.9cm of s1] (m1) {Statistical Predictor\\$\varphi_w$};
\node[data,right=0.9cm of m1] (p1) {Prices\\$\widehat{\boldsymbol r}_t$};
\node[coblock,right=0.9cm of p1] (c1) {MNL choice\\mapping};
\node[data,right=0.9cm of c1] (y1) {Induced choices\\$\widehat{\boldsymbol y}_t$};
\draw[fwd] (s1)--(m1); \draw[fwd] (m1)--(p1); \draw[fwd] (p1)--(c1); \draw[fwd] (c1)--(y1);

\node[font=\bfseries] at (-2.0,-2.4) {ODFL};
\node[data] (s2) at (0,-2.4) {State\\$\boldsymbol x_t$};
\node[block,right=0.9cm of s2] (m2) {Statistical Predictor\\$\varphi_w$};
\node[data,right=0.9cm of m2] (o2) {Opp.~cost\\$\widehat{\boldsymbol\Delta}_t$};
\node[coblock,right=0.9cm of o2] (pm2) {Optimal MNL\\pricing map};
\node[data,right=0.9cm of pm2] (p2) {Prices\\$\widehat{\boldsymbol r}_t$};
\node[coblock,right=0.9cm of p2] (c2) {MNL choice\\mapping};
\node[data,right=0.9cm of c2] (y2) {Induced choices\\$\widehat{\boldsymbol y}_t$};
\draw[fwd] (s2)--(m2); \draw[fwd] (m2)--(o2); \draw[fwd] (o2)--(pm2); \draw[fwd] (pm2)--(p2); \draw[fwd] (p2)--(c2); \draw[fwd] (c2)--(y2);

\end{tikzpicture}}
\caption{\centering The two proposed architectures.}
\label{fig:two_proposed_architectures}
\end{figure}
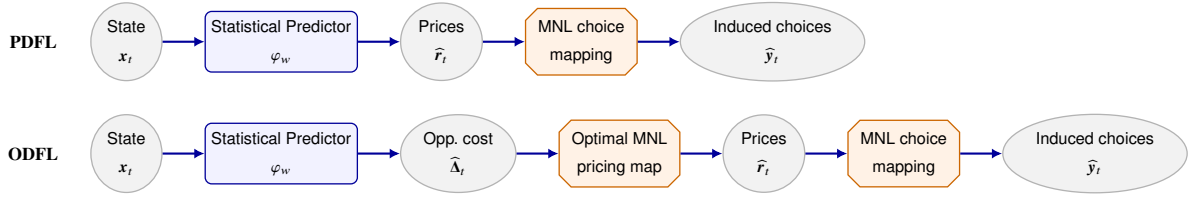

\subsection{Opportunity-Cost-Based Architecture} \label{subsec:oc_architecture}

When the true choice model is MNL, the statewise pricing problem \eqref{eq:focal_optimization_prob} admits a closed-form characterization: the optimal price equals a product-specific inventory opportunity cost plus a common markup.

\begin{lemma}[\cite{dong2009dynamic}]
\label{lem:closed_form_mnl_pricing}
For any nonnegative opportunity-cost vector $\boldsymbol\Delta_t$, the solution of Problem~\eqref{eq:focal_optimization_prob} when the choice model is MNL satisfies
\begin{equation}
r_{i,t}^{*}
=\Delta_t^i+\frac{m_t}{\beta},
\qquad i\in \mathcal{A}(\boldsymbol I_t), \nonumber
\label{eq:closed_form_mnl_price}
\end{equation}
where $m_t>1$ is the unique solution of
\begin{equation}
(m_t-1)e^{m_t}
=\sum_{i\in \mathcal{A}(\boldsymbol I_t)}
\exp\!\left(a_i-\beta\Delta_t^i\right).
\label{eq:closed_form_mnl_markup}
\end{equation}
\end{lemma}
We impose this very interpretable economic structure in our second architecture. As in \PNN, we impose this as a structural training device even if the true underlying model is not MNL. Specifically, we learn a statistical predictor $\varphi_w$, which is followed by the MNL pricing map \(h^{\mathrm{MNL}}\) and the deterministic MNL choice layer $\bbQ$,
such that we predict choice $\hat{\boldsymbol y}_t$ from state $\boldsymbol x_t$ as
\begin{align*}
\hat{\boldsymbol y}_t &=  \bbQ(\hat{\boldsymbol r}_t, \boldsymbol I_t), \\
\hat{\boldsymbol r}_t &= h_{\mathrm{MNL}} (\hat{\boldsymbol \Delta}_t ),\\
\hat{\boldsymbol \Delta}_t &=\varphi_w (\boldsymbol x_t),
\end{align*}
where $h_{\mathrm{MNL}}$ is the mapping given in Lemma~\ref{lem:closed_form_mnl_pricing}. Figure~\ref{fig:two_proposed_architectures} illustrates this mapping.

\subsection{Statistical Predictor and Residual Parameterization} \label{subsec:statistical predictor}

For the statistical predictor $\varphi_w$, we use a residual approach: we assume that we have access to an architecture-specific baseline from the literature, and use a generalized linear model (GLM) to correct it.
This approach has three advantages.
First, the residual approach is generic, allowing us to use the same GLM specification for both \PNN and \OCNN, as shown below. In contrast, a direct parameterization of $\varphi_w(\state)$ may require architecture-specific feature engineering because its interpretation and scale vary across architectures.
Second, it facilitates change management by correcting an existing RM policy. This residual approach provides an improved policy that does not deviate drastically from the reference policy. Third, from a practical perspective, using a residual approach makes learning more stable \citep{friedman2001greedy,he2016deep}.

More specifically, let \(\boldsymbol{\mathrm{Base}}(\state)\) denote the baseline vector, corresponding to baseline prices in \PNN and baseline opportunity costs in \OCNN.
Let \(\boldsymbol{\mathrm{Res}}_w(\state)\) denote the residual vector produced by the GLM, defined as
$$
\boldsymbol{\mathrm{Res}}_w(\state)
=
\left(
\mathrm{Res}_{i,w}(\state)
\right)_{i\in \mathcal{A}(\boldsymbol I_t)}.
$$
We consider two residual specifications:
\begin{enumerate}
\item Additive: $\varphi_w(\state) = \boldsymbol{\mathrm{Base}}(\state) + K_{\mathrm{add}}
\tanh\!\left(\boldsymbol{\mathrm{Res}}_w(\state)\right)$.
\item Multiplicative: $\varphi_w(\state) = \boldsymbol{\mathrm{Base}}(\state) \odot \exp\!\left(K_{\mathrm{mul}}\boldsymbol{\mathrm{Res}}_w(\state)\right)$.
\end{enumerate}
In the latter, we use $\odot$ to denote componentwise multiplication. The additive form learns a bounded correction around the baseline, whereas the multiplicative form learns a relative adjustment.
The hyperparameters $K_{\mathrm{add}}$ and $K_{\mathrm{mul}}$ control the magnitude of the residual correction and are selected based on validation performance. Implementation details are described in Section \ref{subsec:configurations}.

\section{Learning Algorithm} \label{sec:learning_algorithm}

The two architectures introduced in the previous section rely on a statistical predictor $\varphi_w$ parametrized by $w$.  In this section, we introduce the learning algorithm that identifies the parameter $w$. Note that in our work, we use a GLM for $\varphi_w$ but nothing we present is specific to this choice. In particular, the approach seamlessly extends if $\varphi_w$ is a neural network. We briefly introduce the key ingredients needed to use decision-focused learning (DFL), which, as mentioned in the Introduction, will be our approach. We then show how the two proposed architectures can be cast as DFL pipelines. All proofs in this section are relegated to Appendix \ref{appendix:proofs}.

\subsection{Decision-Focused Learning}
\begin{figure}[t]
\centering
\resizebox{0.96\textwidth}{!}{%
\begin{tikzpicture}[
    node distance=1.4cm and 1.4cm, font=\sffamily\small,
    data/.style={ellipse, minimum width=1.4cm, minimum height=1cm, align=center,
        draw=gray!60, fill=gray!10, thick},
    block/.style={rounded corners, minimum width=2.2cm, minimum height=1.5cm,
        align=center, draw=blue!60!black, fill=blue!5, thick},
    coblock/.style={chamfered rectangle, chamfered rectangle xsep=4pt,
        minimum width=2.2cm, minimum height=1.5cm, align=center,
        draw=orange!80!black, fill=orange!10, thick},
    fwd/.style={->, >={Latex[width=2.5mm,length=2.5mm]},
        draw=blue!60!black, line width=1.5pt},
    bwd/.style={->, >={Latex[width=2.5mm,length=2.5mm]},
        draw=red!70!black, text=red!70!black, dashed,
        line width=1.5pt, font=\scriptsize},
    sideInput/.style={draw=gray!40, fill=white, dashed,
        align=center, font=\footnotesize}
]

\node[data] (input) {State \\ $\state\in\mathcal X$};
\node[block, right=of input] (nn) {Statistical Predictor \\ $\varphi_w(\state)$};
\node[data, right=of nn, text width=1.6cm] (theta) {Params $\boldsymbol{\theta}_t$};
\node[coblock, right=of theta] (co) {\textbf{Optimization Layer} \\[-1mm]
    $\displaystyle \boldsymbol y_t\in
    \argmax_{\boldsymbol y\in\mathcal Y(\state)}
    \boldsymbol\theta_t^\top \bfy - \Omega(\bfy)$};
\node[data, right=of co] (output) {Output $\boldsymbol y_t$};
\node[block, right=of output, draw=red!60!black, fill=red!5] (loss)
    {Loss Evaluation \\ $\mathcal L(\boldsymbol y_t,\text{Target})$};
\node[sideInput, above=of loss, yshift=-.5cm] (target) {Target decisions};

\draw[fwd] (input)--(nn);
\draw[fwd] (nn)--node[above,font=\scriptsize,text=blue!60!black]{$\varphi_w(\state)$}(theta);
\draw[fwd] (theta)--(co);
\draw[fwd] (co)--(output);
\draw[fwd] (output)--(loss);
\draw[->, gray, thick] (target)--(loss);

\coordinate (bwdSE) at ($(co.south east)+(0,-1.3cm)$);
\coordinate (bwdSW) at ($(co.south west)+(0,-1.3cm)$);
\coordinate (bwdCorner) at (loss.south |- bwdSE);

\draw[bwd, rounded corners=10pt] (loss.south)--(bwdCorner)
    --node[near start,below,yshift=-2pt]{Backward $\nabla_y\mathcal L$}(bwdSE);
\draw[bwd, line width=2.5pt] (bwdSE)
    --node[midway,below,yshift=-4pt,align=center,font=\bfseries\scriptsize]
    {Differentiation \\ through Solver}(bwdSW);
\draw[bwd, rounded corners=10pt] (bwdSW)-|
    node[near end,below,yshift=-20pt,xshift=10pt]
    {Gradient $\nabla_w\mathcal L$}(nn.south);

\end{tikzpicture}
}
\caption{General DFL architecture. A statistical predictor parameterizes a surrogate optimization problem. We propagate gradients backwards through the optimization layer, enabling end-to-end training of the policy $\pi_w$ to minimize a downstream decision loss.}
\label{fig:general_neural_architecture}
\end{figure}
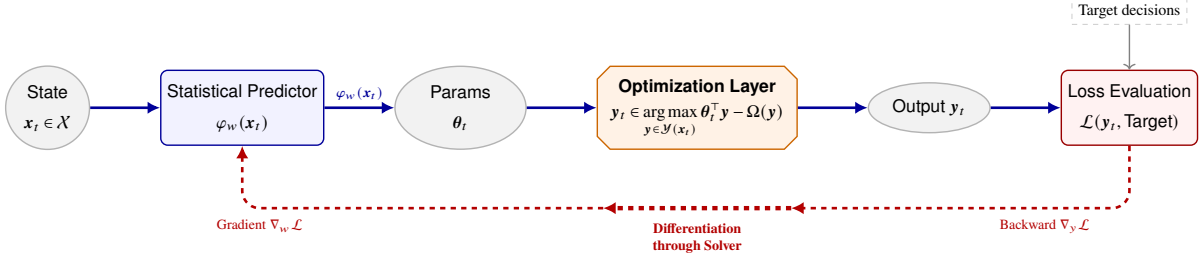

In DFL, a statistical predictor $\varphi_w$ parametrized by $w$ turns a state  $\bfx_t$ into a vector of parameters $\bftheta_t$. This vector is then used in the objective function of an optimization layer. The decision, which is the output of the optimization layer, is evaluated against a target through a loss function $\calL(\cdot, \cdot)$. Importantly, to apply DFL, the key is that the optimization layer can be written as a regularized linear optimization
\begin{equation}\label{eq:regularized_combinatorial_optimization_oracle_polytope}
\begin{aligned}
    \max \ & \  \bftheta^\top \bfy - \Omega(\bfy)\\
        & \boldsymbol y \in \mathcal{Y}
\end{aligned}
\end{equation}
for some convex feasible set $\mathcal{Y}$ and a strictly convex regularization function $\Omega(\cdot)$. Figure~\ref{fig:general_neural_architecture} illustrates a generic DFL architecture for dynamic settings, in which a single statistical predictor $\varphi_w$ maps any possible state $\boldsymbol{x}_t$ to a prediction used for decision making. Given a training set of state-label pairs $\{(\boldsymbol{x}_{t_i}^i,\boldsymbol{y}_{t_i}^i)\}_{i=1}^{S}$, where each pair corresponds to a system state at time step $t_i$, we train the predictor by solving the following supervised learning problem:
\begin{equation}
\label{eq:supervisedLearningWithFYL}
\min_w \frac{1}{S}\sum_{i=1}^{S}
\calL\left(\varphi_w(\boldsymbol{x}_{t_i}^i),\boldsymbol{y}_{t_i}^i\right).
\nonumber
\end{equation}

What makes this approach particularly appealing is that when the loss function is chosen appropriately with respect to the regularization function $\Omega(\cdot)$, the parameter $w$ can be learned efficiently using a first-order approach \citep{blondel2020learning,baty2024combinatorial}.

\paragraph{Fenchel--Young loss.} Inspecting Figure~\ref{fig:general_neural_architecture}, note that the optimization layer is a deterministic layer that maps $\boldsymbol \theta$ to a decision $\boldsymbol y$. This is similar to a softmax layer that transforms a vector of scores into a probability vector. In fact, a softmax layer is a special case of an optimization layer for which the mapping admits a closed form \citep{blondel2020learning}. Because there are no parameters to learn in this layer, we define the Fenchel--Young loss directly in the $\boldsymbol \theta$ space. Specifically, for some parameter $\boldsymbol \theta$ and target output $\boldsymbol{\bar{y}}$, the Fenchel--Young loss is defined as
\begin{equation}\label{eq:FYL_general_form}
    \begin{aligned}
        \mathcal{L}^{\text{FY}}_{\Omega}(\boldsymbol{\theta}, \boldsymbol{\bar{y}}) &= \max_{y \in \mathcal{Y}} \Big(\boldsymbol \theta^\top \boldsymbol y - \Omega(\boldsymbol y)\Big) - \Big(\boldsymbol \theta^\top \bar{\boldsymbol y} - \Omega(\bar{\boldsymbol y}) \Big)   \\
        &= \Omega^*(\boldsymbol \theta) + \Omega(\bar{\boldsymbol y}) - \boldsymbol \theta^\top \bar{\boldsymbol y}, \nonumber
    \end{aligned}
\end{equation}
where
\[
\Omega^*(\boldsymbol{\theta}) = \max_{\boldsymbol y \in \mathcal{Y}} \{\boldsymbol{\theta}^\top \boldsymbol{y} - \Omega(\boldsymbol{y})\}
\]
is the Fenchel conjugate of $\Omega$. By definition, note that the loss is chosen so that it is small when the optimal objective value of \eqref{eq:regularized_combinatorial_optimization_oracle_polytope} is close to the objective value attained at target $\bar{\boldsymbol y}$. For that reason, this loss is also sometimes referred to as a decision-aware loss. As highlighted by the index on the loss, the Fenchel--Young loss is intrinsically built for the specific regularization function $\Omega$. Furthermore, $\mathcal{L}^{\text{FY}}_{\Omega}(\cdot, \cdot)$ has desirable properties \citep{blondel2020learning}: it is non-negative, smooth, convex, and differentiable in $\boldsymbol{\theta}$ with gradient that takes the following form:
\begin{align}
\label{eq:gradien_FYL}
\nabla_{\boldsymbol{\theta}}\mathcal{L}^{\mathrm{FY}}_{\Omega}(\boldsymbol{\theta}, \boldsymbol{\bar{y}}) =
\nabla \Omega^*(\boldsymbol{\theta})- \boldsymbol{\bar{y}} \nonumber.
\end{align}
With these primitives in place, we next show how our two architectures can be cast as DFL. In other words, we need to define in each case what our parameter $\boldsymbol \theta$ is and what our regularized optimization layer is. Once this is done, we can leverage the algorithmic machinery developed for DFL. We further refer the reader to \cite{schiffer2026combinatorial} for a recent survey on DFL.

\subsection{Price-Based Architecture as DFL} \label{subsec:FY for PNN}

To show that \PNN can be learned using DFL, we define, for a given parameter $\boldsymbol \theta$, the following regularized optimization problem.
\begin{equation} \label{eq:optMNL}
\begin{aligned}
\max  \ & \ \boldsymbol \theta^T \boldsymbol{y} - \Omega^{\mathrm{\PNN}}(\boldsymbol{y}), \\
 & \boldsymbol{y} \in \mathcal{C}(\boldsymbol{I}),
\end{aligned}
\end{equation}
where $\mathcal{C}(\boldsymbol{I})$ is the probability simplex restricted to the available products in $\boldsymbol{I}$ defined as
\begin{equation}\label{eq:simplex}
\mathcal{C}(\boldsymbol I) = \left \{ \boldsymbol{y} \in\mathbb{R}_{\geq 0}^{\nb+1} \left |  \sum_{i\in \mathcal{A}(\boldsymbol{I})\cup \{0\}} y_i = 1,\quad y_i =0 \text{ for } i \not\in \mathcal{A}(\boldsymbol{I}) \cup \{0\} \right. \right \}, \nonumber
\end{equation}
and $\Omega^{\mathrm{\PNN}}(\cdot)$ is the negative Shannon entropy,
\begin{equation*}
    \Omega^{\mathrm{\PNN}}(\boldsymbol{y}) = \sum_{i=0}^{\nb} y_i\log y_i.
\end{equation*}
The vector of MNL choice probabilities is an optimal solution to the regularized optimization problem~\eqref{eq:optMNL}. This is a standard result   \citep{luce1959individual,mcfadden1972conditional,anderson1992discrete}, which we repeat here without proof.
\begin{lemma}
For a given $\boldsymbol \theta$, the vector of MNL choice probabilities defined as $y^*_i = \frac{\exp {\theta_i}}{(\sum_{j \in \mathcal{A}(\boldsymbol{I}) \cup \{0\}} \exp {\theta_j})}$ for $i \in \mathcal{A}(\boldsymbol{I}) \cup \{0\}$ and 0 otherwise, is an optimal solution to the regularized optimization problem \eqref{eq:optMNL}.
\end{lemma}
As a result, we can cast our \PNN through DFL by defining, for a state $\boldsymbol{x}_t$ and price vector
$\hat{\boldsymbol{r}}_t=\varphi_w(\boldsymbol{x}_t)$, the following parameter vector $\boldsymbol \theta_{\hat{\boldsymbol{r}}_t} = (\theta_{0},\; \ldots,\; \theta_{\nb})$ where
\[
\theta_{i}
=
\begin{cases}
a_i-\beta \hat{r}_{i,t},
& \text{if } i\in \mathcal{A}(\boldsymbol I_t),\\
0,
& \text{otherwise}.
\end{cases}
\]
Consequently, as illustrated in Figure~\ref{fig:new}, we can learn \PNN using DFL. In this case, the corresponding Fenchel conjugate $\Omega^*$ corresponds to the log-partition function \citep{wainwrightGraphicalModelsExponential2007} $\Omega^*(\boldsymbol{\theta})=\log\sum_{i \in \mathcal A(\boldsymbol I_t)\cup\{0\}} \exp(\theta_i)$.
Moreover, the Fenchel--Young loss coincides
with the negative log-likelihood associated with the MNL, and the gradient of the Fenchel--Young loss is given by
\[
\frac{\partial \mathcal{L}^{\mathrm{FY}}_{\Omega}}{\partial \theta_i}
    =
    \hat y_{i}-\bar y_i,
    \qquad i\in \mathcal{A}(\boldsymbol I_t),
\]
where $\hat{\boldsymbol{y}}$ is the solution to \eqref{eq:optMNL}. The gradient has a very intuitive interpretation: it is equal to the difference between the choice induced by the model $\hat{\boldsymbol{y}}$ and the target choice $\bar{\boldsymbol{y}}$.

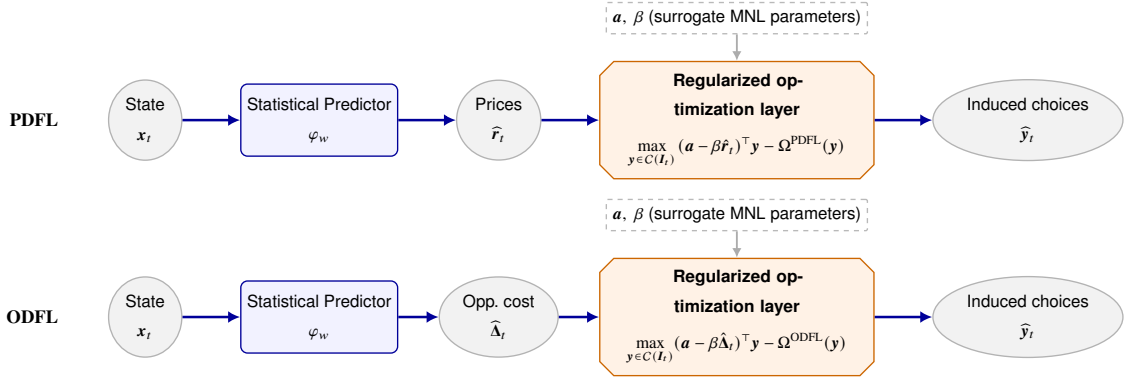
\begin{figure}[t]
\centering
\resizebox{0.9\textwidth}{!}{%
\begin{tikzpicture}[node distance=1.2cm and 1.2cm,font=\sffamily\small,
data/.style={ellipse,minimum width=1.4cm,minimum height=1cm,align=center,draw=gray!60,fill=gray!10,thick},
block/.style={rounded corners,minimum width=2.2cm,minimum height=1.5cm,align=center,draw=blue!60!black,fill=blue!5,thick},
coblock/.style={chamfered rectangle,chamfered rectangle xsep=5pt,minimum width=5.6cm,minimum height=2.2cm,align=center,draw=orange!80!black,fill=orange!10,thick},
param/.style={rectangle,rounded corners=1pt,minimum width=1.9cm,minimum height=0.65cm,align=center,draw=gray!55,dashed,fill=white,thick},
fwd/.style={->,>={Latex[width=2.5mm,length=2.5mm]},draw=blue!60!black,line width=1.5pt},
aux/.style={->,>={Latex[width=2mm,length=2mm]},draw=gray!65,line width=1.1pt}]

\node[data] (state1) {State \\ $\boldsymbol x_t$};
\node[block,right=1.2cm of state1] (model1) {Statistical Predictor \\ $\varphi_w$};
\node[data,right=1.2cm of model1] (price1) {Prices \\ $\widehat{\boldsymbol r}_t$};
\node[coblock,right=1.2cm of price1,text width=5.1cm] (map1) {\textbf{Regularized optimization layer} \\[1.5mm] $\displaystyle \max_{\boldsymbol y\in\mathcal{C}(\boldsymbol I_t)} {(\boldsymbol a-\beta\hat{\boldsymbol r}_{t})}^{\top}\boldsymbol y-\Omega^{\mathrm{\PNN}}(\boldsymbol y)$};
\node[param,above=0.55cm of map1] (param1) {$\boldsymbol a,\ \beta$ (surrogate MNL parameters)};
\node[data,right=1.2cm of map1] (choice1) {Induced choices \\ $\widehat{\boldsymbol y}_t$};
\node[font=\bfseries,left=0.9cm of state1] {PDFL};
\draw[fwd] (state1)--(model1); \draw[fwd] (model1)--(price1); \draw[fwd] (price1)--(map1); \draw[fwd] (map1)--(choice1); \draw[aux] (param1.south)--(map1.north);

\node[data] (state2) at (0,-4.1) {State \\ $\boldsymbol x_t$};
\node[block] (model2) at (model1 |- state2) {Statistical Predictor \\ $\varphi_w$};
\node[data] (opp2) at (price1 |- state2) {Opp.~cost \\ $\widehat{\boldsymbol\Delta}_t$};
\node[coblock,text width=5.1cm] (map2) at (map1 |- state2) {\textbf{Regularized optimization layer} \\[1.5mm] $\displaystyle \max_{\boldsymbol y\in\mathcal{C}(\boldsymbol I_t)} (\boldsymbol a-\beta\hat{\boldsymbol\Delta}_{t})^{\top}\boldsymbol y-\Omega^{\mathrm{\OCNN}}(\boldsymbol y)$};
\node[param,above=0.55cm of map2] (param2) {$\boldsymbol a,\ \beta$ (surrogate MNL parameters)};
\node[data] (choice2) at (choice1 |- state2) {Induced choices \\ $\widehat{\boldsymbol y}_t$};
\node[font=\bfseries,left=0.9cm of state2] {ODFL};
\draw[fwd] (state2)--(model2); \draw[fwd] (model2)--(opp2); \draw[fwd] (opp2)--(map2); \draw[fwd] (map2)--(choice2); \draw[aux] (param2.south)--(map2.north);

\end{tikzpicture}}
\caption{\centering The two proposed DFL architectures.}
\label{fig:new}
\end{figure}

\subsection{Opportunity-Cost-Based Architecture as DFL} \label{subsec:FY loss for OCNN}
To cast \OCNN using a regularized optimization layer, we leverage Lemma~\ref{lem:closed_form_mnl_pricing}. Specifically, the key idea in the proof of that result is to look at the pricing problem as an optimization problem over choice probabilities only. Indeed, one can use the MNL choice probabilities \eqref{eq:MNL-choice} to express the prices as a function of choice probabilities. As a result, for a given vector of opportunity costs $\hat{\boldsymbol\Delta}_t$, the optimal choice probabilities can be obtained by solving the following optimization problem
\begin{equation} \label{eq:OCNNopt}
\begin{aligned}
\max  \ & \ {\boldsymbol{\theta}}_{\hat{\boldsymbol\Delta}_t}^{\top}\boldsymbol{y} - \Omega^{\mathrm{\OCNN}}(\boldsymbol y), \\
 & \boldsymbol{y} \in \mathcal{C}(\boldsymbol{I}_t),
\end{aligned}
\end{equation}
where ${\boldsymbol{\theta}}_{\hat{\boldsymbol\Delta}_t} = (\theta_{0},\; \ldots,\; \theta_{\nb})$ is defined by
\begin{equation}
\label{eq:theta_opportunity_cost}
\theta_{i}
=
\begin{cases}
a_i-\beta \hat{\Delta}_{t}^{i},
& \text{if } i\in \mathcal{A}(\boldsymbol I_t),\\
0,
& \text{otherwise}.
\end{cases}
\end{equation}
and
\[
\Omega^{\mathrm{\OCNN}}(\boldsymbol y)
=
\sum_{i=0}^{\nb}
y_i\bigl(\log y_i-\log y_0\bigr).
\]
Figure~\ref{fig:new} illustrates how we interpret \OCNN as a regularized optimization layer. In this case, the regularizer captures the log-ratio between each product's purchase probability and the no-purchase probability. Unlike the regularization function for \PNN, which has received attention in the literature, $\Omega^{\mathrm{\OCNN}}$ is specifically designed for this problem. We next show that it is a valid regularization function.
\begin{proposition}[Properties of $\Omega^{\mathrm{\OCNN}}$]\label{prop:ocnn_regularizer_properties}
    $\Omega^{\mathrm{\OCNN}}(\cdot)$ is smooth and strictly convex on the relative interior of $\mathcal{C}(\boldsymbol I_t)$ and is of Legendre type.
\end{proposition}
These properties further allow us to derive the corresponding Fenchel conjugate in closed form.

\begin{proposition}[Convex conjugate of the \OCNN regularizer]
\label{prop:ocnn-conjugate} For any opportunity-cost vector $\hat{\boldsymbol{\Delta}}$, let  $\boldsymbol{\theta}$ be defined as in Equation~\eqref{eq:theta_opportunity_cost} with $\hat{\boldsymbol{\Delta}}$, and let $m_{\boldsymbol{\theta}}$ be the corresponding markup obtained from Equation~\eqref{eq:closed_form_mnl_markup}.
Then, the convex conjugate of the \OCNN regularizer is:
\[
\Omega^*(\boldsymbol{\theta})
=
\max_{\boldsymbol{y}\in \mathcal{C}(\boldsymbol I_t)}
\left\{
\boldsymbol{\theta}^{\top}\boldsymbol{y}
-
\Omega^{\mathrm{\OCNN}}(\boldsymbol{y})
\right\}
=
m_{\boldsymbol{\theta}}-1.
\]
\end{proposition}

Based on Proposition \ref{prop:ocnn-conjugate}, we now derive the corresponding loss and gradient expressions as follows.

\begin{proposition}[\OCNN Fenchel--Young loss and gradient]
\label{prop:ocnn-fy-loss-gradient}
    Using $\Omega^*(\boldsymbol{\theta})=m_{\boldsymbol{\theta}}-1$ from Proposition~\ref{prop:ocnn-conjugate}, the Fenchel--Young loss for \OCNN is
\[
    \mathcal{L}_{\Omega}^{\mathrm{FY}}
    (\boldsymbol{\theta},\bar{\boldsymbol y})
    =
    (m_{\boldsymbol{\theta}}-1)
    -
    \boldsymbol{\theta}^{\top}\bar{\boldsymbol y} + \Omega(\bar{\boldsymbol y}).
\]
    Moreover, the gradient of the Fenchel--Young loss with respect to the product score vector is
\[
    \frac{\partial \mathcal{L}^{\mathrm{FY}}_{\Omega}}{\partial \theta_i}
    =
    \hat y_{i}-\bar y_i,
    \qquad i\in \mathcal{A}(\boldsymbol I_t),
\]
where $\hat{\boldsymbol y}$ is an optimal solution to \eqref{eq:OCNNopt}.
\end{proposition}

The \PNN and \OCNN losses share the same gradient form, namely $\hat{\boldsymbol y}-\bar{\boldsymbol y}$, while they differ in the choice of regularizer and thus in the normalization term appearing in the loss. In particular, \PNN leads to the classical log-sum-exp term, whereas \OCNN replaces it with a markup-based term.

\subsection{Anticipative Customer Choice Generation} \label{subsec:anticipative scenario and masking}

We now turn to the construction of the target choices used to train the two DFL architectures. Ideally, we would train the policies using choices induced by the optimal dynamic pricing policy. Generating such labels, however, would require solving the original high-dimensional dynamic program and would therefore defeat the purpose of the learning approach. Instead, we use \emph{anticipative}, or perfect-information, choices. These are obtained by solving sampled demand scenarios under the assumption that all future uncertainty is revealed at the beginning of the selling horizon. Although  such a policy is not implementable in practice, policies with perfect hindsight are often used in RM to provide a perfect-information upper bound \citep{gallego1994optimal}, which means they contain some valuable information. Moreover, such an approach has been shown to be successful in other optimization-augmented policy learning approaches \citep{baty2024combinatorial, hoppe2026structured}.

We first describe perfect information when customer choice follows an MNL model. We then show that, conditional on a realization of the underlying utility shocks, the corresponding pricing problem reduces to a tractable problem. Finally, we explain why the same construction applies directly when scenarios are generated from a \MMNL model.

\subsubsection{Perfect information under the MNL model}
\label{sec:anticipative_mnl}

Recall the random-utility representation of the MNL model. At time $t$, the realized utility of an available product $i\in\mathcal A(\boldsymbol I_t)$ is
\begin{equation*}
U_{i,t}(r_{i,t}) = a_i-\beta r_{i,t}+\eta_{i,t},
\end{equation*}
where $\eta_{i,t}$ is an independent standard Gumbel shock. The utility of the no-purchase option is $U_{0,t}=\eta_{0,t}$. Thus, conditional on the shocks, customer choice is deterministic:
\begin{equation}\label{eq:UtilityFormulation}
Y_t \in \arg\max_{i\in\{0\}\cup\mathcal A(\boldsymbol I_t)} U_{i,t}(r_{i,t}).
\end{equation}
Integrating over the Gumbel shocks recovers the usual MNL choice probabilities in Equation~\eqref{eq:MNL-choice}. A perfect-information scenario corresponds to sampling the entire sequence of utility shocks at the beginning of the horizon. We denote such a scenario by
\begin{equation}\label{eq:shocksSequence}
{\boldsymbol{\eta}} = \left \{ \left( \eta_{0,t}, \eta_{1,t}, \ldots, \eta_{n,t} \right) \right\}_{t=1}^T .
\end{equation}
Under perfect information, the firm observes $\boldsymbol \eta$ before making any pricing decisions. Hence, it knows exactly how every customer over the selling horizon would respond to any price vector. The anticipative oracle chooses prices while exploiting this information, subject only to the inventory constraints.

This problem is substantially simpler than the original stochastic dynamic pricing problem. In particular, once the shocks are known, there is no remaining choice uncertainty: prices only determine which product, if any, is selected at each period. We exploit this observation next to eliminate the pricing variables altogether.

\subsubsection{Solving the perfect-information pricing problem}
\label{sec:anticipative_reduction}
Consider a fixed scenario $\boldsymbol \eta$ and a particular period $t$. Suppose that we wish to induce the customer to purchase product $i$. Because prices of all other available products can be made sufficiently large, the binding comparison is with the no-purchase option. Product $i$ can therefore be induced whenever
\begin{equation*}
a_i-\beta r_{i,t}+\eta_{i,t} \geq \eta_{0,t}.
\end{equation*}
Consequently, the largest price at which product $i$ can be sold in period $t$ is
\begin{equation}
\widetilde r_{i,t} = \frac{ a_i+\eta_{i,t}-\eta_{0,t}}{\beta}.
\label{eq:recovered_price}
\end{equation}
We use the convention that ties are broken in favor of the designated product. Equivalently, one may subtract an arbitrarily small $\varepsilon>0$ from the right-hand side of \eqref{eq:recovered_price} to ensure strict preference.

Because prices are constrained to be nonnegative, product $i$ can only be induced when $\widetilde r_{i,t}\geq 0$. Define
\begin{equation}
\mathcal B_t = \left \{ i\in\mathcal N: \widetilde r_{i,t}\geq 0 \right \}.
\end{equation}
The key implication is that, under perfect information, each feasible pair $(i,t)$ can be assigned a known reward $\widetilde r_{i,t}$: if product $i$ is selected in period $t$, this is the maximum revenue that can be extracted from that customer while inducing that choice. Prices therefore no longer need to appear explicitly in the anticipative optimization problem.

Let $y_{i,t}=1$ if the anticipative solution assigns the customer arriving at time $t$ to product $i$, and let $y_{0,t}=1$ denote no purchase. The perfect-information pricing problem is equivalent to
\begin{equation} \tag{\textsf anticipative pricing}
\label{eq:reduced_anticipative}
\begin{aligned}
\max_{\mathbf y}\quad & \sum_{t=1}^T \sum_{i\in\mathcal N} \widetilde r_{i,t}y_{i,t} \\
\text{s.t.}\quad &\sum_{t=1}^T y_{i,t}\leq c_i, i\in\mathcal N, \\
&\sum_{i\in\mathcal N\cup\{0\}}y_{i,t}=1,
 t=1,\ldots,T, \\
&y_{i,t}=0, i\in\mathcal N\setminus\mathcal B_t,\quad
t=1,\ldots,T, \\
&0\leq y_{i,t}\leq 1, i\in\mathcal N\cup\{0\},\quad t=1,\ldots,T.
\end{aligned}
\end{equation}

\begin{proposition}[Perfect-information reduction]
\label{prop:perfect_information_reduction}
For a fixed MNL scenario $\boldsymbol \eta$, Problem~\eqref{eq:reduced_anticipative} solves the perfect-information dynamic pricing problem. Moreover, it is a capacitated maximum-cost flow problem and its linear-programming relaxation admits an integral optimal solution.
\end{proposition}

Proposition~\ref{prop:perfect_information_reduction} is also the reason we use customer choices, rather than anticipative prices, as training targets. Once the selected product is fixed, prices of the unselected products are essentially arbitrary as long as they are sufficiently high. The choice trajectory, in contrast, precisely captures the relevant allocation of scarce inventory across customers.

For each sampled scenario $\boldsymbol \eta$, we therefore solve Problem~\eqref{eq:reduced_anticipative} and recover the associated sequence of states and choices $\left\{ (\boldsymbol x_t,\boldsymbol y_t) \right\}_{t=1}^T$. Repeating this procedure for $S$ independently sampled scenarios produces $ST$ state-choice pairs that constitute the training dataset for both PDFL and ODFL. It is worth noting that the above procedure can be directly extended to a time-varying MNL model by replacing the parameters $a_i$ and $\beta$ with their time-varying counterparts $a_{i,t}$ and $\beta_t$.

\subsubsection{Extension to \MMNL scenarios}
\label{sec:anticipative_mixmnl}

We now show that even if the model is \MMNL, the same oracle can be used. Since \MMNL can approximate any random utility model arbitrarily closely \citep{mcfadden2000mixed}, our approach is therefore without loss of generality. Suppose instead that demand is generated by a \MMNL model with segments $k\in\mathcal K$, mixture weights $\gamma_t^k$, product quality indices $a_i^k$, and price sensitivities $\beta^k$. Under a \MMNL model, a perfect-information scenario additionally includes the realization of the latent customer segment in each period. We therefore observe $(\boldsymbol{k},\boldsymbol{\eta})$, where $k_t \in \mathcal K$ denotes the segment of the customer arriving at time $t$. Conditional on $(\boldsymbol{k},\boldsymbol{\eta})$, the
realized utility is
\begin{equation}
U_{i,t}(r_{i,t}) = a_i^{k_t} - \beta^{k_t} r_{i,t} + \eta_{i,t}.
\end{equation}
Thus, after the latent segment and Gumbel shocks have been sampled, the problem is identical to the MNL perfect-information problem above, with period-specific parameters $a_{i,t}=a_i^{k_t}$ and $\beta_t=\beta^{k_t}$. In particular, the maximal price that induces product $i$ becomes
\begin{equation}
\widetilde r_{i,t} = \frac{a_i^{k_t} +\eta_{i,t} -\eta_{0,t}}{ \beta^{k_t}
},
\end{equation}
and the anticipative choice trajectory is obtained by solving exactly the same maximum-cost flow problem \eqref{eq:reduced_anticipative}. Hence, customer heterogeneity in the \MMNL model changes the sampled edge rewards but not the structure or computational complexity of the perfect-information oracle.

\section{Experimental Design}
\label{sec:experimental_design}

We evaluate the proposed policies in three increasingly demanding settings. The first assumes that the ground-truth choice model is MNL and fully known. This setting isolates the optimization problem and asks whether our learning approach can recover a high-quality dynamic pricing policy when the MNL structure embedded in the architectures is correctly specified. The second assumes that the ground truth is a known \MMNL model. In this case, although we can generate training scenarios directly from the true demand model, we must approximate it by an MNL model for our PDFL and ODFL. This setting tests the robustness of the MNL-guided architectures to model misspecification. Finally, in the data-driven setting, the demand model itself is unknown and must first be estimated from observed prices and customer choices.

We describe these three settings in Section~\ref{sec:settings_instances}. We then describe the baseline policies in Section~\ref{sec:baselines}, and the implementation of \PNN and \OCNN in Section~\ref{subsec:configurations}.

\subsection{Experimental Settings}
\label{sec:settings_instances}

\subsubsection{Full-information MNL: isolating the optimization-stage contribution}
\label{sec:fi_mnl}

\paragraph{Information setting.}
We first consider the case in which the ground-truth demand model is an MNL model and all of its parameters are known. At time $t$, customer choice probabilities are
\begin{equation}
\label{eq:time-dependent-mnl}
\mathbb{P}_t(i\mid \boldsymbol r_t,\boldsymbol I_t) = \frac{\exp(a_{i,t}-\beta_t r_{i,t})
}{1+ \sum_{j\in\mathcal A(\boldsymbol I_t)} \exp(a_{j,t}-\beta_t r_{j,t})}.
\end{equation}
Because the true model belongs to the same MNL family used by our architectures, no choice-model approximation is necessary. The MNL parameters used by \PNN and \OCNN are therefore the true parameters, $\mathbb{Q}=\mathbb{P}$, and the same model is used to generate the anticipative training scenarios described in Section~\ref{sec:anticipative_mnl}.

This experiment removes both demand-estimation error and model misspecification. It therefore addresses a simple question: assuming that the choice model is known and correctly specified, can the proposed learning approach recover the dynamic pricing policy without solving the full dynamic program?

\paragraph{Problem instances.}
We first consider three small instances with $n=3$ products, a selling horizon of $T=50$, and time-independent parameters, adapted from \cite{dong2009dynamic}. The instances differ in their initial inventories, $C\in\{5,10,15\}$, while sharing the same choice-model parameters. Their state spaces are sufficiently small that we can solve the original dynamic program exactly. These instances therefore allow us to measure the optimality gap of PDFL and ODFL directly.

We then generate larger-scale MNL instances designed to reflect airline revenue management settings with multiple departure options serving the same origin-destination market. Each instance is characterized by the number of itineraries $n$, initial capacity $C$, and selling horizon $T$. As summarized in Table~\ref{tab:instances}, we consider several to a few dozen itineraries, each with tens to a few hundred sellable seats. We also include a few stressed instances with many itineraries and low capacity to test robustness under scarce inventory. For controlled comparisons, all itineraries within an instance have the same initial capacity, although the framework readily accommodates heterogeneous capacities. We further characterize capacity scarcity by the ratio of total initial inventory to expected total demand (denoted by $\rho$ in Table~\ref{tab:instances}), with smaller values indicating tighter capacity. Details on the calibration of the MNL instance parameters are provided in Appendix~\ref{app:instance_parameters}.

\begin{table}[t]
\centering
\caption{Summary of all instances. Instances are named following the pattern $(\nb,C,T)$}
\label{tab:instances}

\small
\begin{threeparttable}
\centering

\begin{tabular}{@{}lccc*{10}{c}@{}}
\toprule
& \multicolumn{3}{c}{Small MNL}
& \multicolumn{10}{c}{Large-scale MNL} \\
\cmidrule(lr){2-4}
\cmidrule(lr){5-14}

Instance
& $S_1$ & $S_2$ & $S_3$
& $L_1$ & $L_2$ & $L_3$ & $L_4$ & $L_5$
& $L_6$ & $L_7$ & $L_8$ & $L_9$ & $L_{10}$ \\
\midrule

Number of itineraries $(\nb)$
& 3 & 3 & 3
& 21 & 12 & 21 & 6 & 15 & 24 & 15 & 12 & 30 & 6 \\

Initial capacity $(C)$
& 5 & 10 & 15
& 30 & 90 & 30 & 210 & 30 & 30 & 120 & 180 & 120 & 150 \\

Number of time steps $(T)$
& 50 & 50 & 50
& 1400 & 1800 & 800 & 1600 & 600 & 600 & 1400 & 2000 & 1400 & 400 \\

Inventory-demand ratio $(\rho)$
& 0.30 & 0.60 & 0.90
& 0.48 & 0.64 & 0.90 & 0.84 & 0.90 & 1.44 & 1.38 & 1.13 & 2.80 & 3.00 \\

\bottomrule
\end{tabular}
\end{threeparttable}
\end{table}

\subsubsection{Full-information \MMNL: robustness to model misspecification}
\label{sec:fi_mixmnl}

\paragraph{Information setting.} We next consider a richer demand environment in which customer choice follows a \MMNL model with segments $k\in\mathcal K$:
\begin{equation}
\mathbb{P}_t(i\mid\boldsymbol r_t,\boldsymbol I_t) = \sum_{k\in\mathcal K}
\gamma_t^k \frac{\exp(a_i^k-\beta^k r_{i,t})}{1+ \sum_{j\in\mathcal A(\boldsymbol I_t)} \exp(a_j^k-\beta^k r_{j,t})}.
\label{eq:mixmnl_exp}
\end{equation}
The complete \MMNL model is known in this setting. We therefore sample the latent segment and utility shocks directly from the ground-truth model when generating the anticipative training data. As shown in Section~\ref{sec:anticipative_mixmnl}, conditional on these realizations the perfect-information oracle reduces to the same maximum-cost-flow problem as in the MNL case.

The PDFL and ODFL architectures, however, deliberately retain an MNL structural layer. We therefore construct an MNL surrogate $\mathbb{Q}$ that approximates the true \MMNL model $\mathbb{P}$:
$$
    \mathbb{Q}=\Pi_{\mathrm{MNL}}(\mathbb{P}),
$$
where the projection operator $\Pi_{\mathrm{MNL}}$ is defined through KL-divergence minimization in Appendix~\ref{appendix:kl_projection}. Hence, the anticipative labels are generated from the true heterogeneous demand model $\mathbb{P}$, while the policy architectures only have access to its MNL approximation $\mathbb{Q}$.

This setting isolates the effect of structural misspecification. In particular, it asks whether MNL-guided policy learning remains useful when the true substitution patterns and price sensitivities are substantially richer than those represented by a single MNL model.

\paragraph{Problem instances.} We reuse the large-scale instances for MNL in Table~\ref{tab:instances}, but assume a \MMNL distribution constructed to reflect a more practical airline RM setting.
To be consistent with airline RM practice, we assume four passenger segments: Business-Hurry (BH), Business (B), Leisure-Family (LF), and Leisure (L). L- and LF-type customers are concentrated earlier in the horizon, whereas B-type and especially BH-type customers become more prevalent closer to departure. Figure~\ref{fig:segment_proportion} illustrates this pattern for an instance with six itineraries, 150 seats per itinerary, and a 400-period selling horizon.
Segment-specific parameters $a_i^k$ and $\beta^k$ introduce heterogeneity in itinerary preferences and willingness to pay.

\begin{figure}[htbp]
    \centering
    \resizebox{0.55\textwidth}{!}{%
        
\begingroup
\definecolor{segLeisure}{HTML}{66C2A5}
\definecolor{segFamily}{HTML}{FC8D62}
\definecolor{segBusiness}{HTML}{E78AC3}
\definecolor{segHurry}{HTML}{8DA0CB}

\begin{tikzpicture}[
    declare function={
        leisure(\t)=0.1+0.9*exp(-0.5*((\t-400*20/150)/(400*40/150))^2);
        family(\t)=0.1+0.9*exp(-0.5*((\t-400*100/150)/(400*50/150))^2);
        business(\t)=0.1+0.9*exp(-0.5*((\t-480)/(400*50/150))^2);
        hurry(\t)=0.1+0.9*exp(-0.5*((\t-480)/80)^2);
        total(\t)=leisure(\t)+family(\t)+business(\t)+hurry(\t);
    }
]
\begin{axis}[
    width=\linewidth,
    height=0.62\linewidth,
    font=\small,
    title={
        Customer Segment Proportions Over Time\\
        (Instance: $n=6$, $C=150$, $T=400$)
    },
    title style={
        align=center,
        font=\bfseries,
    },
    xlabel={Time Step},
    ylabel={Proportion},
    xmin=0, xmax=400,
    ymin=0, ymax=1,
    xtick={0,50,...,400},
    ytick={0,0.25,0.5,0.75,1},
    yticklabel style={
        /pgf/number format/fixed,
        /pgf/number format/precision=2,
        /pgf/number format/fixed zerofill,
    },
    scaled y ticks=false,
    tick align=outside,
    tick style={black},
    axis line style={black},
    axis on top,
    ymajorgrids=true,
    grid style={gray!30,dashed},
    stack plots=y,
    area style,
    domain=1:400,
    samples=200,
    reverse legend,
    legend pos=north east,
    legend cell align=left,
    legend style={
        fill=white,
        fill opacity=0.9,
        text opacity=1,
        draw=gray!40,
        rounded corners=2pt,
        font=\small,
    },
]

\addplot[draw=none,fill=segHurry!88!white]
    {hurry(x)/total(x)} \closedcycle;
\addlegendentry{Business-Hurry}

\addplot[draw=none,fill=segBusiness!88!white]
    {business(x)/total(x)} \closedcycle;
\addlegendentry{Business}

\addplot[draw=none,fill=segFamily!88!white]
    {family(x)/total(x)} \closedcycle;
\addlegendentry{Leisure-Family}

\addplot[draw=none,fill=segLeisure!88!white]
    {leisure(x)/total(x)} \closedcycle;
\addlegendentry{Leisure}

\end{axis}
\end{tikzpicture}
\endgroup

    }
    \caption{\centering Segment proportions over time for instance (6,150,400).}
    \label{fig:segment_proportion}
\end{figure}
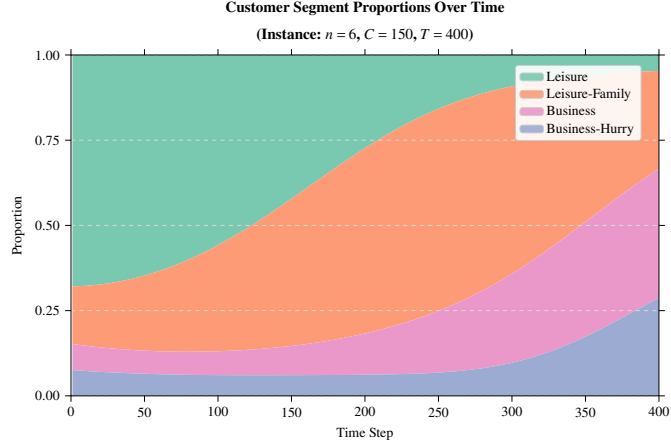

To reproduce the characteristic evolution of the passenger mix over an airline booking horizon, the segment weights $\gamma_t^k$ vary systematically with time. Specifically, we construct the four segment weights using smooth Gaussian-shaped curves whose peaks occur at different stages of the selling horizon. At each time step, the four segment scores are normalized to obtain the segment proportions. Finally, to capture heterogeneity in willingness to pay and itinerary preferences, we independently sample $\beta^k$ and $a_i^k$ from segment-specific distributions.

\subsubsection{Data-driven setting: learning from observed choices}
\label{sec:data_driven_setting}

\paragraph{Information setting.} Finally, we consider the realistic setting in which the decision maker does not observe the parameters of the demand model. Instead, we generate historical observations of posted prices and realized customer choices from the ground-truth environment $\mathbb{P}$, and make only these observations available to the policy-construction procedure.

We reuse the large-scale instances in Table~\ref{tab:instances} from the full-information setting, considering cases in which the true environment $\mathbb{P}$ is either an MNL or a \MMNL model. Historical observations are generated by sampling from $\mathbb{P}$, and the surrogate MNL model $\mathbb{Q}$ is estimated by maximizing the likelihood of the observed data. In the data-driven setting, $\mathbb{Q}$ is also used to generate the training data for our architectures, as described in Section~\ref{subsec:anticipative scenario and masking}.

This final setting therefore evaluates the complete pipeline and combines the two potential sources of performance loss studied separately above: statistical error from estimating the demand model and structural approximation error from representing that estimated model through the MNL layer embedded in \PNN and \OCNN.

\subsection{Baseline Policies}
\label{sec:baselines}

We compare \PNN and \OCNN with three families of tractable revenue-management policies. The first prices each itinerary independently and therefore ignores substitution across products. The other two use the same single-itinerary dynamic programs to approximate inventory opportunity costs, but incorporate these estimates into an MNL joint-pricing rule. For each family, we consider both a time-independent and a time-dependent specification, yielding six baseline policies in total (denoted by \TimeIndItinerary, \TimeDepItinerary, \TimeIndApp, \TimeDepApp, \TimeIndUnified, \TimeDepUnified). Detailed formulations are provided in Appendix~\ref{app:benchmark_policies}.

\paragraph{Independent-itinerary pricing \textbf{ItPri}.}
Our first baseline follows a common approximation in airline revenue management and decomposes the problem by itinerary. For each product $i$, we construct a binary choice model between purchasing itinerary $i$ and not purchasing it, and solve the resulting single-product dynamic program independently of the other itineraries. The resulting policy is computationally tractable and accounts for the dynamic scarcity of each product's inventory, but ignores substitution across products.

\paragraph{Joint pricing with product-specific opportunity costs \textbf{JoPri}.}
The second baseline reintroduces substitution across products while retaining the tractability of the single-itinerary decomposition. Specifically, we use the opportunity costs obtained from the independent dynamic programs as approximations of the true marginal continuation values in the MNL statewise pricing problem. Thus, unlike \textbf{ItPri}, \textbf{JoPri} prices all available products jointly and explicitly accounts for substitution. Its approximation lies in replacing the true high-dimensional inventory opportunity costs by those obtained from the independent single-itinerary dynamic programs.

\paragraph{Joint common pricing with an aggregated opportunity cost \textbf{JoComPri}.}
Our third baseline imposes an additional level of aggregation. Rather than retaining a separate opportunity cost for every product, \textbf{JoComPri} combines the single-itinerary opportunity costs into a single attractiveness-weighted inventory value.
As every product shares the same opportunity cost, \textbf{JoComPri} assigns a common price to all available products at each time step.

\subsection{Architecture Configurations and Model Selection}
\label{subsec:configurations}

Building on the baselines introduced above, we construct RM-structured features that provide informative price and opportunity-cost signals for our learning architectures. We then use these features to define several architecture configurations, which are trained and selected based on validation performance. Finally, we assess the contribution of the residual parameterization by comparing it with a direct-prediction counterpart.

\paragraph{Feature engineering.}
To map the problem state to decision-relevant inputs, we construct RM-structured features beyond raw state variables such as inventory. These features include price- and opportunity-cost-based signals, such as myopic prices and outputs from the DP-based baselines, as well as features capturing inventory and time, customer-choice behavior, marginal scarcity, and cross-product competition. Together, they provide interpretable signals of demand, inventory value, and substitution for the \PNN and \OCNN. Detailed descriptions are provided in Appendix~\ref{app:feature_engineering}.

\paragraph{Architecture configurations.}
We consider several configurations of the statistical predictor $\varphi_w$ as well as its training objective. We first compare direct prediction with residual parameterizations. Table~\ref{tab:pipeline_configurations} summarizes the candidate configurations, and details are provided in Appendix \ref{app:model_selection}.
Following standard train-validation-test practice, we report test performance on the true environment $\mathbb{P}$ for the configurations with the best validation revenue. In the full-information setting, validation is performed directly on $\mathbb{P}$. In the data-driven setting, where $\mathbb{P}$ is unavailable, we instead use the surrogate MNL model $\mathbb{Q}$ as the validation environment. For a fair comparison, we use the validation-selected baseline as the primary benchmark, while reporting the testing revenue of all six baselines.

\begin{table}[htbp]
\centering
\caption{Candidate configurations for the \PNN and \OCNN pipelines.}
\label{tab:pipeline_configurations}
\small
\begin{tabularx}{\textwidth}{lX}
\toprule
\textbf{Configuration component} & \textbf{Candidate values} \\
\midrule
Model parameterization
& Direct, additive residual, multiplicative residual \\
Reference policy
& Mean baseline output, validation-selected baseline output \\
Residual magnitude
& $K_{\mathrm{add}}\in\{1,5\}$,
  $K_{\mathrm{mul}}\in\{0.25,0.7\}$ \\
Additional loss term (\PNN only)
& None, chosen-item squared hinge penalty \\
\bottomrule
\end{tabularx}
\end{table}

\paragraph{Benchmarking architectures.}
We compare direct prediction with residual parameterizations for both the \PNN and \OCNN. Direct prediction models (``direct'' in Figure \ref{fig:residual_comparison}) predict prices or opportunity costs from scratch, whereas residual models learn corrections to a reference RM policy, as introduced in Section~\ref{subsec:statistical predictor}.
    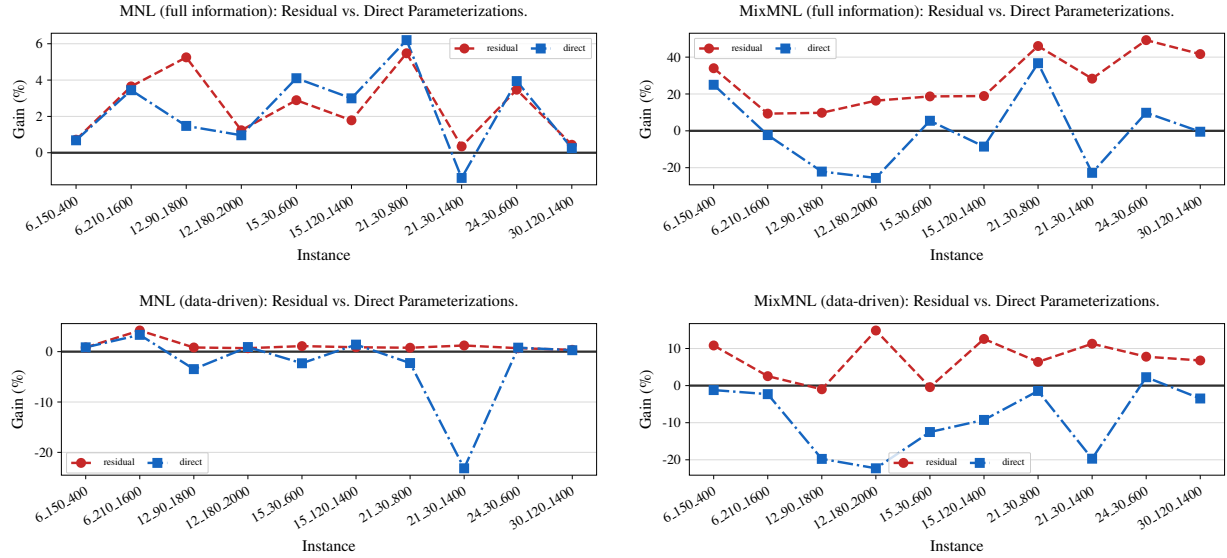
\begin{figure}[htbp]
    \centering
    \resizebox{\textwidth}{!}{%
        
\begin{tikzpicture}
\begin{scope}[shift={(0bp,250.8bp)}]
\begin{axis}[
  at={(47.7352bp,80.4058bp)},
  anchor=south west,
  scale only axis,
  width=471.025bp,
  height=130.922bp,
  xmin=-0.45000000000000001,
  xmax=9.4499999999999993,
  ymin=-1.7654418928801729,
  ymax=6.5782786316052793,
  enlargelimits=false,
  axis lines=box,
  axis line style={line width=0.8bp},
  tick align=outside,
  tick pos=left,
  major tick length=3.5bp,
  tick style={black,line width=0.8bp},
  scaled ticks=false,
  xtick={0,1,2,3,4,5,6,7,8,9},
  xticklabels={{6\_150\_400},{6\_210\_1600},{12\_90\_1800},{12\_180\_2000},{15\_30\_600},{15\_120\_1400},{21\_30\_800},{21\_30\_1400},{24\_30\_600},{30\_120\_1400}},
  ytick={0,2,4,6},
  yticklabels={{0},{2},{4},{6}},
  xticklabel style={font={\fontsize{12}{14.4}\selectfont},rotate=30,anchor=north east},
  yticklabel style={font={\fontsize{12}{14.4}\selectfont}},
  xlabel={Instance},
  ylabel={Gain (\%)},
  label style={font={\fontsize{14}{16.8}\selectfont}},
  title={MNL (full information): Residual vs. Direct Parameterizations.},
  title style={font={\fontsize{14}{16.8}\selectfont},at={(0.5,1)},anchor=south,yshift=3bp},
  ymajorgrids=true,
  xmajorgrids=false,
  grid style={gray!30,line width=0.8bp},
  axis background/.style={fill=white},
  unbounded coords=jump,
  legend columns=2,
  legend style={at={(0.990977,0.967538)},anchor=north east,font={\fontsize{8.5}{10.2}\selectfont},draw=black!20,fill=white,fill opacity=0.8,text opacity=1,rounded corners=1bp,inner sep=3bp,column sep=8bp,/tikz/every even column/.append style={column sep=8bp}},
  legend cell align=left
]
\addplot[color={rgb,1:red,0.066666666666666666;green,0.066666666666666666;blue,0.066666666666666666},line width=2bp,mark=none,opacity=0.85,forget plot] coordinates {(-0.45000000000000001,0) (9.4499999999999993,0)};
\addplot[color={rgb,1:red,0.77647058823529413;green,0.15686274509803921;blue,0.15686274509803921},line width=2.1bp,dash pattern=on 7.77bp off 3.36bp,dash phase=0bp,mark=*,mark size=3.25bp,mark options={solid},opacity=1] coordinates {
  (0,0.71790824486209359)
  (1,3.6518691317687337)
  (2,5.250959629562991)
  (3,1.237413268825595)
  (4,2.8922871774251644)
  (5,1.7852252423910258)
  (6,5.475781494635795)
  (7,0.34791870300651168)
  (8,3.4710353850600839)
  (9,0.43874715244628087)
};
\addlegendentry{residual}
\addplot[color={rgb,1:red,0.082352941176470587;green,0.396078431372549;blue,0.75294117647058822},line width=2.1bp,dash pattern=on 13.44bp off 3.36bp on 2.1bp off 3.36bp,dash phase=0bp,mark=square*,mark size=3.25bp,mark options={solid},opacity=1] coordinates {
  (0,0.68466815661294733)
  (1,3.4428217923038167)
  (2,1.4770858994585094)
  (3,0.96210263191407397)
  (4,4.1020191303759832)
  (5,2.995499763726289)
  (6,6.1990186077650318)
  (7,-1.386181869039925)
  (8,3.9450219584258588)
  (9,0.25309499975036109)
};
\addlegendentry{direct}
\end{axis}
\end{scope}
\begin{scope}[shift={(542.4bp,250.8bp)}]
\begin{axis}[
  at={(56.4725bp,80.4058bp)},
  anchor=south west,
  scale only axis,
  width=462.287bp,
  height=130.922bp,
  xmin=-0.45000000000000001,
  xmax=9.4499999999999993,
  ymin=-29.29524313321965,
  ymax=52.962484373223333,
  enlargelimits=false,
  axis lines=box,
  axis line style={line width=0.8bp},
  tick align=outside,
  tick pos=left,
  major tick length=3.5bp,
  tick style={black,line width=0.8bp},
  scaled ticks=false,
  xtick={0,1,2,3,4,5,6,7,8,9},
  xticklabels={{6\_150\_400},{6\_210\_1600},{12\_90\_1800},{12\_180\_2000},{15\_30\_600},{15\_120\_1400},{21\_30\_800},{21\_30\_1400},{24\_30\_600},{30\_120\_1400}},
  ytick={-20,0,20,40},
  yticklabels={{-20},{0},{20},{40}},
  xticklabel style={font={\fontsize{12}{14.4}\selectfont},rotate=30,anchor=north east},
  yticklabel style={font={\fontsize{12}{14.4}\selectfont}},
  xlabel={Instance},
  ylabel={Gain (\%)},
  label style={font={\fontsize{14}{16.8}\selectfont}},
  title={\MMNL (full information): Residual vs. Direct Parameterizations.},
  title style={font={\fontsize{14}{16.8}\selectfont},at={(0.5,1)},anchor=south,yshift=3bp},
  ymajorgrids=true,
  xmajorgrids=false,
  grid style={gray!30,line width=0.8bp},
  axis background/.style={fill=white},
  unbounded coords=jump,
  legend columns=2,
  legend style={at={(0.00919341,0.967538)},anchor=north west,font={\fontsize{8.5}{10.2}\selectfont},draw=black!20,fill=white,fill opacity=0.8,text opacity=1,rounded corners=1bp,inner sep=3bp,column sep=8bp,/tikz/every even column/.append style={column sep=8bp}},
  legend cell align=left
]
\addplot[color={rgb,1:red,0.066666666666666666;green,0.066666666666666666;blue,0.066666666666666666},line width=2bp,mark=none,opacity=0.85,forget plot] coordinates {(-0.45000000000000001,0) (9.4499999999999993,0)};
\addplot[color={rgb,1:red,0.77647058823529413;green,0.15686274509803921;blue,0.15686274509803921},line width=2.1bp,dash pattern=on 7.77bp off 3.36bp,dash phase=0bp,mark=*,mark size=3.25bp,mark options={solid},opacity=1] coordinates {
  (0,33.977288929111957)
  (1,9.3357659071340215)
  (2,9.7905410933458086)
  (3,16.379393440822561)
  (4,18.680816144572635)
  (5,18.879668691024712)
  (6,46.026262613531287)
  (7,28.313748398003948)
  (8,49.223496759294108)
  (9,41.65358936670907)
};
\addlegendentry{residual}
\addplot[color={rgb,1:red,0.082352941176470587;green,0.396078431372549;blue,0.75294117647058822},line width=2.1bp,dash pattern=on 13.44bp off 3.36bp on 2.1bp off 3.36bp,dash phase=0bp,mark=square*,mark size=3.25bp,mark options={solid},opacity=1] coordinates {
  (0,24.988002813642595)
  (1,-2.3008134410614951)
  (2,-22.111806547320221)
  (3,-25.556255519290424)
  (4,5.458225272972963)
  (5,-8.5491937283543145)
  (6,36.704270485230751)
  (7,-22.786033607415568)
  (8,9.8027113220320956)
  (9,-0.48240884891170399)
};
\addlegendentry{direct}
\end{axis}
\end{scope}
\begin{scope}[shift={(0bp,0bp)}]
\begin{axis}[
  at={(56.4725bp,80.4058bp)},
  anchor=south west,
  scale only axis,
  width=462.287bp,
  height=130.922bp,
  xmin=-0.45000000000000001,
  xmax=9.4499999999999993,
  ymin=-24.5314606270009,
  ymax=5.5688941848591043,
  enlargelimits=false,
  axis lines=box,
  axis line style={line width=0.8bp},
  tick align=outside,
  tick pos=left,
  major tick length=3.5bp,
  tick style={black,line width=0.8bp},
  scaled ticks=false,
  xtick={0,1,2,3,4,5,6,7,8,9},
  xticklabels={{6\_150\_400},{6\_210\_1600},{12\_90\_1800},{12\_180\_2000},{15\_30\_600},{15\_120\_1400},{21\_30\_800},{21\_30\_1400},{24\_30\_600},{30\_120\_1400}},
  ytick={-20,-10,0},
  yticklabels={{-20},{-10},{0}},
  xticklabel style={font={\fontsize{12}{14.4}\selectfont},rotate=30,anchor=north east},
  yticklabel style={font={\fontsize{12}{14.4}\selectfont}},
  xlabel={Instance},
  ylabel={Gain (\%)},
  label style={font={\fontsize{14}{16.8}\selectfont}},
  title={MNL (data-driven): Residual vs. Direct Parameterizations.},
  title style={font={\fontsize{14}{16.8}\selectfont},at={(0.5,1)},anchor=south,yshift=3bp},
  ymajorgrids=true,
  xmajorgrids=false,
  grid style={gray!30,line width=0.8bp},
  axis background/.style={fill=white},
  unbounded coords=jump,
  legend columns=2,
  legend style={at={(0.00919341,0.149192)},anchor=north west,font={\fontsize{8.5}{10.2}\selectfont},draw=black!20,fill=white,fill opacity=0.8,text opacity=1,rounded corners=1bp,inner sep=3bp,column sep=8bp,/tikz/every even column/.append style={column sep=8bp}},
  legend cell align=left
]
\addplot[color={rgb,1:red,0.066666666666666666;green,0.066666666666666666;blue,0.066666666666666666},line width=2bp,mark=none,opacity=0.85,forget plot] coordinates {(-0.45000000000000001,0) (9.4499999999999993,0)};
\addplot[color={rgb,1:red,0.77647058823529413;green,0.15686274509803921;blue,0.15686274509803921},line width=2.1bp,dash pattern=on 7.77bp off 3.36bp,dash phase=0bp,mark=*,mark size=3.25bp,mark options={solid},opacity=1] coordinates {
  (0,0.83037647710698692)
  (1,4.2006962388654676)
  (2,0.82566651777045108)
  (3,0.68880742916839188)
  (4,1.1003942395513839)
  (5,0.87151917193581674)
  (6,0.77340993682643622)
  (7,1.2230984777418308)
  (8,0.68760429797854439)
  (9,0.36242827922492837)
};
\addlegendentry{residual}
\addplot[color={rgb,1:red,0.082352941176470587;green,0.396078431372549;blue,0.75294117647058822},line width=2.1bp,dash pattern=on 13.44bp off 3.36bp on 2.1bp off 3.36bp,dash phase=0bp,mark=square*,mark size=3.25bp,mark options={solid},opacity=1] coordinates {
  (0,0.86092059863964387)
  (1,3.3343461860400883)
  (2,-3.4697316366676239)
  (3,0.91474631883275037)
  (4,-2.3074663021393369)
  (5,1.3972523258011755)
  (6,-2.2636404213622936)
  (7,-23.163262681007264)
  (8,0.79532808438405478)
  (9,0.28532677426549941)
};
\addlegendentry{direct}
\end{axis}
\end{scope}
\begin{scope}[shift={(542.4bp,0bp)}]
\begin{axis}[
  at={(56.4725bp,80.4058bp)},
  anchor=south west,
  scale only axis,
  width=462.287bp,
  height=130.922bp,
  xmin=-0.45000000000000001,
  xmax=9.4499999999999993,
  ymin=-24.1459910401108,
  ymax=16.710051629328898,
  enlargelimits=false,
  axis lines=box,
  axis line style={line width=0.8bp},
  tick align=outside,
  tick pos=left,
  major tick length=3.5bp,
  tick style={black,line width=0.8bp},
  scaled ticks=false,
  xtick={0,1,2,3,4,5,6,7,8,9},
  xticklabels={{6\_150\_400},{6\_210\_1600},{12\_90\_1800},{12\_180\_2000},{15\_30\_600},{15\_120\_1400},{21\_30\_800},{21\_30\_1400},{24\_30\_600},{30\_120\_1400}},
  ytick={-20,-10,0,10},
  yticklabels={{-20},{-10},{0},{10}},
  xticklabel style={font={\fontsize{12}{14.4}\selectfont},rotate=30,anchor=north east},
  yticklabel style={font={\fontsize{12}{14.4}\selectfont}},
  xlabel={Instance},
  ylabel={Gain (\%)},
  label style={font={\fontsize{14}{16.8}\selectfont}},
  title={\MMNL (data-driven): Residual vs. Direct Parameterizations.},
  title style={font={\fontsize{14}{16.8}\selectfont},at={(0.5,1)},anchor=south,yshift=3bp},
  ymajorgrids=true,
  xmajorgrids=false,
  grid style={gray!30,line width=0.8bp},
  axis background/.style={fill=white},
  unbounded coords=jump,
  legend columns=2,
  legend style={at={(0.37245,0.149192)},anchor=north west,font={\fontsize{8.5}{10.2}\selectfont},draw=black!20,fill=white,fill opacity=0.8,text opacity=1,rounded corners=1bp,inner sep=3bp,column sep=8bp,/tikz/every even column/.append style={column sep=8bp}},
  legend cell align=left
]
\addplot[color={rgb,1:red,0.066666666666666666;green,0.066666666666666666;blue,0.066666666666666666},line width=2bp,mark=none,opacity=0.85,forget plot] coordinates {(-0.45000000000000001,0) (9.4499999999999993,0)};
\addplot[color={rgb,1:red,0.77647058823529413;green,0.15686274509803921;blue,0.15686274509803921},line width=2.1bp,dash pattern=on 7.77bp off 3.36bp,dash phase=0bp,mark=*,mark size=3.25bp,mark options={solid},opacity=1] coordinates {
  (0,10.816237355800071)
  (1,2.5460890862890806)
  (2,-0.98590612194628002)
  (3,14.852958780718003)
  (4,-0.41108074367543651)
  (5,12.574022311785622)
  (6,6.3846392231433677)
  (7,11.280265041297694)
  (8,7.7903829360329846)
  (9,6.7755930567049063)
};
\addlegendentry{residual}
\addplot[color={rgb,1:red,0.082352941176470587;green,0.396078431372549;blue,0.75294117647058822},line width=2.1bp,dash pattern=on 13.44bp off 3.36bp on 2.1bp off 3.36bp,dash phase=0bp,mark=square*,mark size=3.25bp,mark options={solid},opacity=1] coordinates {
  (0,-1.2062034585497097)
  (1,-2.300893547237211)
  (2,-19.75726694635506)
  (3,-22.288898191499904)
  (4,-12.515516669111808)
  (5,-9.2491647228081693)
  (6,-1.4360849927481398)
  (7,-19.699942412051769)
  (8,2.2608041802496506)
  (9,-3.473550823609004)
};
\addlegendentry{direct}
\end{axis}
\end{scope}
\pgfresetboundingbox
\path[use as bounding box] (0,0) rectangle (1075.2bp,492bp);
\end{tikzpicture}

    }
\caption{Validation performance of residual and direct parameterizations.
Each point reports the revenue gain relative to the validation-selected baseline (normalized to zero).
For each model class, the displayed result corresponds to its best
validation configuration.}
\label{fig:residual_comparison}
\end{figure}

Figure~\ref{fig:residual_comparison} compares the best residual and direct specifications within each experimental setting. Residual learning provides the clearest gains in the \MMNL environments, where direct prediction is often unstable and can underperform the baseline. The difference is smaller in the MNL environments, although residual models remain more robust across instances. We therefore keep residual parameterizations for the main experiments. Detailed comparisons of the reference policy and residual form (additive, multiplicative) are reported in Appendix~\ref{app:model_selection}.

\paragraph{Computational environment.}
All experiments were run on a workstation with Ubuntu 24.04.1, equipped with an Intel Core i9-14900K CPU, 16 GB of memory, and an NVIDIA GeForce RTX 4070 GPU with 12 GB of memory.
The computational environment used Python 3.12.9 and Julia 1.10.7.

\paragraph{Implementation details.}
Our \PNN and \OCNN are trained on anticipative samples generated from 100 trajectories, as described in Section \ref{subsec:anticipative scenario and masking}. During validation, we select the best-performing DFL configuration and the best-performing baseline policy for each instance based on average revenue over 30 sampled trajectories. Then, the selected architecture configurations are frozen before testing.
The frozen configurations and baselines are evaluated on 100 independent test trajectories generated from the ground-truth demand environment $\mathbb{P}$.

\section{Numerical Results} \label{sec:numerical_results}
In this section, we evaluate the performance of the developed  \PNN and \OCNN policies under both full-information and data-driven settings.
We report the testing results of our DFL architectures and baseline policies in the full-information setting (Section~\ref{subsec:full_information_results}) and the data-driven setting (Section~\ref{subsec:data_driven_results}) under MNL and \MMNL environments.

We compare our proposed policies with several baseline policies (as presented in Section \ref{sec:baselines}). For full-information small MNL instances, we consider the optimal dynamic pricing policy derived by DP \citep{dong2009dynamic} and report optimality gaps. The performance of a policy is measured by the relative improvement over the validation-selected baseline policy:
\[
\mathrm{Gain}(\%)
=
\frac{\mathrm{Rev}^{p}-\mathrm{Rev}^{b}}
{\mathrm{Rev}^{b}}
\times 100\%,
\]
where $\mathrm{Rev}^{p}$ is the mean test revenue of the evaluated policy and $\mathrm{Rev}^{b}$ is that of the validation-selected baseline, both computed over sampled trajectories. A positive value indicates that the considered policy outperforms the validation-selected baseline.

\subsection{Full-Information Setting} \label{subsec:full_information_results}
\subsubsection{On small tractable MNL instances} \label{subsubsec:on small tractable MNL instances}
We first consider small tractable MNL instances, for which the optimal pricing policy can be computed exactly by DP. This allows us to use the optimal expected revenue as an exact baseline for evaluating the proposed policies.
Table \ref{tab:small_mnl_exact_dp_summary} summarizes results on small MNL instances, where green highlights the smallest optimality gap in each row. The results show that our decision-focused policies outperform other DP-based baselines, with an average optimality gap of 0.4\% (for \PNN) and 0.3\% (for \OCNN). Among the baseline methods, \TimeIndUnified and \TimeDepUnified exhibit the worst performance. This is expected because they impose a common price across all items at each time step and therefore cannot exploit item-level heterogeneity, such as differences in product attractiveness, inventory scarcity, and opportunity costs. Moreover, across all baseline families, the time-dependent variants generally outperform their time-independent counterparts. This improvement is intuitive, as time-dependent baseline model parameters can better capture changes in demand conditions and inventory value over the selling horizon.

\definecolor{evalgray}{gray}{0.95}

\begin{table}[htbp]
\centering
\caption{Testing results on small MNL instances, reported as optimality gaps relative to the DP optimal value, computed as $(\mathrm{Opt}-\mathrm{Policy})/\mathrm{Opt}\times100\%$. \textcolor{posgreen}{\textbf{Green}} indicates the best result.}
\label{tab:small_mnl_exact_dp_summary}

\begin{threeparttable}
\small
\begin{tabular}{>{\bfseries}l c c c c c c c c}
\toprule
Instance
& \TimeIndItinerary
& \TimeDepItinerary
& \TimeIndApp
& \TimeDepApp
& \TimeIndUnified
& \TimeDepUnified
& \PNN
& \OCNN \\
\midrule

(3,5,50)
& 5.7\%
& 5.9\%
& 6.2\%
& 3.6\%
& 23.4\%
& 10.1\%
& \textcolor{posgreen}{\textbf{0.2\%}}
& \textcolor{posgreen}{\textbf{0.2\%}} \\

(3,10,50)
& 7.3\%
& 6.7\%
& 10.2\%
& 2.6\%
& 27.5\%
& 13.1\%
& 0.6\%
& \textcolor{posgreen}{\textbf{0.4\%}} \\

(3,15,50)
& 7.3\%
& 4.0\%
& 11.4\%
& 1.7\%
& 23.1\%
& 14.3\%
& 0.4\%
& \textcolor{posgreen}{\textbf{0.3\%}} \\

\midrule
\textit{Mean gap}
& 6.8\%
& 5.5\%
& 9.2\%
& 2.6\%
& 24.7\%
& 12.5\%
& 0.4\%
& \textcolor{posgreen}{\textbf{0.3\%}} \\

\bottomrule
\end{tabular}
\end{threeparttable}
\end{table}

\begin{result}
Our \PNN and \OCNN policies achieve near-optimal performance on small tractable MNL instances, with average optimality gaps of 0.4\% and 0.3\%, respectively.
\end{result}

\subsubsection{On large-scale MNL and \MMNL instances}
We further evaluate our methods on large-scale and intractable instances with full information about the true demand environment. In this setting, the baseline selected during validation also turns out to be the best-performing baseline in testing. Table~\ref{tab:summary:fi} summarizes the testing performance of the baseline policies and our architectures (selected during validation) under both the MNL and \MMNL environments. The validation-selected baseline is highlighted in gray, with its revenue normalized to 0\%. Green highlights the test-best result among all evaluated policies in each row. Mean test revenues of all pricing policies are reported in Table~\ref{tab:full_summary:fi} in Appendix~\ref{appendix:full-info testing}.

From Table~\ref{tab:summary:fi}, we observe that our decision-focused policies outperform the DP-based baselines across all instances. In the MNL settings, the performance gap is relatively modest, with an average gain of 2.9\%, consistent with the performance on small MNL instances in Section \ref{subsubsec:on small tractable MNL instances}, suggesting that the selected baselines already provide effective pricing policies when the underlying choice model is well aligned with their MNL-based design. In contrast, under the more complex \MMNL settings, the advantage of our methods becomes substantially larger, with the average gain increasing to 27.2\%. This result highlights the greater robustness of the proposed DFL framework to more complex choice environments. While the baseline policies are developed primarily under MNL assumptions and may lose effectiveness when the underlying demand model deviates from this structure, our DFL architectures can adapt more effectively to such model mismatch.

\begin{table}[ht]
\centering
\small
\caption{Full-information testing results relative to the \colorbox{evalgray}{\strut validation-selected baselines} (0\%, highlighted in gray). ``Ours'' indicates the testing results of the validation-selected configuration. \textcolor{posgreen}{\textbf{Green}} indicates the best result.}
\label{tab:summary:fi}
\begin{tabular}{>{\bfseries}l c c c c c c c}
\toprule
Instance
& \TimeIndItinerary
& \TimeDepItinerary
& \TimeIndApp
& \TimeDepApp
& \TimeIndUnified
& \TimeDepUnified
& \textbf{Ours} \\
\midrule

\rowcolor{subgray}
\multicolumn{8}{c}{\textbf{MNL} (full information)} \\

(6,150,400)
& -10.5\%
& -3.1\%
& -0.1\%
& \cellcolor{evalgray}0\%
& -10.6\%
& -10.6\%
& \textcolor{posgreen}{\textbf{+0.7\%}} \\

(6,210,1600)
& -17.1\%
& -0.4\%
& -18.5\%
& \cellcolor{evalgray}0\%
& -5.4\%
& -15.4\%
& \textcolor{posgreen}{\textbf{+3.7\%}} \\

(12,90,1800)
& -27.3\%
& -2.8\%
& -31.2\%
& \cellcolor{evalgray}0\%
& -6.3\%
& -24.3\%
& \textcolor{posgreen}{\textbf{+5.3\%}} \\

(12,180,2000)
& -18.9\%
& -4.1\%
& -15.9\%
& \cellcolor{evalgray}0\%
& -9.9\%
& -15.3\%
& \textcolor{posgreen}{\textbf{+1.2\%}} \\

(15,30,600)
& -25.6\%
& -5.3\%
& -23.3\%
& \cellcolor{evalgray}0\%
& -7.5\%
& -16.0\%
& \textcolor{posgreen}{\textbf{+4.1\%}} \\

(15,120,1400)
& -23.7\%
& -4.1\%
& -19.3\%
& \cellcolor{evalgray}0\%
& -11.0\%
& -19.8\%
& \textcolor{posgreen}{\textbf{+3.0\%}} \\

(21,30,800)
& -28.2\%
& -4.2\%
& -23.7\%
& \cellcolor{evalgray}0\%
& -3.0\%
& -15.8\%
& \textcolor{posgreen}{\textbf{+6.2\%}} \\

(21,30,1400)
& -24.4\%
& \cellcolor{evalgray}0\%
& -33.1\%
& -1.4\%
& -9.7\%
& -26.9\%
& \textcolor{posgreen}{\textbf{+0.3\%}} \\

(24,30,600)
& -29.6\%
& -5.2\%
& -16.1\%
& \cellcolor{evalgray}0\%
& -8.7\%
& -16.5\%
& \textcolor{posgreen}{\textbf{+3.9\%}} \\

(30,120,1400)
& -22.8\%
& -6.9\%
& -3.3\%
& \cellcolor{evalgray}0\%
& -19.5\%
& -19.5\%
& \textcolor{posgreen}{\textbf{+0.4\%}} \\
\hline
\textbf{Mean}
& -22.8\%
& -3.6\%
& -18.5\%
& -0.1\%
& -9.2\%
& -18.0\%
& \textcolor{posgreen}{\textbf{+2.9\%}} \\

\midrule

\rowcolor{subgray}
\multicolumn{8}{c}{\textbf{\MMNL} (full information)} \\

(6,150,400)
& -2.8\%
& \cellcolor{evalgray}0\%
& -29.7\%
& -29.7\%
& -16.0\%
& -16.0\%
& \textcolor{posgreen}{\textbf{+34.0\%}} \\

(6,210,1600)
& \cellcolor{evalgray}0\%
& -1.1\%
& -1.5\%
& -5.3\%
& -12.3\%
& -0.6\%
& \textcolor{posgreen}{\textbf{+9.3\%}} \\

(12,90,1800)
& \cellcolor{evalgray}0\%
& -5.7\%
& -5.9\%
& -14.5\%
& -52.5\%
& -17.5\%
& \textcolor{posgreen}{\textbf{+9.8\%}} \\

(12,180,2000)
& \cellcolor{evalgray}0\%
& -1.4\%
& -24.6\%
& -31.0\%
& -59.0\%
& -35.3\%
& \textcolor{posgreen}{\textbf{+16.4\%}} \\

(15,30,600)
& \cellcolor{evalgray}0\%
& -11.1\%
& -15.6\%
& -6.9\%
& -25.3\%
& -11.3\%
& \textcolor{posgreen}{\textbf{+18.7\%}} \\

(15,120,1400)
& -3.1\%
& \cellcolor{evalgray}0\%
& -35.6\%
& -16.1\%
& -45.2\%
& -4.9\%
& \textcolor{posgreen}{\textbf{+18.9\%}} \\

(21,30,800)
& -12.7\%
& -22.0\%
& -11.1\%
& \cellcolor{evalgray}0\%
& -35.8\%
& -10.8\%
& \textcolor{posgreen}{\textbf{+46.0\%}} \\

(21,30,1400)
& -14.5\%
& -20.4\%
& -9.7\%
& \cellcolor{evalgray}0\%
& -28.7\%
& -19.2\%
& \textcolor{posgreen}{\textbf{+28.3\%}} \\

(24,30,600)
& -15.1\%
& -24.8\%
& -8.4\%
& \cellcolor{evalgray}0\%
& -28.6\%
& -19.5\%
& \textcolor{posgreen}{\textbf{+49.2\%}} \\

(30,120,1400)
& \cellcolor{evalgray}0\%
& -23.9\%
& -12.5\%
& -14.1\%
& -28.3\%
& -28.3\%
& \textcolor{posgreen}{\textbf{+41.7\%}} \\

\hline
\textbf{Mean}
& -4.8\%
& -11.0\%
& -15.5\%
& -11.8\%
& -33.2\%
& -16.3\%
& \textcolor{posgreen}{\textbf{+27.2\%}} \\

\bottomrule
\end{tabular}
\end{table}

We further compare the two proposed variants, \OCNN and \PNN, in Table~\ref{tab:odfl_pdfl_diff_fullinfo}, and observe a similar pattern. Under the full-information MNL setting, \OCNN generally outperforms \PNN, reflecting the benefit of explicitly incorporating structural results derived from the MNL model into the policy architecture. Under \MMNL, however, this MNL-specific structure becomes less aligned with the more heterogeneous substitution patterns and price sensitivities, and \PNN consistently achieves better performance. This reversal highlights the value of a more flexible architecture when the underlying customer-choice behavior deviates from the MNL structure. More broadly, these results reveal a trade-off between structural specialization and model flexibility: embedding RM structure can improve performance when the assumed choice model is well specified, whereas a more flexible architecture becomes advantageous under model mismatch.

\definecolor{poscolor}{RGB}{230,126,34}   %
\definecolor{negcolor}{RGB}{42,157,143}   %

\begin{table}[htbp]
\centering
\setlength{\tabcolsep}{2.5pt}
\renewcommand{\arraystretch}{1.15}

\caption{Relative revenue differences between ODFL and PDFL under full information, computed as $(\mathrm{ODFL}-\mathrm{PDFL})/\mathrm{PDFL}\times100\%$. More \textcolor{poscolor}{\textbf{orange}} favors ODFL, while more \textcolor{negcolor}{\textbf{teal}} favors PDFL.}
\label{tab:odfl_pdfl_diff_fullinfo}

\resizebox{\textwidth}{!}{
\begin{tabular}{@{}l*{10}{c}@{}}
\toprule
\textbf{Instance}
& \rotatebox[origin=l]{45}{\textbf{(6,150,400)}}
& \rotatebox[origin=l]{45}{\textbf{(6,210,1600)}}
& \rotatebox[origin=l]{45}{\textbf{(12,90,1800)}}
& \rotatebox[origin=l]{45}{\textbf{(12,180,2000)}}
& \rotatebox[origin=l]{45}{\textbf{(15,30,600)}}
& \rotatebox[origin=l]{45}{\textbf{(15,120,1400)}}
& \rotatebox[origin=l]{45}{\textbf{(21,30,800)}}
& \rotatebox[origin=l]{45}{\textbf{(21,30,1400)}}
& \rotatebox[origin=l]{45}{\textbf{(24,30,600)}}
& \rotatebox[origin=l]{45}{\textbf{(30,120,1400)}} \\
\midrule

\textbf{MNL}
& \cellcolor{poscolor!6}  +0.4\%
& \cellcolor{poscolor!8}  +1.6\%
& \cellcolor{poscolor!6}  +0.6\%
& \cellcolor{poscolor!7}  +1.1\%
& \cellcolor{poscolor!7}  +1.2\%
& \cellcolor{poscolor!9}  +2.2\%
& \cellcolor{poscolor!6}  +0.8\%
& \cellcolor{negcolor!8}  -1.7\%
& \cellcolor{poscolor!8}  +1.9\%
& \cellcolor{poscolor!6}  +0.7\% \\

\textbf{MixMNL}
& \cellcolor{negcolor!41} -23.8\%
& \cellcolor{negcolor!16} -7.2\%
& \cellcolor{negcolor!26} -14.1\%
& \cellcolor{negcolor!48} -28.6\%
& \cellcolor{negcolor!31} -17.5\%
& \cellcolor{negcolor!33} -18.5\%
& \cellcolor{negcolor!40} -23.3\%
& \cellcolor{negcolor!27} -14.8\%
& \cellcolor{negcolor!48} -28.5\%
& \cellcolor{negcolor!52} -31.0\% \\

\bottomrule
\end{tabular}
}
\end{table}

\begin{result}
Our DFL policies consistently outperform the DP-based baselines, and the results reveal a clear trade-off between structural specialization and flexibility. Under MNL, the baselines remain competitive and \OCNN generally outperforms \PNN, as their MNL-specific structures are well matched to the environment. Under \MMNL, this structural advantage diminishes: our performance advantage over the baselines increases substantially, while the more flexible \PNN outperforms \OCNN.
\end{result}

\begin{table}[htbp]
\centering
\small
\caption{Data-driven testing results relative to the \colorbox{evalgray}{\strut validation-selected baselines} (0\%, highlighted in gray). ``Ours'' indicates testing results of the validation-selected configuration. \textcolor{posgreen}{\textbf{Green}} indicates the best result.}
\label{tab: data-driven overall results}
\begin{threeparttable}
\small
\begin{tabular}{>{\bfseries}l c c c c c c c}
\toprule
Instance
& \TimeIndItinerary
& \TimeDepItinerary
& \TimeIndApp
& \TimeDepApp
& \TimeIndUnified
& \TimeDepUnified
& \textbf{Ours} \\
\midrule

\rowcolor{subgray}
\multicolumn{8}{c}{\textbf{MNL} (data-driven)} \\

(6,150,400)
& -12.7\%
& -12.2\%
& 0\%
& \cellcolor{evalgray}0\%
& -13.8\%
& -13.8\%
& \textcolor{posgreen}{\textbf{+0.9\%}} \\

(6,210,1600)
& -21.1\%
& -10.0\%
& -18.7\%
& \cellcolor{evalgray}0\%
& -4.8\%
& -16.9\%
& \textcolor{posgreen}{\textbf{+4.2\%}} \\

(12,90,1800)
& -34.0\%
& -4.1\%
& -29.9\%
& \cellcolor{evalgray}0\%
& -18.3\%
& -37.1\%
& \textcolor{posgreen}{\textbf{+0.8\%}} \\

(12,180,2000)
& -24.6\%
& -16.8\%
& -15.6\%
& \cellcolor{evalgray}0\%
& -15.6\%
& -19.4\%
& \textcolor{posgreen}{\textbf{+0.9\%}} \\

(15,30,600)
& -28.9\%
& -8.7\%
& -18.7\%
& \cellcolor{evalgray}0\%
& -8.9\%
& -28.4\%
& \textcolor{posgreen}{\textbf{+1.1\%}} \\

(15,120,1400)
& -30.1\%
& -20.5\%
& -18.2\%
& \cellcolor{evalgray}0\%
& -18.2\%
& -24.6\%
& \textcolor{posgreen}{\textbf{+1.4\%}} \\

(21,30,800)
& -29.9\%
& -13.2\%
& -18.4\%
& \cellcolor{evalgray}0\%
& -16.9\%
& -30.6\%
& \textcolor{posgreen}{\textbf{+0.8\%}} \\

(21,30,1400)
& \textcolor{posgreen}{\textbf{+34.7\%}}
& -5.5\%
& +34.2\%
& \cellcolor{evalgray}0\%
& +8.8\%
& -36.6\%
& +23.1\% \\

(24,30,600)
& -34.0\%
& -36.2\%
& -14.2\%
& \cellcolor{evalgray}0\%
& -17.2\%
& -22.2\%
& \textcolor{posgreen}{\textbf{+0.8\%}} \\

(30,120,1400)
& -28.4\%
& -46.0\%
& -2.5\%
& \cellcolor{evalgray}0\%
& -22.5\%
& -22.5\%
& \textcolor{posgreen}{\textbf{+0.4\%}} \\

\hline
\textbf{Mean}
& -20.9\%
& -17.3\%
& -10.2\%
& 0\%
& -12.7\%
& -25.2\%
& \textcolor{posgreen}{\textbf{+3.4\%}} \\

\midrule

\rowcolor{subgray}
\multicolumn{8}{c}{\textbf{\MMNL} (data-driven)} \\

(6,150,400)
& -35.0\%
& -19.9\%
& -5.6\%
& \cellcolor{evalgray}0\%
& -16.0\%
& -16.0\%
& \textcolor{posgreen}{\textbf{+10.8\%}} \\

(6,210,1600)
& -11.9\%
& -7.8\%
& -6.1\%
& \cellcolor{evalgray}0\%
& -2.9\%
& -11.7\%
& \textcolor{posgreen}{\textbf{+2.5\%}} \\

(12,90,1800)
& -2.6\%
& +2.0\%
& \textcolor{posgreen}{\textbf{+4.5\%}}
& \cellcolor{evalgray}0\%
& -18.4\%
& -33.2\%
& +3.3\% \\

(12,180,2000)
& +15.5\%
& +5.1\%
& +8.2\%
& \cellcolor{evalgray}0\%
& -10.2\%
& -15.9\%
& \textcolor{posgreen}{\textbf{+26.5\%}} \\

(15,30,600)
& +9.0\%
& +26.8\%
& +18.8\%
& \textcolor{posgreen}{\textbf{+34.0\%}}
& +7.7\%
& \cellcolor{evalgray}0\%
& +33.5\% \\

(15,120,1400)
& -5.8\%
& -5.2\%
& -3.4\%
& \cellcolor{evalgray}0\%
& -18.4\%
& -20.2\%
& \textcolor{posgreen}{\textbf{+12.6\%}} \\

(21,30,800)
& -17.4\%
& +8.5\%
& -2.3\%
& +20.8\%
& +18.3\%
& \cellcolor{evalgray}0\%
& \textcolor{posgreen}{\textbf{+28.3\%}} \\

(21,30,1400)
& -27.2\%
& -1.7\%
& -17.0\%
& \cellcolor{evalgray}0\%
& -6.6\%
& -41.5\%
& \textcolor{posgreen}{\textbf{+11.3\%}} \\

(24,30,600)
& -22.6\%
& -25.7\%
& -4.6\%
& -1.3\%
& \cellcolor{evalgray}0\%
& -5.2\%
& \textcolor{posgreen}{\textbf{+7.8\%}} \\

(30,120,1400)
& -6.2\%
& -40.0\%
& -2.6\%
& \cellcolor{evalgray}0\%
& -8.8\%
& -8.8\%
& \textcolor{posgreen}{\textbf{+6.8\%}} \\

\hline
\textbf{Mean}
& -10.4\%
& -5.8\%
& -1.0\%
& +5.4\%
& -5.5\%
& -15.3\%
& \textcolor{posgreen}{\textbf{+14.3\%}} \\

\bottomrule
\end{tabular}
\end{threeparttable}
\end{table}

\definecolor{poscolor}{RGB}{230,126,34}   %
\definecolor{negcolor}{RGB}{42,157,143}   %

\begin{table}[tbp]
\centering
\setlength{\tabcolsep}{2.5pt}
\renewcommand{\arraystretch}{1.15}

\caption{Relative revenue differences between ODFL and PDFL in the data-driven setting, computed as $(\mathrm{ODFL}-\mathrm{PDFL})/\mathrm{PDFL}\times100\%$. More \textcolor{poscolor}{\textbf{orange}} favors ODFL, while more \textcolor{negcolor}{\textbf{teal}} favors PDFL.}
\label{tab:odfl_pdfl_diff_datadriven}

\resizebox{\textwidth}{!}{
\begin{tabular}{@{}l*{10}{c}@{}}
\toprule
\textbf{Instance}
& \rotatebox[origin=l]{45}{\textbf{(6,150,400)}}
& \rotatebox[origin=l]{45}{\textbf{(6,210,1600)}}
& \rotatebox[origin=l]{45}{\textbf{(12,90,1800)}}
& \rotatebox[origin=l]{45}{\textbf{(12,180,2000)}}
& \rotatebox[origin=l]{45}{\textbf{(15,30,600)}}
& \rotatebox[origin=l]{45}{\textbf{(15,120,1400)}}
& \rotatebox[origin=l]{45}{\textbf{(21,30,800)}}
& \rotatebox[origin=l]{45}{\textbf{(21,30,1400)}}
& \rotatebox[origin=l]{45}{\textbf{(24,30,600)}}
& \rotatebox[origin=l]{45}{\textbf{(30,120,1400)}} \\
\midrule

\textbf{MNL}
& \cellcolor{poscolor!6}  +0.5\%
& \cellcolor{poscolor!10} +1.4\%
& 0.0\%
& \cellcolor{poscolor!7}  +0.6\%
& \cellcolor{poscolor!5}  +0.2\%
& \cellcolor{poscolor!7}  +0.6\%
& \cellcolor{poscolor!4}  +0.1\%
& \cellcolor{negcolor!45} -17.1\%
& \cellcolor{poscolor!4}  +0.1\%
& \cellcolor{poscolor!4}  +0.1\% \\

\textbf{MixMNL}
& \cellcolor{poscolor!30} +5.5\%
& \cellcolor{negcolor!5}  -0.2\%
& \cellcolor{poscolor!9}  +1.2\%
& \cellcolor{negcolor!10} -1.7\%
& \cellcolor{poscolor!5}  +0.3\%
& \cellcolor{negcolor!16} -4.0\%
& 0.0\%
& \cellcolor{negcolor!9}  -1.5\%
& \cellcolor{negcolor!8}  -1.2\%
& \cellcolor{negcolor!6}  -0.7\% \\

\bottomrule
\end{tabular}
}
\end{table}

\subsection{Data-Driven Setting}
\label{subsec:data_driven_results}
We now consider the data-driven setting, where the ground-truth environment is unknown during policy construction and selection.
In this setting, the surrogate MNL model $\mathbb{Q}$ is estimated from observed price--choice data and is then used throughout the policy-development process:
it generates the anticipative training trajectories for the DFL architectures, provides the parameters required by JoPri and JoComPri, and serves as the validation environment for selecting the final policy configurations. Once selected, the policies are fixed and evaluated on independent trajectories generated from the ground-truth environment.

The resulting test performance is summarized in Table~\ref{tab: data-driven overall results}. Overall, we observe a pattern similar to that in the full-information setting reported in Table~\ref{tab:summary:fi}: our architectures consistently outperform the validation-selected baselines across all cases and achieve the highest testing revenue among all evaluated policies on 17 of the 20 instances. The improvement is relatively modest under MNL, with an average gain of \(3.4\%\), but becomes substantially larger under \MMNL, reaching \(14.3\%\) on average.

A similar pattern emerges from the comparison between \PNN and \OCNN in Table~\ref{tab:odfl_pdfl_diff_datadriven}. Under MNL, \OCNN generally outperforms \PNN, benefiting from its MNL-based structural design. Under \MMNL, however, \PNN exhibits more stable and generally stronger performance, suggesting that its greater flexibility is advantageous when the underlying choice behavior deviates from the MNL structure.

It is worth noting that, since validation is conducted using the fitted MNL model in the data-driven setting, the validation-selected policy need not perform best in the true testing environment. For example, on instance $(21,30,1400)$ in the MNL setting, \TimeIndItinerary achieves the highest testing revenue among the baselines, exceeding the validation-selected baseline by 34.7\%. This suggests a substantial mismatch between the fitted and true demand environments. A similar effect arises for our architecture on the same instance: the test-best configuration (Table \ref{tab:summary:dd}) improves upon the validation-selected baseline by 36.4\% and therefore also outperforms \TimeIndItinerary, despite not being selected during validation under the fitted model.

\begin{result}
Our DFL policy outperforms the validation-selected baseline on all 20 data-driven instances and achieves the highest testing revenue among all evaluated policies on 17 of the 20 instances.
\end{result}

\begin{table}[htbp]
\centering
\setlength{\tabcolsep}{1.5pt}
\renewcommand{\arraystretch}{1.15}

\caption{Relative differences of validation-selected baselines between the data-driven (DD) and full-information (FI) settings under \MMNL, computed as $(\mathrm{DD}-\mathrm{FI})/\mathrm{FI}\times100\%$. More \textcolor{poscolor}{\textbf{orange}} favors DD, while more \textcolor{negcolor}{\textbf{teal}} favors FI.}
\label{tab:fi_dd_validation_baseline_mixmnl}

\footnotesize
\begin{tabular}{@{}l*{10}{c}@{}}
\toprule
\textbf{Instance}
& \rotatebox[origin=l]{45}{\textbf{(6,150,400)}}
& \rotatebox[origin=l]{45}{\textbf{(6,210,1600)}}
& \rotatebox[origin=l]{45}{\textbf{(12,90,1800)}}
& \rotatebox[origin=l]{45}{\textbf{(12,180,2000)}}
& \rotatebox[origin=l]{45}{\textbf{(15,30,600)}}
& \rotatebox[origin=l]{45}{\textbf{(15,120,1400)}}
& \rotatebox[origin=l]{45}{\textbf{(21,30,800)}}
& \rotatebox[origin=l]{45}{\textbf{(21,30,1400)}}
& \rotatebox[origin=l]{45}{\textbf{(24,30,600)}}
& \rotatebox[origin=l]{45}{\textbf{(30,120,1400)}} \\
\midrule

Full-information
& 2,238.3
& 6,582.0
& 10,376.9
& 15,650.9
& 3,305.2
& 10,231.3
& 3,584.5
& 6,596.4
& 3,047.0
& 8,423.2 \\

Data-driven
& 2,682.7
& 7,182.6
& 11,198.6
& 12,935.1
& 2,993.8
& 11,010.5
& 4,179.7
& 7,821.8
& 3,953.8
& 10,238.1 \\

\midrule

\textbf{Diff. (\%)}
& \cellcolor{poscolor!40} +19.9\%
& \cellcolor{poscolor!21} +9.1\%
& \cellcolor{poscolor!19} +7.9\%
& \cellcolor{negcolor!45} -17.4\%
& \cellcolor{negcolor!27} -9.4\%
& \cellcolor{poscolor!18} +7.6\%
& \cellcolor{poscolor!35} +16.6\%
& \cellcolor{poscolor!38} +18.6\%
& \cellcolor{poscolor!55} +29.8\%
& \cellcolor{poscolor!43} +21.5\% \\

\bottomrule
\end{tabular}
\end{table}

An interesting pattern emerges when comparing the validation-selected baseline results in the full-information and data-driven settings under \MMNL as shown in Table \ref{tab:fi_dd_validation_baseline_mixmnl}. The validation-selected baseline in the data-driven setting achieves higher test revenue than in the full-information setting on 8 out of 10 instances.
Recall that, in the full-information setting, the MNL parameters used by the baseline pricing policies are obtained by minimizing the KL divergence between the MNL and the ground-truth \MMNL choice distributions.
In the data-driven setting, these parameters are instead estimated by maximum likelihood from a finite sample of observed customer choices. The former therefore provides a closer distributional approximation to the ground-truth \MMNL model.
Nevertheless, despite its less accurate distributional fit, the data-driven baseline often achieves better downstream performance.

This observation suggests that a better approximation of the underlying choice distribution does not necessarily translate into better decisions in the dynamic pricing problem.
Estimation errors that are unfavorable from a distributional perspective may nevertheless lead to parameter estimates that achieve better downstream revenue performance.
This finding also supports the motivation behind our architecture: rather than directly imitating prices or fitting the full choice distribution, we learn from anticipative customer-choice targets.
These targets provide a more compact and decision-oriented learning space that focuses directly on the decisions relevant to downstream pricing performance.

\begin{result}
Under \MMNL, the validation-selected baseline in the data-driven case outperforms its full-information counterpart on 8 out of 10 instances, suggesting that a better distributional fit does not necessarily lead to better downstream pricing decisions.
This further motivates learning from compact, decision-oriented customer-choice targets rather than directly imitating prices or the full choice distribution.
\end{result}

\begin{figure}[htbp]
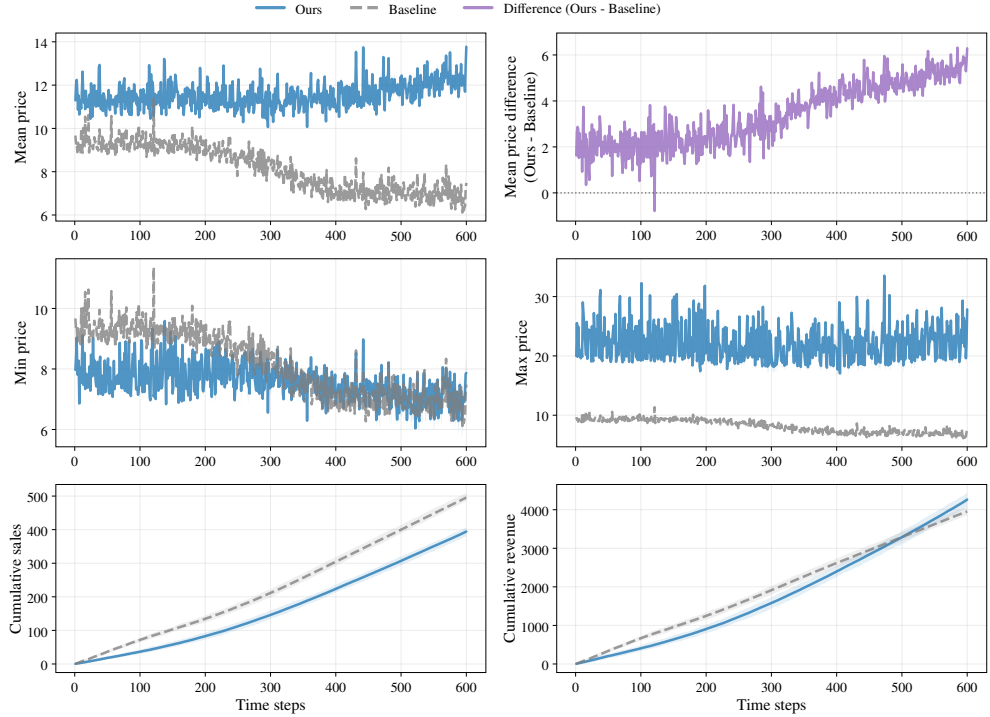

  \centering
  \resizebox{0.8\textwidth}{!}{%
    


    }
  \caption{Analysis of (24, 30, 600) in the data-driven \MMNL setting, where our architecture achieves a 7.8\% gain over the validation-selected baseline.}
  \label{fig:case:MMNL:dd:good:24:30:600}
\end{figure}

To better understand the source of the performance gains achieved by our architectures, we conduct a trajectory-level comparison on a representative instance in the data-driven \MMNL setting.
We consider instance $(24,30,600)$, for which our DFL-based pricing policy achieves a $7.8\%$ revenue improvement over the validation-selected baseline (which is also the overall-best baseline in testing). Figure~\ref{fig:case:MMNL:dd:good:24:30:600} compares our pricing policy with the best baseline over the selling horizon in terms of the average, minimum, and maximum prices, as well as cumulative sales and cumulative revenue.
The blue solid line and gray dashed line represent the average trajectories of the selected DFL policy and the best baseline, respectively, while the shaded regions indicate variation across 100 trajectories. The comparison reveals a markedly different pricing pattern between the two policies.
While the baseline follows a relatively conservative pricing policy with decreasing prices throughout the horizon, the DFL-based policy maintains substantially higher average prices and a much wider price range. In particular, the maximum price posted by our policy remains larger than that of the baseline, whereas the minimum price is initially lower but eventually exceeds the baseline's minimum price toward the end of the horizon. These pricing differences translate into distinct sales dynamics. The baseline achieves higher cumulative sales over the entire horizon, indicating that its lower prices stimulate additional demand. However, our policy is able to extract considerably more revenue from each sale through more aggressive price differentiation. As a result, despite selling fewer units, the proposed policy gradually overtakes the baseline in cumulative revenue later in the selling horizon, and the revenue advantage continues to widen.

\begin{result}
Our DFL policy can better adapt to the heterogeneous demand structure under \MMNL, using stronger price differentiation to generate higher revenue despite fewer sales.
\end{result}

\section{Conclusion and Remarks} \label{sec:conclusion and extensions}
In this paper, we develop a structured DFL framework for dynamic pricing and inventory control with substitutable products. Rather than approximating the full Bellman value function or learning a completely structure-free policy, we replace the high-dimensional DP state-to-price mapping with simpler MNL-guided surrogate mappings.
We propose two architectures: a more flexible price-based architecture and a more structured opportunity-cost-based architecture. Both architectures can be cast as regularized linear optimization layers and trained within a common DFL framework.
Rather than imitating anticipative price decisions, we learn from anticipative customer choices and develop a perfect-information flow oracle to generate these targets in polynomial time.

Our numerical results provide consistent evidence for the effectiveness of the proposed DFL policies. On small tractable MNL instances, \PNN and \OCNN achieve near-optimal performance, with average optimality gaps of only 0.4\% and 0.3\%, respectively.
On larger instances, both policies consistently outperform the DP-based baselines, while revealing a clear trade-off between structural specialization and flexibility. Under MNL, where the embedded MNL structure is well aligned with the environment, the baselines remain competitive and the more structured \OCNN generally outperforms \PNN. Under heterogeneous \MMNL demand, this structural advantage weakens: the performance gains over the baselines become substantially larger, and the more flexible \PNN performs better than \OCNN. The same overall pattern persists in the data-driven setting, where the true choice environment is unknown.
Interestingly, under \MMNL, the data-driven baseline even outperforms its full-information counterpart on 8 out of 10 instances, illustrating that a more accurate approximation of the underlying choice distribution does not necessarily translate into better downstream pricing decisions. This finding reinforces the motivation for learning from compact, decision-oriented customer-choice targets rather than directly imitating prices or the full choice distribution. Our trajectory-level analysis further shows that the DFL policy adapts to heterogeneous demand through stronger price differentiation, generating higher revenue despite fewer sales.

Several directions remain for future research. The framework could be extended to richer revenue-management settings, including network RM, joint assortment and pricing, and environments with cancellations or no-shows. From a methodological perspective, a potential direction is to develop theoretical guarantees linking anticipative-choice supervision to the quality of nonanticipative policies.
More broadly, our study illustrates that DFL can be used to replace a computationally demanding optimization stage, while preserving useful RM problem structure and training directly for the decisions that ultimately matter.


\begin{APPENDICES}
\renewcommand{\theHsection}{appendix.\Alph{section}}

\newpage
\section{Proofs of Theoretical Results} \label{appendix:proofs}
\subsection{Proof of Proposition~\ref{prop:ocnn_regularizer_properties}}
\label{appendix:ocnn_regularizer_properties}

\begin{proof}
     By definition, $\Omega$ is smooth. Then, we show that $\Omega$ is a convex function. From the definition of $\Omega$, we have:
    \begin{align}
    &\frac{\partial^2\Omega}{\partial y_i \partial y_i} = \frac{1}{y_i}, i \in \Products, \quad \frac{\partial^2\Omega}{\partial y_i \partial y_j} = 0, j \neq i \nonumber\\
    &\frac{\partial^2\Omega}{\partial y_0 \partial y_0} = \frac{\sum_i y_i}{y_0^2}, \quad  \frac{\partial^2\Omega}{\partial y_0 \partial y_i} = \frac{-1}{y_0}, i \in \Products \nonumber
    \end{align}
    where $i$ and $j$ are available product indices. Then, the Hessian matrix of $\Omega$ is:
\begin{equation}
\setlength{\arraycolsep}{8pt}
H_{\Omega}=
\left(
\begin{array}{ccccc}
\sum_i y_i/y_0^2 & -1/y_0 & -1/y_0 & \cdots & -1/y_0\\
-1/y_0 & 1/y_1 & 0 & \cdots & 0\\
-1/y_0 & 0 & 1/y_2 & \cdots & 0\\
\vdots & \vdots & \vdots & \ddots & \vdots\\
-1/y_0 & 0 & 0 & \cdots & 1/y_n
\end{array}
\right).
\end{equation}
    Let $\boldsymbol z = ( z_0, z_1, \cdots, z_n)$ be a non-zero vector. Then,
    \begin{equation}
        \boldsymbol{z}H_{\Omega}\boldsymbol{z}^{\mathrm{T}} = \sum \limits_{i=1}^n y_i (\frac{z_i}{y_i} - \frac{z_0}{y_0})^2 \geq 0 \nonumber
    \end{equation}
    Therefore, $\Omega$ is convex on $\mathcal C(\boldsymbol I_t)$.

   Assuming that the domain of $\Omega$ is denoted by $\mathrm{dom}(\Omega)$ and the interior of $\mathrm{dom}(\Omega)$ is $\mathrm{int}(\mathrm{dom}(\Omega))$, we now prove that $\Omega$ is of Legendre type.

    Proving that $\Omega$ is a Legendre-type function is equivalent to proving that $\Omega$ satisfies the following conditions:
    \begin{itemize}
        \item [1)] $\mathrm{dom}(\Omega)$ is non-empty.
        \item [2)] $\Omega$ is differentiable throughout $\mathrm{int}(\mathrm{dom}(\Omega))$.
        \item [3)] $\lim_{j \rightarrow \infty}\nabla \Omega(\boldsymbol{y}^{(j)}) = +\infty$ for any sequence $\{\boldsymbol{y}^{(j)}\}$ contained in $\mathrm{dom}(\Omega)$ and converging to a boundary point of $\mathrm{dom}(\Omega)$.
        \item [4)] $\Omega$ is strictly convex on $\mathrm{int}(\mathrm{dom}(\Omega))$.
    \end{itemize}

From the definition of $\Omega$, we have
\begin{equation}
\mathrm{dom}(\Omega) = \mathcal{C}(\boldsymbol I_t) =  \left \{ \boldsymbol{y} \in\mathbb{R}_{\geq 0}^{\nb+1} \left |  \sum_{i\in \mathcal{A}(\boldsymbol{I}_t)\cup \{0\}} y_i = 1,\quad y_i =0 \text{ for } i \not\in \mathcal{A}(\boldsymbol{I}_t) \cup \{0\} \right. \right \}, \nonumber
\end{equation}
Then,  (1) and (2) hold directly, and we focus on proving (3) and (4).
    \begin{enumerate}
        \item To prove (3), we first calculate $\nabla \Omega(\boldsymbol y)$: \begin{equation}
       \nabla \Omega(\boldsymbol y) = \left(
        \begin{array}{c}
         \frac{-\sum_i y_i}{y_0} \\
            \log(y_1/y_0) + 1\\
            \log(y_2/y_0) + 1 \\
            \vdots\\
            \log(y_n/y_0) + 1           \nonumber
        \end{array}
        \right)
    \end{equation}
    Since  $\{\boldsymbol{y}^{(j)} = (y_0^{(j)},\ldots, y_\nb^{(j)})\}$ converges to a boundary point of $\mathrm{dom}(\Omega)=\mathcal{C}(\boldsymbol I_t)$, there must exist an $i\in \mathcal{A}\cup\{0\}$ such that $\lim_{j\rightarrow\infty} y_i^{(j)} = 0$. Two cases arise:
    \begin{itemize}
     \item If $i = 0$, then $\frac{-\sum_i y_i}{y_0} \rightarrow - \infty$ and we also have  $\lim_{j \rightarrow \infty}\nabla \Omega(\boldsymbol{y}^{(j)}) = +\infty$.
        \item If $i \in \mathcal{A}(\boldsymbol I_t)$, then $\log(y_i^{(j)}/y_0^{(j)}) + 1 \rightarrow - \infty$, and thus we have $\lim_{j \rightarrow \infty}\nabla \Omega(\boldsymbol{y}^{(j)}) = +\infty$.
    \end{itemize}
    \item To prove (4), we use the convexity of $\Omega$: for a non-zero vector $\boldsymbol{z} = (z_0, z_1, \cdots, z_n)\in  \mathcal{C}(\boldsymbol I_t)$, $\boldsymbol{z}H_{\Omega}\boldsymbol{z}^{\mathrm{T}} = \sum_i y_i (\frac{z_i}{y_i} - \frac{z_0}{y_0})^2 \geq 0$.
    Clearly, $\boldsymbol{z}H_{\Omega}\boldsymbol{z}^{\mathrm{T}} = 0$ for any non-negative vector $\boldsymbol z$. Then, we have $\boldsymbol{z}H_{\Omega}\boldsymbol{z}^{\mathrm{T}} > 0$ on $\mathrm{int}(\mathrm{dom}(\Omega))$ for  $\mathrm{dom}(\Omega)=\mathcal{C}(\boldsymbol I_t)$, and thus $\Omega$ is strictly convex on $\mathrm{int}(\mathrm{dom}(\Omega))$.
    \end{enumerate}

\end{proof}

\subsection{Convex Conjugate and Fenchel--Young Loss for \OCNN}
\label{appendix:ocnn_fy_loss}

\begin{proof}[Proof of Proposition~\ref{prop:ocnn-conjugate}]
Based on the definition of the convex conjugate of $\Omega$, we have
\[
\Omega^*(\boldsymbol \theta)
=
\max_{\boldsymbol y \in \mathcal{C}(\boldsymbol I_t)}
\left\{
\boldsymbol \theta^\top \boldsymbol y
-
\Omega(\boldsymbol y)
\right\}
=
\boldsymbol \theta^\top \hat{\boldsymbol y}
-
\Omega(\hat{\boldsymbol y}),
\]
where $\hat{\boldsymbol y} = (\hat{y}_0, \ldots, \hat{y}_\nb)$ with
\[
\hat{y}_i
=
\frac{1}{m_{\boldsymbol{\theta}}}
\exp(\theta_i - m_{\boldsymbol{\theta}}), \qquad i\in\mathcal A(\boldsymbol I_t),
\]
and
\[
\hat{y}_0
=
\frac{1}{m_{\boldsymbol{\theta}}}.
\]

Then,
\[
\begin{aligned}
\Omega(\hat{\boldsymbol y})
&=
\sum_{i\in\mathcal{A}(\boldsymbol I_t)}
\hat{y}_i
\left(
\log \hat{y}_i
-
\log \hat{y}_0
\right) \\
&=
\sum_{i\in\mathcal{A}(\boldsymbol I_t)}
\hat{y}_i
\log
\frac{
\frac{1}{m_{\boldsymbol{\theta}}}
\exp(\theta_i-m_{\boldsymbol{\theta}})
}{
\frac{1}{m_{\boldsymbol{\theta}}}
} \\
&=
\sum_{i\in\mathcal{A}(\boldsymbol I_t)}
\hat{y}_i
(\theta_i-m_{\boldsymbol{\theta}}) \\
&=
\boldsymbol \theta^\top \hat{\boldsymbol y}
-
m_{\boldsymbol{\theta}}
\sum_{i\in\mathcal{A}(\boldsymbol I_t)}
\hat{y}_i \\
&=
\boldsymbol \theta^\top \hat{\boldsymbol y}
-
m_{\boldsymbol{\theta}}
(1-\hat{y}_0) \\
&=
\boldsymbol \theta^\top \hat{\boldsymbol y}
-
m_{\boldsymbol{\theta}}
\left(
1-\frac{1}{m_{\boldsymbol{\theta}}}
\right) \\
&=
\boldsymbol \theta^\top \hat{\boldsymbol y}
-
m_{\boldsymbol{\theta}}
+1.
\end{aligned}
\]
Therefore,
\[
\begin{aligned}
\Omega^*(\boldsymbol \theta)
&=
\boldsymbol \theta^\top \hat{\boldsymbol y}
-
\Omega(\hat{\boldsymbol y}) \\
&=
\boldsymbol \theta^\top \hat{\boldsymbol y}
-
\left(
\boldsymbol \theta^\top \hat{\boldsymbol y}
-
m_{\boldsymbol{\theta}}
+1
\right) \\
&=
m_{\boldsymbol{\theta}}-1.
\end{aligned}
\]
\end{proof}

\begin{proof}[Proof of Proposition~\ref{prop:ocnn-fy-loss-gradient}]
By definition, the Fenchel--Young loss is
\[
\mathcal L^{\mathrm{FY}}_{\Omega}
(\boldsymbol\theta,\bar{\boldsymbol y})
=
\Omega^*(\boldsymbol\theta)
+
\Omega(\bar{\boldsymbol y})
-
\boldsymbol\theta^\top \bar{\boldsymbol y}.
\]

From Proposition~\ref{prop:ocnn-conjugate},
\[
\Omega^*(\boldsymbol\theta)
=
m_{\boldsymbol\theta}-1.
\]
Hence,
\[
\mathcal L^{\mathrm{FY}}_{\Omega}
(\boldsymbol\theta,\bar{\boldsymbol y})
=
(m_{\boldsymbol\theta}-1)
+
\Omega(\bar{\boldsymbol y})
-
\boldsymbol\theta^\top \bar{\boldsymbol y}.
\]

For a no-purchase target, $\bar{\boldsymbol y}=\boldsymbol e_0$.
Since
\[
\Omega(\boldsymbol e_0)=0
\qquad\text{and}\qquad
\boldsymbol\theta^\top \boldsymbol e_0=0,
\]
we obtain
\[
\mathcal L^{\mathrm{FY}}_{\Omega}
(\boldsymbol\theta,\boldsymbol e_0)
=
m_{\boldsymbol\theta}-1.
\]

Finally, since $\Omega$ is of Legendre type,
$\Omega^*$ is differentiable and
\[
\nabla \Omega^*(\boldsymbol\theta)
=
\hat{\boldsymbol y}_{\Omega}(\boldsymbol\theta),
\]
where $\hat{\boldsymbol y}_{\Omega}(\boldsymbol\theta)$ is the choice
probability vector induced by the ODFL optimization layer. Therefore,
\[
\nabla_{\boldsymbol\theta}
\mathcal L^{\mathrm{FY}}_{\Omega}
(\boldsymbol\theta,\bar{\boldsymbol y})
=
\hat{\boldsymbol y}_{\Omega}(\boldsymbol\theta)
-
\bar{\boldsymbol y}.
\]
\end{proof}

\subsection{Proof of Proposition \ref{prop:perfect_information_reduction}}
\label{appendix:anticipative_oracle_auxiliary}

\begin{proof}
First, we show that, for a fixed realization $\boldsymbol{\eta}$ of the Gumbel utility shocks, Problem~\eqref{eq:reduced_anticipative} is equivalent to the perfect-information dynamic pricing problem. In particular, if $\boldsymbol{y}^*$ is an optimal solution to Problem~\eqref{eq:reduced_anticipative}, then $\boldsymbol{y}^*$ can be extended to an optimal solution $(\boldsymbol{y}^*,\boldsymbol{r}^*)$ of the following perfect-information pricing problem:
\begin{subequations}
\label{eq:generic_anticipative_problem}
\begin{align}
\max_{\boldsymbol y, \boldsymbol r}\quad
    & \sum_{t=1}^{T}
      \sum_{i\in\Products}
      r_{i,t}y_{i,t}
      \label{opt:generic_obj}
      \\
    \text{s.t.}\quad
    & I_{i,1}=c_i,
      && \forall i\in\Products,
      \\
    & I_{i,t+1}=I_{i,t}-y_{i,t},
      && \forall i\in\Products,\ t=1,\ldots,T,
      \\
    & \boldsymbol{y}_t
      \in
      \argmax_{\boldsymbol{y}\in \mathcal{C}(\boldsymbol{I}_t)}
      \boldsymbol{U}_t^\top \boldsymbol{y},
      && t=1,\ldots,T,
      \label{cons:choice_consistency}
      \\
    & I_{i,t}\geq 0,
      && \forall i\in\Products,\ t=1,\ldots,T,
      \\
    & y_{i,t}\in\{0,1\},
      && \forall i\in \Products\cup\{0\},\ t=1,\ldots,T,
      \\
    & r_{i,t}\geq 0,
      && \forall i\in\Products,\ t=1,\ldots,T,
\end{align}
\end{subequations}
Here, $\boldsymbol{U}_t=(\eta_{0,t},U_{1,t}(r_{1,t}),\ldots,U_{\nb,t}(r_{\nb,t}))$. This equivalence follows from the following two claims.

\begin{enumerate}
\item \textit{$\boldsymbol{y}^*$ can be extended to a feasible solution $(\boldsymbol{y}^*,\boldsymbol{r}^*)$ of the perfect-information pricing problem with revenue $$\sum_{t=1}^T\sum_{i\in\Products}\widetilde r_{i,t}y^*_{i,t}.$$}

An integral optimal solution exists; let $\boldsymbol{y}^*$ denote one such solution, that is, $y^*_{i,t}\in\{0,1\}, \forall i \in \mathcal{A}(\boldsymbol{I}_t)$. Let
$$
r^*_{i,t}
=
\begin{cases}
\widetilde r_{i,t},
& \text{if } y^*_{i,t}=1,\\[0.2cm]
M_{\boldsymbol{\eta}},
& \text{otherwise}.
\end{cases}
$$
Since the realization $\boldsymbol{\eta}$ is fixed and the number of products and time periods is finite, we can choose $M_{\boldsymbol{\eta}}$ sufficiently large such that if $y^*_{i,t}=1$ for some $i\in\Products$, then
$$
i\in
\arg\max_{j\in\{0\}\cup\mathcal A(\boldsymbol I_t)}
U_{j,t}(r^*_{j,t}).
$$
Similarly, if $y^*_{0,t}=1$, then we choose a sufficiently large $M_{\boldsymbol{\eta}}$ for $i \in A(\boldsymbol I_t)$ such that the no-purchase option has the maximum utility. By convention, we have $r^*_{0,t}=0$.

Therefore, $(\boldsymbol{y}^*,\boldsymbol{r}^*)$ is feasible for the perfect-information pricing problem and achieves revenue
$$
\sum_{t=1}^T\sum_{i\in\Products}
\widetilde r_{i,t}y^*_{i,t}.
$$

\item \textit{Any feasible solution $(\boldsymbol{y},\boldsymbol{r})$ of the perfect-information pricing problem has revenue no greater than $$\sum_{t=1}^T\sum_{i\in\Products}\widetilde r_{i,t}y^*_{i,t}.$$}

If customer $t$ purchases product $i$, i.e., $y_{i,t}=1$, choice consistency implies
$$
a_i-\beta r_{i,t}+\eta_{i,t}\geq\eta_{0,t},
$$
and therefore
$$
r_{i,t}
\leq
\frac{a_i+\eta_{i,t}-\eta_{0,t}}{\beta}
=
\widetilde r_{i,t}.
$$
It follows that
$$
\sum_{t=1}^T\sum_{i\in\Products}
r_{i,t}y_{i,t}
\leq
\sum_{t=1}^T\sum_{i\in\Products}
\widetilde r_{i,t}y_{i,t}.
$$
Moreover, $\boldsymbol{y}$ is feasible for Problem~\eqref{eq:reduced_anticipative}, since it satisfies the same inventory and choice-feasibility constraints. Because $\boldsymbol{y}^*$ is an optimal solution to Problem~\eqref{eq:reduced_anticipative},
$$
\sum_{t=1}^T\sum_{i\in\Products}
\widetilde r_{i,t}y_{i,t}
\leq
\sum_{t=1}^T\sum_{i\in\Products}
\widetilde r_{i,t}y^*_{i,t}.
$$
Combining the two inequalities yields
$$
\sum_{t=1}^T\sum_{i\in\Products}
r_{i,t}y_{i,t}
\leq
\sum_{t=1}^T\sum_{i\in\Products}
\widetilde r_{i,t}y^*_{i,t}.
$$
\end{enumerate}

The first claim shows that the optimal value of the perfect-information pricing problem is at least the optimal value of Problem~\eqref{eq:reduced_anticipative}, while the second claim establishes the reverse inequality. Hence, $(\boldsymbol{y}^*,\boldsymbol{r}^*)$ is an optimal solution to the perfect-information pricing problem.
\end{proof}

\newpage
\section{Instance Parameters} \label{app:instance_parameters}
We use small instances from \cite{dong2009dynamic} in our experiments. These instances share time-independent parameters $\boldsymbol{a}$ and $\beta$. We set $\beta = 1$ and $\boldsymbol{a} = (11.75, 9, 6.25)$, as in \cite{dong2009dynamic}. To generate large-scale time-dependent MNL instances, we calibrate them from the previously introduced \MMNL model (Equation \eqref{eq:mixmnl_exp}) rather than specifying time-varying parameters manually from scratch. At each time step, we choose the MNL parameters to approximate the \MMNL choice probabilities over a relevant price region. Since these probabilities determine expected demand, revenue, and resource consumption, we measure the approximation error using the Kullback--Leibler (KL) divergence \citep{kullback1997information}.

\subsection{KL Projection for Time-Dependent MNL Parameters}
\label{appendix:kl_projection}

We use KL projection for two purposes in our experimental design. First, it is used to construct the large-scale time-dependent MNL instances from the \MMNL model introduced above. In this case, the projected MNL model defines the ground-truth environment $\mathbb P$. Second, in the full-information \MMNL setting, we use the same procedure to construct the surrogate MNL model $\mathbb Q$ that approximates the true \MMNL environment $\mathbb P$.

For two probability vectors $\boldsymbol p$ and $\boldsymbol q$ over $\Products\cup\{0\}$, the KL divergence is defined as
\[
\KL(\boldsymbol p\,\|\,\boldsymbol q)
=
\sum_{i\in \Products\cup\{0\}}
p_i\log\frac{p_i}{q_i}.
\]

For each time step \(t\), let
\(p_t(\cdot\mid\boldsymbol r,\boldsymbol I_t)\)
denote the choice probabilities induced by the reference \MMNL model under price vector \(\boldsymbol r\) and inventory vector \(\boldsymbol I_t\), and let
\(q_t(\cdot\mid\boldsymbol r,\boldsymbol I_t;\boldsymbol a,\beta)\)
denote the corresponding probabilities under an MNL model with parameters
\((\boldsymbol a,\beta)\).

Given a reference distribution \(f_t\) over price vectors, the projected MNL parameters
\((\tilde{\boldsymbol a}_t,\tilde\beta_t)\)
are obtained by minimizing the expected KL divergence:
\begin{equation}
\label{eq:expected_kl_projection}
(\tilde{\boldsymbol a}_t,\tilde\beta_t)
\in
\arg\min_{\boldsymbol a,\beta}
\mathbb E_{\boldsymbol r\sim f_t}
\left[
\KL
\left(
p_t(\cdot\mid\boldsymbol r,\boldsymbol I_t)
\,\|\,
q_t(\cdot\mid\boldsymbol r,\boldsymbol I_t;\boldsymbol a,\beta)
\right)
\right].
\end{equation}

In practice, we approximate this expectation by sampling price vectors from \(f_t\) and minimizing the resulting empirical KL divergence using gradient descent. Repeating this procedure for each time step yields a time-dependent MNL model. Depending on the experimental setting, this projected model is used either as the ground-truth MNL environment $\mathbb P$ when constructing the large-scale MNL instances, or as the surrogate MNL model $\mathbb Q$ in the full-information \MMNL setting.

\newpage
\section{Pipeline Configuration Results} \label{app:model_selection}

We compare the performance of our architectures under different reference policies (Section \ref{app:reference_policy}), residual forms (Section \ref{app:residual_form}), and hinge-penalty settings (Section \ref{app:hinge_penality}). Each point in Figures \ref{fig:residual_reference_comparison} to \ref{fig:hinge_loss_comparison} reports the revenue gain relative to the validation-selected baseline (normalized to zero). For each model class, the displayed result corresponds to its best validation configuration.

\subsection{Reference Policy} \label{app:reference_policy}
The reference vector used in the residual framework is obtained either by averaging baseline outputs or by using the output of the validation-selected baseline.
    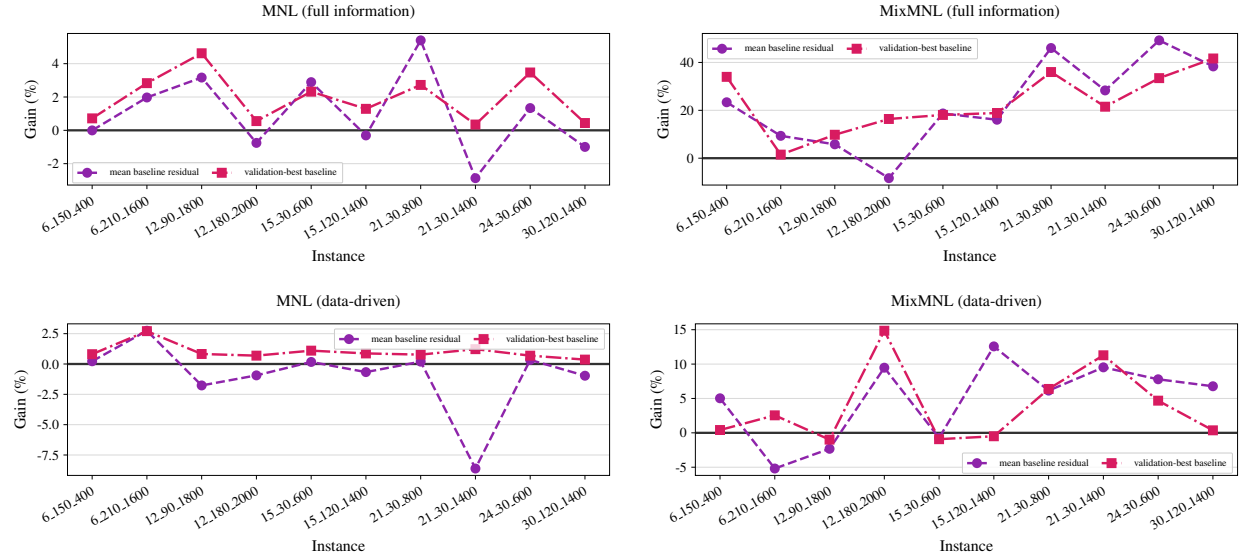
\begin{figure}[htbp]
    \centering
    \resizebox{\textwidth}{!}{%
      
\begin{tikzpicture}
\begin{scope}[shift={(0bp,250.8bp)}]
\begin{axis}[
  at={(50.5325bp,80.4058bp)},
  anchor=south west,
  scale only axis,
  width=468.227bp,
  height=130.922bp,
  xmin=-0.45000000000000001,
  xmax=9.4499999999999993,
  ymin=-3.2833881094209598,
  ymax=5.8126384510672979,
  enlargelimits=false,
  axis lines=box,
  axis line style={line width=0.8bp},
  tick align=outside,
  tick pos=left,
  major tick length=3.5bp,
  tick style={black,line width=0.8bp},
  scaled ticks=false,
  xtick={0,1,2,3,4,5,6,7,8,9},
  xticklabels={{6\_150\_400},{6\_210\_1600},{12\_90\_1800},{12\_180\_2000},{15\_30\_600},{15\_120\_1400},{21\_30\_800},{21\_30\_1400},{24\_30\_600},{30\_120\_1400}},
  ytick={-2,0,2,4},
  yticklabels={{-2},{0},{2},{4}},
  xticklabel style={font={\fontsize{12}{14.4}\selectfont},rotate=30,anchor=north east},
  yticklabel style={font={\fontsize{12}{14.4}\selectfont}},
  xlabel={Instance},
  ylabel={Gain (\%)},
  label style={font={\fontsize{14}{16.8}\selectfont}},
  title={MNL (full information)},
  title style={font={\fontsize{14}{16.8}\selectfont},at={(0.5,1)},anchor=south,yshift=3bp},
  ymajorgrids=true,
  xmajorgrids=false,
  grid style={gray!30,line width=0.8bp},
  axis background/.style={fill=white},
  unbounded coords=jump,
  legend columns=2,
  legend style={at={(0.00907678,0.149192)},anchor=north west,font={\fontsize{8.5}{10.2}\selectfont},draw=black!20,fill=white,fill opacity=0.8,text opacity=1,rounded corners=1bp,inner sep=3bp,column sep=8bp,/tikz/every even column/.append style={column sep=8bp}},
  legend cell align=left
]
\addplot[color={rgb,1:red,0.066666666666666666;green,0.066666666666666666;blue,0.066666666666666666},line width=2bp,mark=none,opacity=0.85,forget plot] coordinates {(-0.45000000000000001,0) (9.4499999999999993,0)};
\addplot[color={rgb,1:red,0.55686274509803924;green,0.14117647058823529;blue,0.66666666666666663},line width=2.1bp,dash pattern=on 7.77bp off 3.36bp,dash phase=0bp,mark=*,mark size=3.25bp,mark options={solid},opacity=1] coordinates {
  (0,-0.0046331110793207996)
  (1,1.9711845868181173)
  (2,3.1668074035655271)
  (3,-0.75349385952768466)
  (4,2.8922871774251644)
  (5,-0.30535515629596238)
  (6,5.3991826983178317)
  (7,-2.8699323566714936)
  (8,1.3361741481787763)
  (9,-0.99462713108938239)
};
\addlegendentry{mean baseline residual}
\addplot[color={rgb,1:red,0.84705882352941175;green,0.10588235294117647;blue,0.37647058823529411},line width=2.1bp,dash pattern=on 13.44bp off 3.36bp on 2.1bp off 3.36bp,dash phase=0bp,mark=square*,mark size=3.25bp,mark options={solid},opacity=1] coordinates {
  (0,0.71790824486209359)
  (1,2.8315992799356962)
  (2,4.6249170377633284)
  (3,0.55464288928501226)
  (4,2.3200245803404025)
  (5,1.2898573812769629)
  (6,2.7244218280583663)
  (7,0.34791870300651168)
  (8,3.4710353850600839)
  (9,0.43874715244628087)
};
\addlegendentry{validation-best baseline}
\end{axis}
\end{scope}
\begin{scope}[shift={(542.4bp,250.8bp)}]
\begin{axis}[
  at={(56.4725bp,80.4058bp)},
  anchor=south west,
  scale only axis,
  width=462.287bp,
  height=130.922bp,
  xmin=-0.45000000000000001,
  xmax=9.4499999999999993,
  ymin=-11.218178087038551,
  ymax=52.101671751976617,
  enlargelimits=false,
  axis lines=box,
  axis line style={line width=0.8bp},
  tick align=outside,
  tick pos=left,
  major tick length=3.5bp,
  tick style={black,line width=0.8bp},
  scaled ticks=false,
  xtick={0,1,2,3,4,5,6,7,8,9},
  xticklabels={{6\_150\_400},{6\_210\_1600},{12\_90\_1800},{12\_180\_2000},{15\_30\_600},{15\_120\_1400},{21\_30\_800},{21\_30\_1400},{24\_30\_600},{30\_120\_1400}},
  ytick={0,20,40},
  yticklabels={{0},{20},{40}},
  xticklabel style={font={\fontsize{12}{14.4}\selectfont},rotate=30,anchor=north east},
  yticklabel style={font={\fontsize{12}{14.4}\selectfont}},
  xlabel={Instance},
  ylabel={Gain (\%)},
  label style={font={\fontsize{14}{16.8}\selectfont}},
  title={\MMNL (full information)},
  title style={font={\fontsize{14}{16.8}\selectfont},at={(0.5,1)},anchor=south,yshift=3bp},
  ymajorgrids=true,
  xmajorgrids=false,
  grid style={gray!30,line width=0.8bp},
  axis background/.style={fill=white},
  unbounded coords=jump,
  legend columns=2,
  legend style={at={(0.00919341,0.967538)},anchor=north west,font={\fontsize{8.5}{10.2}\selectfont},draw=black!20,fill=white,fill opacity=0.8,text opacity=1,rounded corners=1bp,inner sep=3bp,column sep=8bp,/tikz/every even column/.append style={column sep=8bp}},
  legend cell align=left
]
\addplot[color={rgb,1:red,0.066666666666666666;green,0.066666666666666666;blue,0.066666666666666666},line width=2bp,mark=none,opacity=0.85,forget plot] coordinates {(-0.45000000000000001,0) (9.4499999999999993,0)};
\addplot[color={rgb,1:red,0.55686274509803924;green,0.14117647058823529;blue,0.66666666666666663},line width=2.1bp,dash pattern=on 7.77bp off 3.36bp,dash phase=0bp,mark=*,mark size=3.25bp,mark options={solid},opacity=1] coordinates {
  (0,23.362080619046694)
  (1,9.3357659071340215)
  (2,5.7917002447432164)
  (3,-8.3400030943560424)
  (4,18.680816144572635)
  (5,16.063273491497881)
  (6,46.026262613531287)
  (7,28.313748398003948)
  (8,49.223496759294108)
  (9,38.338788693488873)
};
\addlegendentry{mean baseline residual}
\addplot[color={rgb,1:red,0.84705882352941175;green,0.10588235294117647;blue,0.37647058823529411},line width=2.1bp,dash pattern=on 13.44bp off 3.36bp on 2.1bp off 3.36bp,dash phase=0bp,mark=square*,mark size=3.25bp,mark options={solid},opacity=1] coordinates {
  (0,33.977288929111957)
  (1,1.4947478437888966)
  (2,9.7905410933458086)
  (3,16.379393440822561)
  (4,18.069559651004411)
  (5,18.879668691024712)
  (6,35.989268717260224)
  (7,21.509088452988717)
  (8,33.445531191975576)
  (9,41.65358936670907)
};
\addlegendentry{validation-best baseline}
\end{axis}
\end{scope}
\begin{scope}[shift={(0bp,0bp)}]
\begin{axis}[
  at={(50.5325bp,80.4058bp)},
  anchor=south west,
  scale only axis,
  width=468.227bp,
  height=130.922bp,
  xmin=-0.45000000000000001,
  xmax=9.4499999999999993,
  ymin=-9.1821885230212921,
  ymax=3.3098651827357291,
  enlargelimits=false,
  axis lines=box,
  axis line style={line width=0.8bp},
  tick align=outside,
  tick pos=left,
  major tick length=3.5bp,
  tick style={black,line width=0.8bp},
  scaled ticks=false,
  xtick={0,1,2,3,4,5,6,7,8,9},
  xticklabels={{6\_150\_400},{6\_210\_1600},{12\_90\_1800},{12\_180\_2000},{15\_30\_600},{15\_120\_1400},{21\_30\_800},{21\_30\_1400},{24\_30\_600},{30\_120\_1400}},
  ytick={-7.5,-5,-2.5,0,2.5},
  yticklabels={{-7.5},{-5.0},{-2.5},{0.0},{2.5}},
  xticklabel style={font={\fontsize{12}{14.4}\selectfont},rotate=30,anchor=north east},
  yticklabel style={font={\fontsize{12}{14.4}\selectfont}},
  xlabel={Instance},
  ylabel={Gain (\%)},
  label style={font={\fontsize{14}{16.8}\selectfont}},
  title={MNL (data-driven)},
  title style={font={\fontsize{14}{16.8}\selectfont},at={(0.5,1)},anchor=south,yshift=3bp},
  ymajorgrids=true,
  xmajorgrids=false,
  grid style={gray!30,line width=0.8bp},
  axis background/.style={fill=white},
  unbounded coords=jump,
  legend columns=2,
  legend style={at={(0.990923,0.967538)},anchor=north east,font={\fontsize{8.5}{10.2}\selectfont},draw=black!20,fill=white,fill opacity=0.8,text opacity=1,rounded corners=1bp,inner sep=3bp,column sep=8bp,/tikz/every even column/.append style={column sep=8bp}},
  legend cell align=left
]
\addplot[color={rgb,1:red,0.066666666666666666;green,0.066666666666666666;blue,0.066666666666666666},line width=2bp,mark=none,opacity=0.85,forget plot] coordinates {(-0.45000000000000001,0) (9.4499999999999993,0)};
\addplot[color={rgb,1:red,0.55686274509803924;green,0.14117647058823529;blue,0.66666666666666663},line width=2.1bp,dash pattern=on 7.77bp off 3.36bp,dash phase=0bp,mark=*,mark size=3.25bp,mark options={solid},opacity=1] coordinates {
  (0,0.22559928997237891)
  (1,2.7420445597467733)
  (2,-1.7675803717482494)
  (3,-0.93701081262035724)
  (4,0.16867932312626729)
  (5,-0.66750340520118012)
  (6,0.19471124211763849)
  (7,-8.6143679000323363)
  (8,0.33702848852000961)
  (9,-0.9663687460418392)
};
\addlegendentry{mean baseline residual}
\addplot[color={rgb,1:red,0.84705882352941175;green,0.10588235294117647;blue,0.37647058823529411},line width=2.1bp,dash pattern=on 13.44bp off 3.36bp on 2.1bp off 3.36bp,dash phase=0bp,mark=square*,mark size=3.25bp,mark options={solid},opacity=1] coordinates {
  (0,0.8086321420534005)
  (1,2.7087057642945549)
  (2,0.82566651777045108)
  (3,0.68880742916839188)
  (4,1.1003942395513839)
  (5,0.87151917193581674)
  (6,0.77340993682643622)
  (7,1.2230984777418308)
  (8,0.68760429797854439)
  (9,0.36242827922492837)
};
\addlegendentry{validation-best baseline}
\end{axis}
\end{scope}
\begin{scope}[shift={(542.4bp,0bp)}]
\begin{axis}[
  at={(50.4425bp,80.4058bp)},
  anchor=south west,
  scale only axis,
  width=468.317bp,
  height=130.922bp,
  xmin=-0.45000000000000001,
  xmax=9.4499999999999993,
  ymin=-6.191938073681003,
  ymax=15.855096726165575,
  enlargelimits=false,
  axis lines=box,
  axis line style={line width=0.8bp},
  tick align=outside,
  tick pos=left,
  major tick length=3.5bp,
  tick style={black,line width=0.8bp},
  scaled ticks=false,
  xtick={0,1,2,3,4,5,6,7,8,9},
  xticklabels={{6\_150\_400},{6\_210\_1600},{12\_90\_1800},{12\_180\_2000},{15\_30\_600},{15\_120\_1400},{21\_30\_800},{21\_30\_1400},{24\_30\_600},{30\_120\_1400}},
  ytick={-5,0,5,10,15},
  yticklabels={{-5},{0},{5},{10},{15}},
  xticklabel style={font={\fontsize{12}{14.4}\selectfont},rotate=30,anchor=north east},
  yticklabel style={font={\fontsize{12}{14.4}\selectfont}},
  xlabel={Instance},
  ylabel={Gain (\%)},
  label style={font={\fontsize{14}{16.8}\selectfont}},
  title={\MMNL (data-driven)},
  title style={font={\fontsize{14}{16.8}\selectfont},at={(0.5,1)},anchor=south,yshift=3bp},
  ymajorgrids=true,
  xmajorgrids=false,
  grid style={gray!30,line width=0.8bp},
  axis background/.style={fill=white},
  unbounded coords=jump,
  legend columns=2,
  legend style={at={(0.990925,0.149192)},anchor=north east,font={\fontsize{8.5}{10.2}\selectfont},draw=black!20,fill=white,fill opacity=0.8,text opacity=1,rounded corners=1bp,inner sep=3bp,column sep=8bp,/tikz/every even column/.append style={column sep=8bp}},
  legend cell align=left
]
\addplot[color={rgb,1:red,0.066666666666666666;green,0.066666666666666666;blue,0.066666666666666666},line width=2bp,mark=none,opacity=0.85,forget plot] coordinates {(-0.45000000000000001,0) (9.4499999999999993,0)};
\addplot[color={rgb,1:red,0.55686274509803924;green,0.14117647058823529;blue,0.66666666666666663},line width=2.1bp,dash pattern=on 7.77bp off 3.36bp,dash phase=0bp,mark=*,mark size=3.25bp,mark options={solid},opacity=1] coordinates {
  (0,5.0258990170011044)
  (1,-5.189800128233431)
  (2,-2.3221340942401563)
  (3,9.4749741701195127)
  (4,-0.73165544203966093)
  (5,12.574022311785622)
  (6,6.1419514195106144)
  (7,9.5247805566123311)
  (8,7.7903829360329846)
  (9,6.7755930567049063)
};
\addlegendentry{mean baseline residual}
\addplot[color={rgb,1:red,0.84705882352941175;green,0.10588235294117647;blue,0.37647058823529411},line width=2.1bp,dash pattern=on 13.44bp off 3.36bp on 2.1bp off 3.36bp,dash phase=0bp,mark=square*,mark size=3.25bp,mark options={solid},opacity=1] coordinates {
  (0,0.41061749649488438)
  (1,2.5460890862890806)
  (2,-0.98590612194628002)
  (3,14.852958780718003)
  (4,-0.92286197268828796)
  (5,-0.49681740641006711)
  (6,6.3846392231433677)
  (7,11.280265041297694)
  (8,4.6810075150820758)
  (9,0.3648967498720781)
};
\addlegendentry{validation-best baseline}
\end{axis}
\end{scope}
\pgfresetboundingbox
\path[use as bounding box] (0,0) rectangle (1075.2bp,492bp);
\end{tikzpicture}

    }
\caption{\centering Validation performance with different references in the residual parameterization.}
\label{fig:residual_reference_comparison}
\end{figure}
Figure~\ref{fig:residual_reference_comparison} shows that neither reference policy dominates across all instances. Using the validation-best baseline as the reference is generally more robust, while the mean-baseline-output reference remains competitive in some \MMNL instances. We therefore treat the reference policy as a configuration choice and select it separately based on validation revenue.

\subsection{Residual Form}\label{app:residual_form}
We compare the additive and multiplicative residual specifications introduced in Section~\ref{subsec:statistical predictor}. As shown in Figure~\ref{fig:residual_form_comparison}, neither residual form dominates across all settings. The multiplicative specification performs particularly well in the \MMNL environments, especially under full information, whereas the additive specification is generally more robust in the MNL environments. We therefore treat the residual form as a configuration choice and select it separately  based on validation revenue.
    \begin{figure}[htbp]
    \centering
    \resizebox{\textwidth}{!}{%
        
\begin{tikzpicture}
\begin{scope}[shift={(0bp,250.8bp)}]
\begin{axis}[
  at={(50.5325bp,80.4058bp)},
  anchor=south west,
  scale only axis,
  width=468.227bp,
  height=130.922bp,
  xmin=-0.45000000000000001,
  xmax=9.4499999999999993,
  ymin=-3.0152162906698816,
  ymax=5.8801147225074937,
  enlargelimits=false,
  axis lines=box,
  axis line style={line width=0.8bp},
  tick align=outside,
  tick pos=left,
  major tick length=3.5bp,
  tick style={black,line width=0.8bp},
  scaled ticks=false,
  xtick={0,1,2,3,4,5,6,7,8,9},
  xticklabels={{6\_150\_400},{6\_210\_1600},{12\_90\_1800},{12\_180\_2000},{15\_30\_600},{15\_120\_1400},{21\_30\_800},{21\_30\_1400},{24\_30\_600},{30\_120\_1400}},
  ytick={-2,0,2,4},
  yticklabels={{-2},{0},{2},{4}},
  xticklabel style={font={\fontsize{12}{14.4}\selectfont},rotate=30,anchor=north east},
  yticklabel style={font={\fontsize{12}{14.4}\selectfont}},
  xlabel={Instance},
  ylabel={Gain (\%)},
  label style={font={\fontsize{14}{16.8}\selectfont}},
  title={MNL (full information)},
  title style={font={\fontsize{14}{16.8}\selectfont},at={(0.5,1)},anchor=south,yshift=3bp},
  ymajorgrids=true,
  xmajorgrids=false,
  grid style={gray!30,line width=0.8bp},
  axis background/.style={fill=white},
  unbounded coords=jump,
  legend columns=2,
  legend style={at={(0.990923,0.967538)},anchor=north east,font={\fontsize{8.5}{10.2}\selectfont},draw=black!20,fill=white,fill opacity=0.8,text opacity=1,rounded corners=1bp,inner sep=3bp,column sep=8bp,/tikz/every even column/.append style={column sep=8bp}},
  legend cell align=left
]
\addplot[color={rgb,1:red,0.066666666666666666;green,0.066666666666666666;blue,0.066666666666666666},line width=2bp,mark=none,opacity=0.85,forget plot] coordinates {(-0.45000000000000001,0) (9.4499999999999993,0)};
\addplot[color={rgb,1:red,0.77647058823529413;green,0.15686274509803921;blue,0.15686274509803921},line width=2.1bp,dash pattern=on 7.77bp off 3.36bp,dash phase=0bp,mark=*,mark size=3.25bp,mark options={solid},opacity=1] coordinates {
  (0,0.71790824486209359)
  (1,3.6518691317687337)
  (2,3.5203281158797561)
  (3,1.237413268825595)
  (4,2.3782753121877582)
  (5,1.7852252423910258)
  (6,4.8534842241232923)
  (7,0.34791870300651168)
  (8,3.4710353850600839)
  (9,0.43833509813855748)
};
\addlegendentry{additive}
\addplot[color={rgb,1:red,0;green,0.51372549019607838;blue,0.5607843137254902},line width=2.1bp,dash pattern=on 13.44bp off 3.36bp on 2.1bp off 3.36bp,dash phase=0bp,mark=square*,mark size=3.25bp,mark options={solid},opacity=1] coordinates {
  (0,0.71790824486209359)
  (1,2.4387422092013518)
  (2,5.250959629562991)
  (3,0.55464288928501226)
  (4,2.8922871774251644)
  (5,1.0150500338293122)
  (6,5.475781494635795)
  (7,-2.6108830627981829)
  (8,2.5904472301126611)
  (9,0.43874715244628087)
};
\addlegendentry{multiplicative}
\end{axis}
\end{scope}
\begin{scope}[shift={(542.4bp,250.8bp)}]
\begin{axis}[
  at={(49.7225bp,80.4058bp)},
  anchor=south west,
  scale only axis,
  width=469.037bp,
  height=130.922bp,
  xmin=-0.45000000000000001,
  xmax=9.4499999999999993,
  ymin=-2.4611748379647054,
  ymax=51.684671597258813,
  enlargelimits=false,
  axis lines=box,
  axis line style={line width=0.8bp},
  tick align=outside,
  tick pos=left,
  major tick length=3.5bp,
  tick style={black,line width=0.8bp},
  scaled ticks=false,
  xtick={0,1,2,3,4,5,6,7,8,9},
  xticklabels={{6\_150\_400},{6\_210\_1600},{12\_90\_1800},{12\_180\_2000},{15\_30\_600},{15\_120\_1400},{21\_30\_800},{21\_30\_1400},{24\_30\_600},{30\_120\_1400}},
  ytick={0,20,40},
  yticklabels={{0},{20},{40}},
  xticklabel style={font={\fontsize{12}{14.4}\selectfont},rotate=30,anchor=north east},
  yticklabel style={font={\fontsize{12}{14.4}\selectfont}},
  xlabel={Instance},
  ylabel={Gain (\%)},
  label style={font={\fontsize{14}{16.8}\selectfont}},
  title={\MMNL (full information)},
  title style={font={\fontsize{14}{16.8}\selectfont},at={(0.5,1)},anchor=south,yshift=3bp},
  ymajorgrids=true,
  xmajorgrids=false,
  grid style={gray!30,line width=0.8bp},
  axis background/.style={fill=white},
  unbounded coords=jump,
  legend columns=2,
  legend style={at={(0.00906111,0.967538)},anchor=north west,font={\fontsize{8.5}{10.2}\selectfont},draw=black!20,fill=white,fill opacity=0.8,text opacity=1,rounded corners=1bp,inner sep=3bp,column sep=8bp,/tikz/every even column/.append style={column sep=8bp}},
  legend cell align=left
]
\addplot[color={rgb,1:red,0.066666666666666666;green,0.066666666666666666;blue,0.066666666666666666},line width=2bp,mark=none,opacity=0.85,forget plot] coordinates {(-0.45000000000000001,0) (9.4499999999999993,0)};
\addplot[color={rgb,1:red,0.77647058823529413;green,0.15686274509803921;blue,0.15686274509803921},line width=2.1bp,dash pattern=on 7.77bp off 3.36bp,dash phase=0bp,mark=*,mark size=3.25bp,mark options={solid},opacity=1] coordinates {
  (0,33.977288929111957)
  (1,4.4192999656204028)
  (2,1.157427679749309)
  (3,16.379393440822561)
  (4,2.1234415764505421)
  (5,13.56544577846531)
  (6,12.045025929711882)
  (7,9.3737411465003397)
  (8,11.119344639472979)
  (9,21.8243591614353)
};
\addlegendentry{additive}
\addplot[color={rgb,1:red,0;green,0.51372549019607838;blue,0.5607843137254902},line width=2.1bp,dash pattern=on 13.44bp off 3.36bp on 2.1bp off 3.36bp,dash phase=0bp,mark=square*,mark size=3.25bp,mark options={solid},opacity=1] coordinates {
  (0,28.139386986672815)
  (1,9.3357659071340215)
  (2,9.7905410933458086)
  (3,6.7906163781687132)
  (4,18.680816144572635)
  (5,18.879668691024712)
  (6,46.026262613531287)
  (7,28.313748398003948)
  (8,49.223496759294108)
  (9,41.65358936670907)
};
\addlegendentry{multiplicative}
\end{axis}
\end{scope}
\begin{scope}[shift={(0bp,0bp)}]
\begin{axis}[
  at={(65.3825bp,80.4058bp)},
  anchor=south west,
  scale only axis,
  width=453.377bp,
  height=130.922bp,
  xmin=-0.45000000000000001,
  xmax=9.4499999999999993,
  ymin=-10.986086987604489,
  ymax=4.9238763925068945,
  enlargelimits=false,
  axis lines=box,
  axis line style={line width=0.8bp},
  tick align=outside,
  tick pos=left,
  major tick length=3.5bp,
  tick style={black,line width=0.8bp},
  scaled ticks=false,
  xtick={0,1,2,3,4,5,6,7,8,9},
  xticklabels={{6\_150\_400},{6\_210\_1600},{12\_90\_1800},{12\_180\_2000},{15\_30\_600},{15\_120\_1400},{21\_30\_800},{21\_30\_1400},{24\_30\_600},{30\_120\_1400}},
  ytick={-10,-5,0},
  yticklabels={{-10},{-5},{0}},
  xticklabel style={font={\fontsize{12}{14.4}\selectfont},rotate=30,anchor=north east},
  yticklabel style={font={\fontsize{12}{14.4}\selectfont}},
  xlabel={Instance},
  ylabel={Gain (\%)},
  label style={font={\fontsize{14}{16.8}\selectfont}},
  title={MNL (data-driven)},
  title style={font={\fontsize{14}{16.8}\selectfont},at={(0.5,1)},anchor=south,yshift=3bp},
  ymajorgrids=true,
  xmajorgrids=false,
  grid style={gray!30,line width=0.8bp},
  axis background/.style={fill=white},
  unbounded coords=jump,
  legend columns=2,
  legend style={at={(0.990626,0.967538)},anchor=north east,font={\fontsize{8.5}{10.2}\selectfont},draw=black!20,fill=white,fill opacity=0.8,text opacity=1,rounded corners=1bp,inner sep=3bp,column sep=8bp,/tikz/every even column/.append style={column sep=8bp}},
  legend cell align=left
]
\addplot[color={rgb,1:red,0.066666666666666666;green,0.066666666666666666;blue,0.066666666666666666},line width=2bp,mark=none,opacity=0.85,forget plot] coordinates {(-0.45000000000000001,0) (9.4499999999999993,0)};
\addplot[color={rgb,1:red,0.77647058823529413;green,0.15686274509803921;blue,0.15686274509803921},line width=2.1bp,dash pattern=on 7.77bp off 3.36bp,dash phase=0bp,mark=*,mark size=3.25bp,mark options={solid},opacity=1] coordinates {
  (0,0.8086321420534005)
  (1,4.2006962388654676)
  (2,0.82566651777045108)
  (3,0.52099712750792215)
  (4,1.0165545037594601)
  (5,0.87151917193581674)
  (6,0.63097669351282815)
  (7,1.2230984777418308)
  (8,0.67872866696619749)
  (9,0.36212488201717341)
};
\addlegendentry{additive}
\addplot[color={rgb,1:red,0;green,0.51372549019607838;blue,0.5607843137254902},line width=2.1bp,dash pattern=on 13.44bp off 3.36bp on 2.1bp off 3.36bp,dash phase=0bp,mark=square*,mark size=3.25bp,mark options={solid},opacity=1] coordinates {
  (0,0.83037647710698692)
  (1,3.3513948346139397)
  (2,0.81069129490530045)
  (3,0.68880742916839188)
  (4,1.1003942395513839)
  (5,0.77044744975216684)
  (6,0.77340993682643622)
  (7,-10.262906833963063)
  (8,0.68760429797854439)
  (9,0.36242827922492837)
};
\addlegendentry{multiplicative}
\end{axis}
\end{scope}
\begin{scope}[shift={(542.4bp,0bp)}]
\begin{axis}[
  at={(49.7225bp,80.4058bp)},
  anchor=south west,
  scale only axis,
  width=469.037bp,
  height=130.922bp,
  xmin=-0.45000000000000001,
  xmax=9.4499999999999993,
  ymin=-2.8067556675196963,
  ymax=15.693897563967418,
  enlargelimits=false,
  axis lines=box,
  axis line style={line width=0.8bp},
  tick align=outside,
  tick pos=left,
  major tick length=3.5bp,
  tick style={black,line width=0.8bp},
  scaled ticks=false,
  xtick={0,1,2,3,4,5,6,7,8,9},
  xticklabels={{6\_150\_400},{6\_210\_1600},{12\_90\_1800},{12\_180\_2000},{15\_30\_600},{15\_120\_1400},{21\_30\_800},{21\_30\_1400},{24\_30\_600},{30\_120\_1400}},
  ytick={0,5,10,15},
  yticklabels={{0},{5},{10},{15}},
  xticklabel style={font={\fontsize{12}{14.4}\selectfont},rotate=30,anchor=north east},
  yticklabel style={font={\fontsize{12}{14.4}\selectfont}},
  xlabel={Instance},
  ylabel={Gain (\%)},
  label style={font={\fontsize{14}{16.8}\selectfont}},
  title={\MMNL (data-driven)},
  title style={font={\fontsize{14}{16.8}\selectfont},at={(0.5,1)},anchor=south,yshift=3bp},
  ymajorgrids=true,
  xmajorgrids=false,
  grid style={gray!30,line width=0.8bp},
  axis background/.style={fill=white},
  unbounded coords=jump,
  legend columns=2,
  legend style={at={(0.990939,0.967538)},anchor=north east,font={\fontsize{8.5}{10.2}\selectfont},draw=black!20,fill=white,fill opacity=0.8,text opacity=1,rounded corners=1bp,inner sep=3bp,column sep=8bp,/tikz/every even column/.append style={column sep=8bp}},
  legend cell align=left
]
\addplot[color={rgb,1:red,0.066666666666666666;green,0.066666666666666666;blue,0.066666666666666666},line width=2bp,mark=none,opacity=0.85,forget plot] coordinates {(-0.45000000000000001,0) (9.4499999999999993,0)};
\addplot[color={rgb,1:red,0.77647058823529413;green,0.15686274509803921;blue,0.15686274509803921},line width=2.1bp,dash pattern=on 7.77bp off 3.36bp,dash phase=0bp,mark=*,mark size=3.25bp,mark options={solid},opacity=1] coordinates {
  (0,10.816237355800071)
  (1,2.5460890862890806)
  (2,-0.98590612194628002)
  (3,9.4749741701195127)
  (4,-0.41108074367543651)
  (5,12.574022311785622)
  (6,6.3846392231433677)
  (7,11.280265041297694)
  (8,7.7903829360329846)
  (9,6.7755930567049063)
};
\addlegendentry{additive}
\addplot[color={rgb,1:red,0;green,0.51372549019607838;blue,0.5607843137254902},line width=2.1bp,dash pattern=on 13.44bp off 3.36bp on 2.1bp off 3.36bp,dash phase=0bp,mark=square*,mark size=3.25bp,mark options={solid},opacity=1] coordinates {
  (0,6.4206748759871566)
  (1,1.2161661069109091)
  (2,-1.9658168842702819)
  (3,14.852958780718003)
  (4,-1.0577720697705475)
  (5,11.733216505790072)
  (6,6.1447511988503676)
  (7,10.931047513811118)
  (8,7.5453159260418312)
  (9,5.9965977591765194)
};
\addlegendentry{multiplicative}
\end{axis}
\end{scope}
\pgfresetboundingbox
\path[use as bounding box] (0,0) rectangle (1075.2bp,492bp);
\end{tikzpicture}

    }
\caption{\centering Validation performance of additive and multiplicative residual specifications.}
\label{fig:residual_form_comparison}
\end{figure}
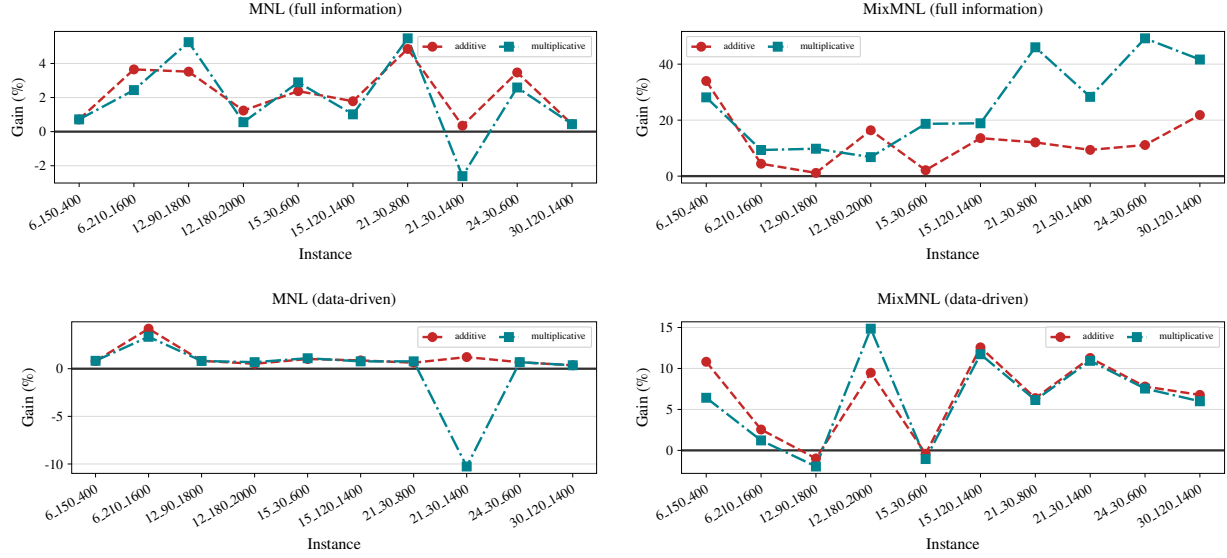

\subsection{Chosen-Item Squared Hinge}\label{app:hinge_penality}
For \PNN, we observe that adding a chosen-item squared hinge penalty to the training objective improves performance on two instances. Therefore, we include it as an optional loss component. Specifically, we consider
\[
\mathcal{L}'
=
\mathcal{L}
+
\lambda_{\mathrm{hinge}} \times
\max\left\{
0.5-\left(r-r_{\mathrm{reference}}\right),\,0
\right\}^{2},
\qquad
\lambda_{\mathrm{hinge}}\in\{0,1\}.
\]
which encourages the learned price of the purchased item to exceed its reference price by a margin of at least 0.5. Here, $\lambda_{\mathrm{hinge}}$ controls whether the chosen-item squared hinge penalty is included. The penalty is applied only to samples that contain a purchased item as no-purchase samples make no contribution to this term.
Figure~\ref{fig:hinge_loss_comparison} compares training with and without the chosen-item squared hinge penalty. In the full-information setting, adding the chosen-item squared hinge penalty generally does not improve the best-performing configuration for either MNL or \MMNL instances. In contrast, under the data-driven setting, the chosen-item squared hinge penalty provides improvements for two instances. We therefore include the chosen-item squared hinge as an optional component of the \PNN training objective and select whether to use it during the validation process.

    \begin{figure}[htbp]
    \centering
    \resizebox{\textwidth}{!}{%
        
\begin{tikzpicture}
\begin{scope}[shift={(0bp,250.8bp)}]
\begin{axis}[
  at={(50.5325bp,80.4058bp)},
  anchor=south west,
  scale only axis,
  width=468.227bp,
  height=130.922bp,
  xmin=-0.45000000000000001,
  xmax=9.4499999999999993,
  ymin=-6.1491025330749851,
  ymax=6.0293474007172607,
  enlargelimits=false,
  axis lines=box,
  axis line style={line width=0.8bp},
  tick align=outside,
  tick pos=left,
  major tick length=3.5bp,
  tick style={black,line width=0.8bp},
  scaled ticks=false,
  xtick={0,1,2,3,4,5,6,7,8,9},
  xticklabels={{6\_150\_400},{6\_210\_1600},{12\_90\_1800},{12\_180\_2000},{15\_30\_600},{15\_120\_1400},{21\_30\_800},{21\_30\_1400},{24\_30\_600},{30\_120\_1400}},
  ytick={-5,-2.5,0,2.5,5},
  yticklabels={{-5.0},{-2.5},{0.0},{2.5},{5.0}},
  xticklabel style={font={\fontsize{12}{14.4}\selectfont},rotate=30,anchor=north east},
  yticklabel style={font={\fontsize{12}{14.4}\selectfont}},
  xlabel={Instance},
  ylabel={Gain (\%)},
  label style={font={\fontsize{14}{16.8}\selectfont}},
  title={MNL (full information)},
  title style={font={\fontsize{14}{16.8}\selectfont},at={(0.5,1)},anchor=south,yshift=3bp},
  ymajorgrids=true,
  xmajorgrids=false,
  grid style={gray!30,line width=0.8bp},
  axis background/.style={fill=white},
  unbounded coords=jump,
  legend columns=2,
  legend style={at={(0.00907678,0.150395)},anchor=north west,font={\fontsize{8.5}{10.2}\selectfont},draw=black!20,fill=white,fill opacity=0.8,text opacity=1,rounded corners=1bp,inner sep=3bp,column sep=8bp,/tikz/every even column/.append style={column sep=8bp}},
  legend cell align=left
]
\addplot[color={rgb,1:red,0.066666666666666666;green,0.066666666666666666;blue,0.066666666666666666},line width=2bp,mark=none,opacity=0.85,forget plot] coordinates {(-0.45000000000000001,0) (9.4499999999999993,0)};
\addplot[color={rgb,1:red,0.1803921568627451;green,0.49019607843137253;blue,0.19607843137254902},line width=2.1bp,dash pattern=on 7.77bp off 3.36bp,dash phase=0bp,mark=*,mark size=3.25bp,mark options={solid},opacity=1] coordinates {
  (0,0.71790824486209359)
  (1,3.6518691317687337)
  (2,5.250959629562991)
  (3,1.237413268825595)
  (4,2.8922871774251644)
  (5,1.7852252423910258)
  (6,5.475781494635795)
  (7,0.34791870300651168)
  (8,3.4710353850600839)
  (9,0.43874715244628087)
};
\addlegendentry{Without penalty}
\addplot[color={rgb,1:red,0.77647058823529413;green,0.15686274509803921;blue,0.15686274509803921},line width=2.1bp,dash pattern=on 13.44bp off 3.36bp on 2.1bp off 3.36bp,dash phase=0bp,mark=square*,mark size=3.25bp,mark options={solid},opacity=1] coordinates {
  (0,-0.0046331110793207996)
  (1,-1.4566917002715112)
  (2,-1.976071023516544)
  (3,-0.57601259222857193)
  (4,-1.4067920976718149)
  (5,-0.4666562255257895)
  (6,-1.2569277517647457)
  (7,-5.5955366269935194)
  (8,-0.95256596058259924)
  (9,-0.55224345761617699)
};
\addlegendentry{With penalty}
\end{axis}
\end{scope}
\begin{scope}[shift={(542.4bp,250.8bp)}]
\begin{axis}[
  at={(56.4725bp,80.4058bp)},
  anchor=south west,
  scale only axis,
  width=462.287bp,
  height=130.922bp,
  xmin=-0.45000000000000001,
  xmax=9.4499999999999993,
  ymin=-13.494951475754844,
  ymax=52.210089532391677,
  enlargelimits=false,
  axis lines=box,
  axis line style={line width=0.8bp},
  tick align=outside,
  tick pos=left,
  major tick length=3.5bp,
  tick style={black,line width=0.8bp},
  scaled ticks=false,
  xtick={0,1,2,3,4,5,6,7,8,9},
  xticklabels={{6\_150\_400},{6\_210\_1600},{12\_90\_1800},{12\_180\_2000},{15\_30\_600},{15\_120\_1400},{21\_30\_800},{21\_30\_1400},{24\_30\_600},{30\_120\_1400}},
  ytick={0,20,40},
  yticklabels={{0},{20},{40}},
  xticklabel style={font={\fontsize{12}{14.4}\selectfont},rotate=30,anchor=north east},
  yticklabel style={font={\fontsize{12}{14.4}\selectfont}},
  xlabel={Instance},
  ylabel={Gain (\%)},
  label style={font={\fontsize{14}{16.8}\selectfont}},
  title={\MMNL (full information)},
  title style={font={\fontsize{14}{16.8}\selectfont},at={(0.5,1)},anchor=south,yshift=3bp},
  ymajorgrids=true,
  xmajorgrids=false,
  grid style={gray!30,line width=0.8bp},
  axis background/.style={fill=white},
  unbounded coords=jump,
  legend columns=2,
  legend style={at={(0.00919341,0.967538)},anchor=north west,font={\fontsize{8.5}{10.2}\selectfont},draw=black!20,fill=white,fill opacity=0.8,text opacity=1,rounded corners=1bp,inner sep=3bp,column sep=8bp,/tikz/every even column/.append style={column sep=8bp}},
  legend cell align=left
]
\addplot[color={rgb,1:red,0.066666666666666666;green,0.066666666666666666;blue,0.066666666666666666},line width=2bp,mark=none,opacity=0.85,forget plot] coordinates {(-0.45000000000000001,0) (9.4499999999999993,0)};
\addplot[color={rgb,1:red,0.1803921568627451;green,0.49019607843137253;blue,0.19607843137254902},line width=2.1bp,dash pattern=on 7.77bp off 3.36bp,dash phase=0bp,mark=*,mark size=3.25bp,mark options={solid},opacity=1] coordinates {
  (0,33.977288929111957)
  (1,9.3357659071340215)
  (2,9.7905410933458086)
  (3,16.379393440822561)
  (4,18.680816144572635)
  (5,18.879668691024712)
  (6,46.026262613531287)
  (7,28.313748398003948)
  (8,49.223496759294108)
  (9,41.65358936670907)
};
\addlegendentry{Without penalty}
\addplot[color={rgb,1:red,0.77647058823529413;green,0.15686274509803921;blue,0.15686274509803921},line width=2.1bp,dash pattern=on 13.44bp off 3.36bp on 2.1bp off 3.36bp,dash phase=0bp,mark=square*,mark size=3.25bp,mark options={solid},opacity=1] coordinates {
  (0,-10.508358702657274)
  (1,-2.2518416423830083)
  (2,1.0166532171777529)
  (3,-5.389341797379215)
  (4,-1.7308732279808483)
  (5,-10.029068646055524)
  (6,2.9753949154894652)
  (7,23.873942078652394)
  (8,-8.7251659390530385)
  (9,-6.6582494318752046)
};
\addlegendentry{With penalty}
\end{axis}
\end{scope}
\begin{scope}[shift={(0bp,0bp)}]
\begin{axis}[
  at={(50.5325bp,80.4058bp)},
  anchor=south west,
  scale only axis,
  width=468.227bp,
  height=130.922bp,
  xmin=-0.45000000000000001,
  xmax=9.4499999999999993,
  ymin=-2.9246220326753969,
  ymax=4.5399971089388425,
  enlargelimits=false,
  axis lines=box,
  axis line style={line width=0.8bp},
  tick align=outside,
  tick pos=left,
  major tick length=3.5bp,
  tick style={black,line width=0.8bp},
  scaled ticks=false,
  xtick={0,1,2,3,4,5,6,7,8,9},
  xticklabels={{6\_150\_400},{6\_210\_1600},{12\_90\_1800},{12\_180\_2000},{15\_30\_600},{15\_120\_1400},{21\_30\_800},{21\_30\_1400},{24\_30\_600},{30\_120\_1400}},
  ytick={-2,0,2,4},
  yticklabels={{-2},{0},{2},{4}},
  xticklabel style={font={\fontsize{12}{14.4}\selectfont},rotate=30,anchor=north east},
  yticklabel style={font={\fontsize{12}{14.4}\selectfont}},
  xlabel={Instance},
  ylabel={Gain (\%)},
  label style={font={\fontsize{14}{16.8}\selectfont}},
  title={MNL (data-driven)},
  title style={font={\fontsize{14}{16.8}\selectfont},at={(0.5,1)},anchor=south,yshift=3bp},
  ymajorgrids=true,
  xmajorgrids=false,
  grid style={gray!30,line width=0.8bp},
  axis background/.style={fill=white},
  unbounded coords=jump,
  legend columns=2,
  legend style={at={(0.990923,0.967538)},anchor=north east,font={\fontsize{8.5}{10.2}\selectfont},draw=black!20,fill=white,fill opacity=0.8,text opacity=1,rounded corners=1bp,inner sep=3bp,column sep=8bp,/tikz/every even column/.append style={column sep=8bp}},
  legend cell align=left
]
\addplot[color={rgb,1:red,0.066666666666666666;green,0.066666666666666666;blue,0.066666666666666666},line width=2bp,mark=none,opacity=0.85,forget plot] coordinates {(-0.45000000000000001,0) (9.4499999999999993,0)};
\addplot[color={rgb,1:red,0.1803921568627451;green,0.49019607843137253;blue,0.19607843137254902},line width=2.1bp,dash pattern=on 7.77bp off 3.36bp,dash phase=0bp,mark=*,mark size=3.25bp,mark options={solid},opacity=1] coordinates {
  (0,0.83037647710698692)
  (1,4.2006962388654676)
  (2,0.82566651777045108)
  (3,0.68880742916839188)
  (4,1.1003942395513839)
  (5,0.87151917193581674)
  (6,0.77340993682643622)
  (7,-2.5853211626020225)
  (8,0.68760429797854439)
  (9,0.36242827922492837)
};
\addlegendentry{Without penalty}
\addplot[color={rgb,1:red,0.77647058823529413;green,0.15686274509803921;blue,0.15686274509803921},line width=2.1bp,dash pattern=on 13.44bp off 3.36bp on 2.1bp off 3.36bp,dash phase=0bp,mark=square*,mark size=3.25bp,mark options={solid},opacity=1] coordinates {
  (0,0.3923744904862847)
  (1,-0.95840297927996243)
  (2,-1.0065717439892834)
  (3,-0.010091053853868)
  (4,-0.41365654453012729)
  (5,-0.15526866807792919)
  (6,-0.20775133256709299)
  (7,1.2230984777418308)
  (8,-0.6011673887383554)
  (9,-0.51418003129333179)
};
\addlegendentry{With penalty}
\end{axis}
\end{scope}
\begin{scope}[shift={(542.4bp,0bp)}]
\begin{axis}[
  at={(50.4425bp,80.4058bp)},
  anchor=south west,
  scale only axis,
  width=468.317bp,
  height=130.922bp,
  xmin=-0.45000000000000001,
  xmax=9.4499999999999993,
  ymin=-5.8882564462047116,
  ymax=15.840635696285752,
  enlargelimits=false,
  axis lines=box,
  axis line style={line width=0.8bp},
  tick align=outside,
  tick pos=left,
  major tick length=3.5bp,
  tick style={black,line width=0.8bp},
  scaled ticks=false,
  xtick={0,1,2,3,4,5,6,7,8,9},
  xticklabels={{6\_150\_400},{6\_210\_1600},{12\_90\_1800},{12\_180\_2000},{15\_30\_600},{15\_120\_1400},{21\_30\_800},{21\_30\_1400},{24\_30\_600},{30\_120\_1400}},
  ytick={-5,0,5,10,15},
  yticklabels={{-5},{0},{5},{10},{15}},
  xticklabel style={font={\fontsize{12}{14.4}\selectfont},rotate=30,anchor=north east},
  yticklabel style={font={\fontsize{12}{14.4}\selectfont}},
  xlabel={Instance},
  ylabel={Gain (\%)},
  label style={font={\fontsize{14}{16.8}\selectfont}},
  title={\MMNL (data-driven)},
  title style={font={\fontsize{14}{16.8}\selectfont},at={(0.5,1)},anchor=south,yshift=3bp},
  ymajorgrids=true,
  xmajorgrids=false,
  grid style={gray!30,line width=0.8bp},
  axis background/.style={fill=white},
  unbounded coords=jump,
  legend columns=2,
  legend style={at={(0.990925,0.967538)},anchor=north east,font={\fontsize{8.5}{10.2}\selectfont},draw=black!20,fill=white,fill opacity=0.8,text opacity=1,rounded corners=1bp,inner sep=3bp,column sep=8bp,/tikz/every even column/.append style={column sep=8bp}},
  legend cell align=left
]
\addplot[color={rgb,1:red,0.066666666666666666;green,0.066666666666666666;blue,0.066666666666666666},line width=2bp,mark=none,opacity=0.85,forget plot] coordinates {(-0.45000000000000001,0) (9.4499999999999993,0)};
\addplot[color={rgb,1:red,0.1803921568627451;green,0.49019607843137253;blue,0.19607843137254902},line width=2.1bp,dash pattern=on 7.77bp off 3.36bp,dash phase=0bp,mark=*,mark size=3.25bp,mark options={solid},opacity=1] coordinates {
  (0,10.816237355800071)
  (1,2.3509493453757706)
  (2,-0.98590612194628002)
  (3,14.852958780718003)
  (4,-0.41108074367543651)
  (5,12.574022311785622)
  (6,6.3846392231433677)
  (7,9.6051459740672325)
  (8,7.7903829360329846)
  (9,6.6973755975946716)
};
\addlegendentry{Without penalty}
\addplot[color={rgb,1:red,0.77647058823529413;green,0.15686274509803921;blue,0.15686274509803921},line width=2.1bp,dash pattern=on 13.44bp off 3.36bp on 2.1bp off 3.36bp,dash phase=0bp,mark=square*,mark size=3.25bp,mark options={solid},opacity=1] coordinates {
  (0,-4.9005795306369633)
  (1,2.5460890862890806)
  (2,-4.1771564394852696)
  (3,9.4749741701195127)
  (4,-0.99222081351184765)
  (5,8.9660601779860318)
  (6,3.4663603118481019)
  (7,11.280265041297694)
  (8,7.6566741528120579)
  (9,6.7755930567049063)
};
\addlegendentry{With penalty}
\end{axis}
\end{scope}
\pgfresetboundingbox
\path[use as bounding box] (0,0) rectangle (1075.2bp,492bp);
\end{tikzpicture}

    }
\caption{\centering Validation performance of \PNN with and without chosen-item squared hinge penalty.}
\label{fig:hinge_loss_comparison}
\end{figure}
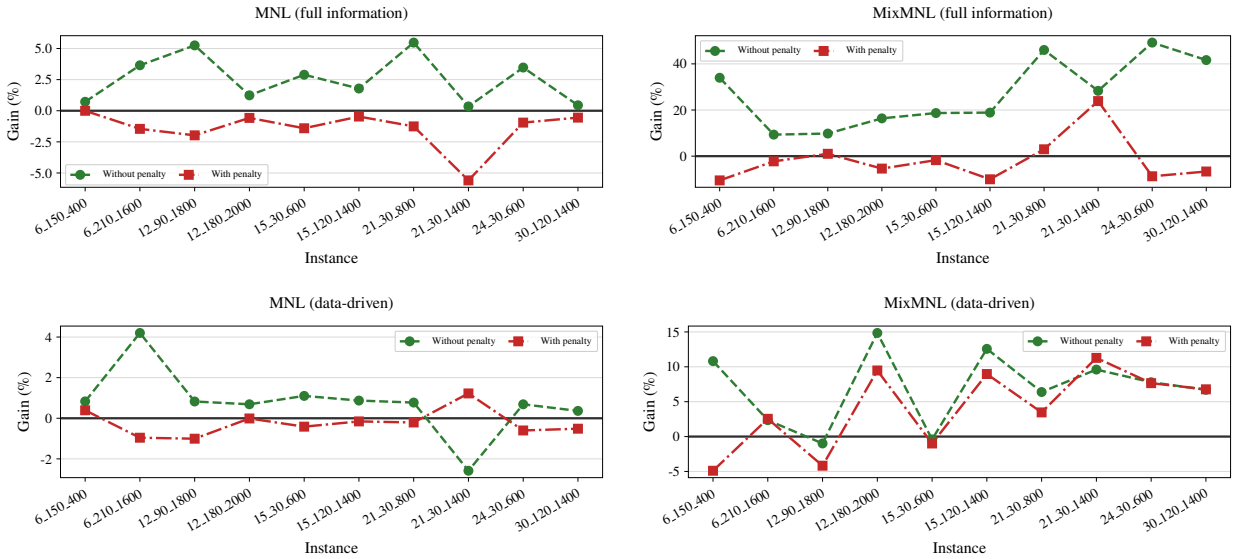

\newpage
\section{Detailed Description of the DP-Based Benchmarks and Features}
\label{app:benchmark_policies}
\subsection{Independent-Itinerary Pricing}\label{app:Indenpenden_itinerary}
For each itinerary $i$, the purchase probability follows the two-stage binary logit model, which is defined as follows:
\begin{equation}
    \mathbb{P}_{\mathrm{Indep}}(i|\boldsymbol{r}_t)
    =
    \alpha_i
    \frac{
        \exp\!\left( a_{i}-\beta r_{i,t}\right)
    }{
        1+\exp\!\left(a_{i}-\beta r_{i,t}\right)
    },
    \label{eq:itinerary_binary_logit}
\end{equation}
where $\alpha_i$ is the probability that itinerary $i$ enters the customer's consideration set, and $a_{i}$ and $\beta$ are the corresponding quality index and price sensitivity. $\mathbb{P}_{\mathrm{Indep}}$ is obtained by minimizing the KL divergence from the known true environment $\bbP$ in the full-information setting, and by maximizing the likelihood of the historical data in the data-driven setting. By allowing $a_i$ and $\beta$ to be either time-varying or time-invariant, we obtain two variants of $\mathbb{P}_{\mathrm{Indep}}$: a time-dependent version and a time-independent version.

To obtain the pricing decisions, we solve an independent dynamic program for each itinerary. The obtained single-itinerary opportunity costs are then used to define other baselines.

\subsection{Joint Pricing with Product-Specific Opportunity Costs}\label{app:JoPri}
The \textit{JoPri} policy inserts the single-itinerary opportunity costs
$\bar o_{i,t}(\boldsymbol I_t)$ directly into the MNL pricing structure of
\citet{dong2009dynamic}:
\begin{equation}
    r_i^{\mathrm{raw}}(\state)
    =
    \bar o_{i,t}(\boldsymbol I_t)
    +
    \frac{m^{\mathrm{raw}}(\state)}{\beta},
    \qquad i\in \mathcal{A} (\boldsymbol I_t),
    \label{eq:joint_raw_policy}
\end{equation}
where $m^{\mathrm{raw}}(\state)$ is the corresponding common MNL markup. We use the same MNL surrogate model \(\mathbb Q\) as in our architectures, calibrated as described in Section \ref{sec:settings_instances}. Allowing its parameters to vary over time or holding them fixed yields time-dependent and time-independent variants.

\subsection{Unified Pricing with Aggregated Opportunity Costs}
\label{app:JoComPri}

The \textit{JoComPri} policy adapts the unified dynamic pricing rule of \citet{dong2009dynamic}.
It first aggregates the single-itinerary opportunity costs into an attractiveness-weighted opportunity cost:
\begin{equation}
    \bar o_t^{w}(\boldsymbol I_t)
    =
    \frac{
        \sum_{i\in \mathcal{A}(\boldsymbol I_t)}
        \exp(a_{i})\,
        \bar o_{i,t}(\boldsymbol I_t)
    }{
        \sum_{j\in \mathcal{A}(\boldsymbol I_t)}
        \exp(a_{j})
    }.
    \nonumber
\end{equation}
Thus, $\bar o_t^{w}(\boldsymbol I_t)$ is a scalar representing the attractiveness-weighted opportunity cost across all available products. Here, we use superscript \(w\) to denote attractiveness weighting.

Following the unified pricing structure of \citet{dong2009dynamic}, we then
assign the same price to all available products:
\begin{equation}
    r^{\mathrm{agg}}(\state)
    =
    \bar o_t^{w}(\boldsymbol I_t)
    +
    \frac{m_t^{\mathrm{agg}}(\state)}{\beta},
    \label{eq:unified_aggregated_policy}
\end{equation}
where $m_t^{\mathrm{agg}}(\state)>1$ is the unique solution of
\begin{equation}
    \left(m_t^{\mathrm{agg}}(\state)-1\right)
    \exp\!\left(m_t^{\mathrm{agg}}(\state)\right)
    =
    \exp\!\left(
        a^{w}
        -
        \beta \bar o_t^{w}(\boldsymbol I_t)
    \right),
    \nonumber
\end{equation}
with the aggregated attractiveness
\begin{equation}
    a^{w}
    =
    \log\!\left(
        \sum_{i\in\mathcal{A}(\boldsymbol I_t)}
        \exp(a_{i})
    \right).
    \nonumber
\end{equation}

\subsection{Detailed Feature Definitions}
\label{app:feature_engineering}

Table~\ref{tab:features_summary} reports the complete set of features used by the statistical predictors. Baseline-derived features use the prices and opportunity costs of the corresponding time-dependent or time-independent policies. Unless otherwise stated, relative and competitor-based features are constructed from the \TimeIndItinerary\ policy.

\begin{table}[htbp]
\centering
\caption{Features used by the statistical predictors.}
\label{tab:features_summary}
\scriptsize
\renewcommand{\arraystretch}{1.08}
\setlength{\tabcolsep}{4pt}
\begin{threeparttable}
\begin{tabularx}{\textwidth}{
    >{\raggedright\arraybackslash}p{2.25cm}
    >{\raggedright\arraybackslash}p{5.1cm}
    X
}
\toprule
\textbf{Family}
&
\textbf{Feature}
&
\textbf{Description}
\\
\midrule

\multirow{4}{*}{State and choice}
&
Inverse price sensitivity
&
Price-scale signal derived from the choice-model parameters
\\

&
Capacity ratio
&
$I_{i,t}/\max\{T-t,1\}$
\\

&
Remaining-horizon-scaled $\bar{o}^{\mathrm{TI}}_{i,t}$
&
$\bar{o}^{\mathrm{TI}}_{i,t}(T-t)/T$
\\

&
State and choice parameters
&
Remaining inventory, time, and product and choice-model parameters
\\
\midrule

Myopic prices
&
Single- and multi-product myopic prices
&
One-period prices obtained by optimizing products separately or jointly
\\
\midrule

\multirow{2}{*}{Itinerary DP}
&
Prices of \TimeIndItinerary\ and \TimeDepItinerary
&
Prices obtained from the corresponding single-itinerary DPs
\\

&
Opportunity costs of \TimeIndItinerary\ and \TimeDepItinerary
&
$\bar{o}^{\mathrm{TI}}_{i,t}$ and
$\bar{o}^{\mathrm{TD}}_{i,t}$ obtained from the corresponding DPs
\\
\midrule

Marginal inventory value
&
Changes in $\bar{o}^{\mathrm{TI}}_{i,t}$ $(\pm1)$
&
Changes in the itinerary-level opportunity cost when inventory is increased or decreased by one unit
\\
\midrule

\multirow{2}{*}{Joint pricing}
&
Prices of \TimeIndApp\ and \TimeDepApp
&
Joint differentiated prices using itinerary-level opportunity costs directly
\\

&
Prices of \TimeIndUnified\ and \TimeDepUnified
&
Common prices based on attractiveness-weighted opportunity costs
\\
\midrule

Littlewood
&
Littlewood-based opportunity-cost proxy
&
Reference price multiplied by the probability that demand exceeds inventory
\\
\midrule

\multirow{2}{*}{Relative signals}
&
Price-to-opportunity-cost ratio and price rank
&
Relative price-to-inventory-value signal and percentile rank among available products
\\

&
Competitor prices and opportunity costs
&
Mean prices and opportunity costs of the other available products
\\

\bottomrule
\end{tabularx}

\begin{tablenotes}[flushleft]
\footnotesize
\item Opportunity cost refers to the marginal value of one unit of remaining inventory.
$\bar{o}^{\mathrm{TI}}_{i,t}$ and $\bar{o}^{\mathrm{TD}}_{i,t}$ are obtained from
\TimeIndItinerary\ and \TimeDepItinerary, respectively.
\end{tablenotes}
\end{threeparttable}
\end{table}

\paragraph{State and choice features.}
These features include the remaining inventory, time index, quality index, price sensitivity, inverse price sensitivity, and the normalized capacity ratio
\[
    \frac{I_{i,t}}{\max\{T-t,1\}}.
\]
We also include the remaining-horizon-scaled opportunity cost from \TimeIndItinerary.
\[
    \bar{o}^{\mathrm{TI}}_{i,t}\frac{T-t}{T}.
\]

\paragraph{Myopic pricing features.}
These features are computed by solving one-period pricing problems.

\paragraph{DP-based features.}
We include the prices and opportunity costs produced by the
\TimeIndItinerary\ and \TimeDepItinerary\ policies, as well as the prices
generated by \TimeIndApp, \TimeDepApp, \TimeIndUnified, and
\TimeDepUnified. Their policy formulations are given in
Appendix~\ref{app:benchmark_policies}. These features provide the statistical predictor with tractable approximations of inventory value and joint pricing.

\paragraph{Marginal opportunity-cost features.}
To capture the local sensitivity of inventory value, we use
\[
    \Delta^- \bar o^{\mathrm{TI}}_{i,t}
    =
    \bar o^{\mathrm{TI}}_{i,t}(I_{i,t}-1)
    -
    \bar o^{\mathrm{TI}}_{i,t}(I_{i,t}),
\]
and
\[
    \Delta^+ \bar o^{\mathrm{TI}}_{i,t}
    =
    \bar o^{\mathrm{TI}}_{i,t}(I_{i,t}+1)
    -
    \bar o^{\mathrm{TI}}_{i,t}(I_{i,t}).
\]

\paragraph{Relative and cross-product features.}
We further include the price-to-opportunity-cost ratio, the percentile rank of the price of an itinerary among available products, and the mean price and opportunity cost of competing products. Unless otherwise stated, these features are constructed from the \TimeIndItinerary\ policy.


\newpage

\section{Supplementary Computational Results}
\subsection{Full-Information Test Results}\label{appendix:full-info testing}
Table~\ref{tab:full_summary:fi} summarizes the detailed revenue results for all evaluated policies, where $\pm$ denotes the standard deviation. Among the baselines, \TimeDepApp achieves the best overall performance under MNL, whereas under the \MMNL setting, \TimeIndItinerary performs best. Notably, \TimeDepItinerary outperforms \TimeIndItinerary in only a few \MMNL instances, despite consistently performing better under MNL. This reversal highlights the sensitivity of DP-based baselines to model misspecification: structural refinements that are beneficial when the MNL assumptions are well aligned with the true demand environment may lose their advantage, or even become detrimental, when the underlying choice behavior deviates substantially from MNL.

\begin{table}[htbp]
\centering
\caption{Summary of full-information results. Validation-selected baselines are highlighted in gray.}
\label{tab:full_summary:fi}

\adjustbox{max width=0.9\textwidth}{
\begin{threeparttable}
\setlength{\tabcolsep}{1pt}
\renewcommand{\arraystretch}{1.25}
\footnotesize

\begin{tabular}{>{\bfseries}l c c c c c c c}
\toprule
Instance
& \TimeIndItinerary
& \TimeDepItinerary
& \TimeIndApp
& \TimeDepApp
& \TimeIndUnified
& \TimeDepUnified
& \textbf{Ours} \\
\midrule

\rowcolor{subgray}\multicolumn{8}{c}{\textbf{MNL (full information)}} \\

(6,150,400)
& 3,044.7$_{\scriptscriptstyle \pm 142.2}$
& 3,296.5$_{\scriptscriptstyle \pm 143.6}$
& 3,399.9$_{\scriptscriptstyle \pm 126.8}$
& \cellcolor{evalgray}3,402.5$_{\scriptscriptstyle \pm 128.2}$
& 3,043.5$_{\scriptscriptstyle \pm 76.4}$
& 3,043.5$_{\scriptscriptstyle \pm 76.4}$
& 3,427.0$_{\scriptscriptstyle \pm 112.8}$ \\

(6,210,1600)
& 9,583.5$_{\scriptscriptstyle \pm 289.7}$
& 11,511.1$_{\scriptscriptstyle \pm 293.6}$
& 9,416.3$_{\scriptscriptstyle \pm 335.6}$
& \cellcolor{evalgray}11,558.2$_{\scriptscriptstyle \pm 240.9}$
& 10,930.6$_{\scriptscriptstyle \pm 177.6}$
& 9,776.0$_{\scriptscriptstyle \pm 163.2}$
& 11,980.2$_{\scriptscriptstyle \pm 206.9}$ \\

(12,90,1800)
& 11,682.2$_{\scriptscriptstyle \pm 361.0}$
& 15,618.2$_{\scriptscriptstyle \pm 323.2}$
& 11,063.8$_{\scriptscriptstyle \pm 344.2}$
& \cellcolor{evalgray}16,074.3$_{\scriptscriptstyle \pm 261.9}$
& 15,061.5$_{\scriptscriptstyle \pm 263.4}$
& 12,161.6$_{\scriptscriptstyle \pm 48.4}$
& 16,918.4$_{\scriptscriptstyle \pm 184.2}$ \\

(12,180,2000)
& 17,445.9$_{\scriptscriptstyle \pm 383.6}$
& 20,641.7$_{\scriptscriptstyle \pm 393.1}$
& 18,094.5$_{\scriptscriptstyle \pm 384.7}$
& \cellcolor{evalgray}21,521.1$_{\scriptscriptstyle \pm 300.8}$
& 19,387.3$_{\scriptscriptstyle \pm 218.3}$
& 18,218.7$_{\scriptscriptstyle \pm 193.1}$
& 21,787.5$_{\scriptscriptstyle \pm 277.1}$ \\

(15,30,600)
& 4,500.7$_{\scriptscriptstyle \pm 173.1}$
& 5,726.5$_{\scriptscriptstyle \pm 188.3}$
& 4,639.8$_{\scriptscriptstyle \pm 170.4}$
& \cellcolor{evalgray}6,048.0$_{\scriptscriptstyle \pm 133.6}$
& 5,593.4$_{\scriptscriptstyle \pm 108.8}$
& 5,079.4$_{\scriptscriptstyle \pm 70.7}$
& 6,296.1$_{\scriptscriptstyle \pm 163.8}$ \\

(15,120,1400)
& 13,522.3$_{\scriptscriptstyle \pm 354.8}$
& 16,997.7$_{\scriptscriptstyle \pm 364.1}$
& 14,294.8$_{\scriptscriptstyle \pm 355.3}$
& \cellcolor{evalgray}17,715.3$_{\scriptscriptstyle \pm 288.1}$
& 15,775.2$_{\scriptscriptstyle \pm 201.7}$
& 14,202.8$_{\scriptscriptstyle \pm 157.7}$
& 18,246.0$_{\scriptscriptstyle \pm 268.4}$ \\

(21,30,800)
& 5,944.5$_{\scriptscriptstyle \pm 229.7}$
& 7,930.3$_{\scriptscriptstyle \pm 233.3}$
& 6,319.2$_{\scriptscriptstyle \pm 220.0}$
& \cellcolor{evalgray}8,278.2$_{\scriptscriptstyle \pm 176.4}$
& 8,026.8$_{\scriptscriptstyle \pm 115.9}$
& 6,966.2$_{\scriptscriptstyle \pm 83.1}$
& 8,791.4$_{\scriptscriptstyle \pm 156.6}$ \\

(21,30,1400)
& 9,534.6$_{\scriptscriptstyle \pm 283.8}$
& \cellcolor{evalgray}12,610.3$_{\scriptscriptstyle \pm 289.8}$
& 8,441.9$_{\scriptscriptstyle \pm 354.0}$
& 12,428.2$_{\scriptscriptstyle \pm 250.0}$
& 11,391.6$_{\scriptscriptstyle \pm 290.1}$
& 9,218.1$_{\scriptscriptstyle \pm 71.9}$
& 12,654.2$_{\scriptscriptstyle \pm 281.8}$ \\

(24,30,600)
& 5,097.0$_{\scriptscriptstyle \pm 185.9}$
& 6,863.8$_{\scriptscriptstyle \pm 201.4}$
& 6,077.5$_{\scriptscriptstyle \pm 170.3}$
& \cellcolor{evalgray}7,243.8$_{\scriptscriptstyle \pm 139.9}$
& 6,612.3$_{\scriptscriptstyle \pm 88.7}$
& 6,048.6$_{\scriptscriptstyle \pm 76.4}$
& 7,529.5$_{\scriptscriptstyle \pm 151.9}$ \\

(30,120,1400)
& 14,232.4$_{\scriptscriptstyle \pm 325.8}$
& 17,177.0$_{\scriptscriptstyle \pm 367.7}$
& 17,837.5$_{\scriptscriptstyle \pm 270.1}$
& \cellcolor{evalgray}18,444.2$_{\scriptscriptstyle \pm 283.6}$
& 14,841.7$_{\scriptscriptstyle \pm 95.6}$
& 14,841.5$_{\scriptscriptstyle \pm 95.5}$
& 18,525.0$_{\scriptscriptstyle \pm 274.8}$ \\

\midrule
\rowcolor{subgray}\multicolumn{8}{c}{\textbf{\MMNL (full information)}} \\

(6,150,400)
& 2,176.5$_{\scriptscriptstyle \pm 191.6}$
& \cellcolor{evalgray}2,238.3$_{\scriptscriptstyle \pm 189.3}$
& 1,574.5$_{\scriptscriptstyle \pm 130.2}$
& 1,574.5$_{\scriptscriptstyle \pm 130.2}$
& 1,879.9$_{\scriptscriptstyle \pm 103.9}$
& 1,879.9$_{\scriptscriptstyle \pm 103.9}$
& 2,998.8$_{\scriptscriptstyle \pm 124.1}$ \\

(6,210,1600)
& \cellcolor{evalgray}6,582.0$_{\scriptscriptstyle \pm 317.5}$
& 6,509.8$_{\scriptscriptstyle \pm 288.4}$
& 6,481.6$_{\scriptscriptstyle \pm 328.4}$
& 6,233.5$_{\scriptscriptstyle \pm 226.0}$
& 5,772.6$_{\scriptscriptstyle \pm 207.4}$
& 6,544.9$_{\scriptscriptstyle \pm 189.6}$
& 7,196.4$_{\scriptscriptstyle \pm 122.9}$ \\

(12,90,1800)
& \cellcolor{evalgray}10,376.9$_{\scriptscriptstyle \pm 467.7}$
& 9,786.7$_{\scriptscriptstyle \pm 399.5}$
& 9,767.3$_{\scriptscriptstyle \pm 503.9}$
& 8,867.8$_{\scriptscriptstyle \pm 414.6}$
& 4,932.7$_{\scriptscriptstyle \pm 345.2}$
& 8,557.6$_{\scriptscriptstyle \pm 258.8}$
& 11,392.8$_{\scriptscriptstyle \pm 265.1}$ \\

(12,180,2000)
& \cellcolor{evalgray}15,650.9$_{\scriptscriptstyle \pm 604.1}$
& 15,428.8$_{\scriptscriptstyle \pm 554.9}$
& 11,805.1$_{\scriptscriptstyle \pm 604.5}$
& 10,792.8$_{\scriptscriptstyle \pm 505.1}$
& 6,410.3$_{\scriptscriptstyle \pm 420.3}$
& 10,132.6$_{\scriptscriptstyle \pm 297.7}$
& 18,214.5$_{\scriptscriptstyle \pm 534.4}$ \\

(15,30,600)
& \cellcolor{evalgray}3,305.2$_{\scriptscriptstyle \pm 265.6}$
& 2,937.6$_{\scriptscriptstyle \pm 282.3}$
& 2,791.1$_{\scriptscriptstyle \pm 283.4}$
& 3,077.8$_{\scriptscriptstyle \pm 201.9}$
& 2,470.5$_{\scriptscriptstyle \pm 193.3}$
& 2,931.6$_{\scriptscriptstyle \pm 132.6}$
& 3,922.6$_{\scriptscriptstyle \pm 99.1}$ \\

(15,120,1400)
& 9,916.7$_{\scriptscriptstyle \pm 513.3}$
& \cellcolor{evalgray}10,231.3$_{\scriptscriptstyle \pm 503.3}$
& 6,585.1$_{\scriptscriptstyle \pm 456.1}$
& 8,582.0$_{\scriptscriptstyle \pm 354.4}$
& 5,607.4$_{\scriptscriptstyle \pm 318.3}$
& 9,730.5$_{\scriptscriptstyle \pm 238.5}$
& 12,162.9$_{\scriptscriptstyle \pm 197.7}$ \\

(21,30,800)
& 3,130.0$_{\scriptscriptstyle \pm 275.5}$
& 2,796.7$_{\scriptscriptstyle \pm 260.0}$
& 3,186.2$_{\scriptscriptstyle \pm 261.3}$
& \cellcolor{evalgray}3,584.5$_{\scriptscriptstyle \pm 187.8}$
& 2,301.5$_{\scriptscriptstyle \pm 209.7}$
& 3,199.0$_{\scriptscriptstyle \pm 204.6}$
& 5,234.3$_{\scriptscriptstyle \pm 129.4}$ \\

(21,30,1400)
& 5,643.0$_{\scriptscriptstyle \pm 443.5}$
& 5,251.0$_{\scriptscriptstyle \pm 416.1}$
& 5,956.6$_{\scriptscriptstyle \pm 453.5}$
& \cellcolor{evalgray}6,596.4$_{\scriptscriptstyle \pm 307.8}$
& 4,702.0$_{\scriptscriptstyle \pm 385.0}$
& 5,333.0$_{\scriptscriptstyle \pm 267.7}$
& 8,464.1$_{\scriptscriptstyle \pm 161.9}$ \\

(24,30,600)
& 2,586.9$_{\scriptscriptstyle \pm 259.6}$
& 2,292.5$_{\scriptscriptstyle \pm 275.4}$
& 2,790.7$_{\scriptscriptstyle \pm 250.8}$
& \cellcolor{evalgray}3,047.0$_{\scriptscriptstyle \pm 197.0}$
& 2,174.9$_{\scriptscriptstyle \pm 205.8}$
& 2,453.1$_{\scriptscriptstyle \pm 203.4}$
& 4,546.9$_{\scriptscriptstyle \pm 103.5}$ \\

(30,120,1400)
& \cellcolor{evalgray}8,423.2$_{\scriptscriptstyle \pm 425.4}$
& 6,406.6$_{\scriptscriptstyle \pm 417.4}$
& 7,366.2$_{\scriptscriptstyle \pm 381.0}$
& 7,234.8$_{\scriptscriptstyle \pm 313.1}$
& 6,043.5$_{\scriptscriptstyle \pm 330.4}$
& 6,043.5$_{\scriptscriptstyle \pm 330.4}$
& 11,931.7$_{\scriptscriptstyle \pm 166.9}$ \\

\bottomrule
\end{tabular}
\end{threeparttable}
}

\end{table}

\subsection{Data-Driven Overall Results}

Table~\ref{tab:summary:dd} complements the main results in Table~\ref{tab: data-driven overall results} by reporting the test revenues of all six baselines, together with the validation-selected and test-best DFL configurations. In the realistic data-driven setting, the true choice model $\bbP$ is unknown, so baseline configurations are selected based on their validation performance under the fitted MNL model $\bbQ$. Such model-based selection may not identify the baseline that performs best in the ground-truth environment $\bbP$. We therefore also report, as an ex post diagnostic, the best baseline according to its actual test performance. This comparison does not alter the selection protocol used in our main results, but provides a stronger benchmark for evaluation. Even against this more favorable benchmark, our approach achieves the best overall performance, with only two \MMNL instances showing slightly lower revenue, both by less than $1\%$. These results further demonstrate the robustness of the proposed architecture.

\begin{table}[htbp]
\centering
\caption{Data-driven test revenues for the six baselines and for validation-selected and test-best DFL configurations. Gray highlights the validation-selected baseline.}
\label{tab:summary:dd}
\setlength{\tabcolsep}{1.5pt}
\renewcommand{\arraystretch}{1.15}

\resizebox{\textwidth}{!}{%
\begin{tabular}{lcccccccc}
\toprule
\multirow{2}{*}{\textbf{Instance}}
& \multirow{2}{*}{\textbf{\TimeIndItinerary }}
& \multirow{2}{*}{\textbf{\TimeDepItinerary}}
& \multirow{2}{*}{\textbf{\TimeIndApp}}
& \multirow{2}{*}{\textbf{\TimeDepApp}}
& \multirow{2}{*}{\textbf{\TimeIndUnified}}
& \multirow{2}{*}{\textbf{\TimeDepUnified}}
& \multicolumn{2}{c}{\textbf{Ours}} \\
\cmidrule(lr){8-9}
&
&
&
&
&
&
&
\textbf{validation-selected}
& \textbf{test-best}
\\
\midrule

\rowcolor{subgray}\multicolumn{9}{c}{\textbf{\large MNL (data-driven)}} \\

$(6,150,400)$
& $2{,}927.0_{\pm145.1}$
& $2{,}941.4_{\pm124.0}$
& $3{,}351.2_{\pm119.9}$
& \cellcolor{evalgray}$3{,}351.2_{\pm119.6}$
& $2{,}887.8_{\pm67.1}$
& $2{,}887.8_{\pm67.1}$
& $3{,}380.1_{\pm112.2}$
& $3{,}380.1_{\pm112.2}$ \\

$(6,210,1600)$
& $8{,}942.6_{\pm310.1}$
& $10{,}202.8_{\pm262.3}$
& $9{,}223.1_{\pm302.7}$
& \cellcolor{evalgray}$11{,}338.1_{\pm243.5}$
& $10{,}794.9_{\pm174.4}$
& $9{,}423.6_{\pm143.4}$
& $11{,}814.4_{\pm220.9}$
& $11{,}814.4_{\pm220.9}$ \\

$(12,90,1800)$
& $10{,}860.9_{\pm328.5}$
& $15{,}782.8_{\pm220.6}$
& $11{,}532.5_{\pm325.0}$
& \cellcolor{evalgray}$16{,}460.5_{\pm218.9}$
& $13{,}452.3_{\pm154.7}$
& $10{,}356.2_{\pm32.6}$
& $16{,}596.4_{\pm204.0}$
& $16{,}596.4_{\pm204.0}$ \\

$(12,180,2000)$
& $16{,}084.1_{\pm425.0}$
& $17{,}741.2_{\pm323.3}$
& $17{,}991.0_{\pm396.4}$
& \cellcolor{evalgray}$21{,}321.7_{\pm319.8}$
& $17{,}988.4_{\pm186.6}$
& $17{,}177.4_{\pm173.3}$
& $21{,}516.7_{\pm284.5}$
& $21{,}516.7_{\pm284.5}$ \\

$(15,30,600)$
& $4{,}305.1_{\pm174.9}$
& $5{,}526.4_{\pm147.8}$
& $4{,}920.9_{\pm164.3}$
& \cellcolor{evalgray}$6{,}052.0_{\pm122.8}$
& $5{,}516.3_{\pm88.9}$
& $4{,}332.6_{\pm26.4}$
& $6{,}118.6_{\pm138.1}$
& $6{,}118.6_{\pm138.1}$ \\

$(15,120,1400)$
& $12{,}394.9_{\pm357.5}$
& $14{,}094.7_{\pm268.8}$
& $14{,}507.4_{\pm332.3}$
& \cellcolor{evalgray}$17{,}728.0_{\pm285.7}$
& $14{,}501.6_{\pm176.1}$
& $13{,}374.9_{\pm150.1}$
& $17{,}975.7_{\pm262.1}$
& $17{,}975.7_{\pm262.1}$ \\

$(21,30,800)$
& $5{,}988.2_{\pm219.9}$
& $7{,}415.1_{\pm132.6}$
& $6{,}969.7_{\pm206.0}$
& \cellcolor{evalgray}$8{,}543.7_{\pm149.0}$
& $7{,}096.7_{\pm92.1}$
& $5{,}930.3_{\pm73.3}$
& $8{,}609.7_{\pm130.1}$
& $8{,}609.7_{\pm130.1}$ \\

$(21,30,1400)$
& $12{,}293.7_{\pm137.1}$
& $8{,}625.8_{\pm91.9}$
& $12{,}242.3_{\pm110.7}$
& \cellcolor{evalgray}$9{,}125.8_{\pm65.4}$
& $9{,}926.3_{\pm37.6}$
& $5{,}783.6_{\pm10.6}$
& $11{,}234.7_{\pm106.7}$
& $12{,}444.1_{\pm151.7}$ \\

$(24,30,600)$
& $4{,}827.2_{\pm181.6}$
& $4{,}669.8_{\pm148.5}$
& $6{,}279.5_{\pm182.8}$
& \cellcolor{evalgray}$7{,}317.8_{\pm137.4}$
& $6{,}061.4_{\pm80.5}$
& $5{,}689.8_{\pm66.5}$
& $7{,}376.0_{\pm145.0}$
& $7{,}376.0_{\pm145.0}$ \\

$(30,120,1400)$
& $12{,}973.6_{\pm327.5}$
& $9{,}797.9_{\pm187.6}$
& $17{,}681.1_{\pm267.6}$
& \cellcolor{evalgray}$18{,}128.1_{\pm291.0}$
& $14{,}056.1_{\pm81.7}$
& $14{,}056.0_{\pm81.7}$
& $18{,}193.8_{\pm244.2}$
& $18{,}193.8_{\pm244.2}$ \\

\midrule
\rowcolor{subgray}\multicolumn{9}{c}{\textbf{\large MixMNL (data-driven)}} \\

$(6,150,400)$
& $1{,}743.4_{\pm237.7}$
& $2{,}148.7_{\pm198.3}$
& $2{,}533.3_{\pm51.6}$
& \cellcolor{evalgray}$2{,}682.7_{\pm45.3}$
& $2{,}252.9_{\pm18.0}$
& $2{,}252.4_{\pm17.9}$
& $2{,}972.8_{\pm82.4}$
& $2{,}972.8_{\pm82.4}$ \\

$(6,210,1600)$
& $6{,}329.4_{\pm366.1}$
& $6{,}619.4_{\pm294.4}$
& $6{,}741.9_{\pm249.4}$
& \cellcolor{evalgray}$7{,}182.6_{\pm152.4}$
& $6{,}970.8_{\pm136.2}$
& $6{,}340.3_{\pm83.9}$
& $7{,}365.5_{\pm146.9}$
& $7{,}365.5_{\pm146.9}$ \\

$(12,90,1800)$
& $10{,}909.1_{\pm510.9}$
& $11{,}418.7_{\pm335.8}$
& $11{,}708.1_{\pm481.1}$
& \cellcolor{evalgray}$11{,}198.6_{\pm223.1}$
& $9{,}141.9_{\pm151.6}$
& $7{,}482.3_{\pm47.6}$
& $11{,}569.5_{\pm298.2}$
& $11{,}592.7_{\pm272.0}$ \\

$(12,180,2000)$
& $14{,}945.6_{\pm592.5}$
& $13{,}594.7_{\pm456.9}$
& $13{,}990.2_{\pm465.2}$
& \cellcolor{evalgray}$12{,}935.1_{\pm287.2}$
& $11{,}618.5_{\pm110.3}$
& $10{,}873.6_{\pm81.7}$
& $16{,}361.7_{\pm477.6}$
& $17{,}165.4_{\pm430.7}$ \\

$(15,30,600)$
& $3{,}261.8_{\pm276.3}$
& $3{,}794.7_{\pm145.5}$
& $3{,}556.6_{\pm223.9}$
& $4{,}013.0_{\pm101.1}$
& $3{,}222.9_{\pm116.4}$
& \cellcolor{evalgray}$2{,}993.8_{\pm42.0}$
& $3{,}996.6_{\pm152.5}$
& $3{,}996.6_{\pm152.5}$ \\

$(15,120,1400)$
& $10{,}375.0_{\pm476.1}$
& $10{,}442.9_{\pm317.9}$
& $10{,}638.9_{\pm292.4}$
& \cellcolor{evalgray}$11{,}010.5_{\pm200.0}$
& $8{,}985.1_{\pm75.7}$
& $8{,}791.8_{\pm59.0}$
& $12{,}395.0_{\pm281.8}$
& $12{,}395.0_{\pm281.8}$ \\

$(21,30,800)$
& $3{,}452.7_{\pm265.9}$
& $4{,}535.7_{\pm145.1}$
& $4{,}082.0_{\pm184.9}$
& $5{,}050.4_{\pm97.1}$
& $4{,}944.1_{\pm87.4}$
& \cellcolor{evalgray}$4{,}179.7_{\pm40.8}$
& $5{,}360.7_{\pm119.4}$
& $5{,}372.9_{\pm100.6}$ \\

$(21,30,1400)$
& $5{,}691.7_{\pm402.0}$
& $7{,}692.0_{\pm162.9}$
& $6{,}493.0_{\pm350.9}$
& \cellcolor{evalgray}$7{,}821.8_{\pm119.3}$
& $7{,}304.0_{\pm180.2}$
& $4{,}573.6_{\pm27.1}$
& $8{,}704.1_{\pm122.8}$
& $8{,}704.1_{\pm122.8}$ \\

$(24,30,600)$
& $3{,}060.9_{\pm241.7}$
& $2{,}937.6_{\pm110.1}$
& $3{,}773.4_{\pm137.0}$
& $3{,}901.5_{\pm79.8}$
& \cellcolor{evalgray}$3{,}953.8_{\pm70.0}$
& $3{,}747.4_{\pm38.5}$
& $4{,}261.8_{\pm113.5}$
& $4{,}261.8_{\pm113.5}$ \\

$(30,120,1400)$
& $9{,}599.5_{\pm409.6}$
& $6{,}139.7_{\pm184.0}$
& $9{,}968.3_{\pm165.6}$
& \cellcolor{evalgray}$10{,}238.1_{\pm117.9}$
& $9{,}335.7_{\pm42.2}$
& $9{,}334.6_{\pm42.2}$
& $10{,}931.8_{\pm199.2}$
& $10{,}931.8_{\pm199.2}$ \\

\bottomrule
\end{tabular}%
}
\end{table}

\end{APPENDICES}

\end{document}